\documentclass[11pt, reqno]{amsart}
\usepackage{amssymb}
\usepackage{amsmath}
\usepackage{amsthm}
\usepackage{mathabx}
\usepackage{braket}
\usepackage{pdfpages}
\usepackage{graphicx}
\usepackage{caption}
\usepackage{subcaption}
\usepackage{mathrsfs} 
\usepackage{tikz-cd} 
\usepackage{bm}
\usepackage{enumitem}
\usepackage{mathtools}
\usepackage{pgfplots}
\usepackage{import}
\usepackage{xifthen}
\usepackage{transparent}
\usepackage{verbatim}
\usepackage{mathtools}
\usepackage{xargs}                      % Use more than one optional parameter in a new commands
\usepackage[colorinlistoftodos,prependcaption,textsize=tiny]{todonotes}
\newcommandx{\unsure}[2][1=]{\todo[linecolor=red,backgroundcolor=red!25,bordercolor=red,#1]{#2}}
\newcommandx{\change}[2][1=]{\todo[linecolor=blue,backgroundcolor=blue!25,bordercolor=blue,#1]{#2}}
\newcommandx{\info}[2][1=]{\todo[linecolor=OliveGreen,backgroundcolor=OliveGreen!25,bordercolor=OliveGreen,#1]{#2}}
\newcommandx{\improvement}[2][1=]{\todo[linecolor=Plum,backgroundcolor=Plum!25,bordercolor=Plum,#1]{#2}}

\pgfplotsset{compat=1.14}

\usepackage{tikz}
\usetikzlibrary{automata, arrows.meta, positioning}

\usepackage[utf8]{inputenc}
\usepackage[english]{babel}

\usepackage{hyperref}
\usepackage{cleveref}

\makeatletter
\DeclareFontFamily{U}{tipa}{}
\DeclareFontShape{U}{tipa}{m}{n}{<->tipa10}{}
\newcommand{\arc@char}{{\usefont{U}{tipa}{m}{n}\symbol{62}}}%

\newcommand{\arc}[1]{\mathpalette\arc@arc{#1}}
\newcommand{\arc@arc}[2]{%
  \sbox0{$\m@th#1#2$}%
  \vbox{
    \hbox{\resizebox{\wd0}{\height}{\arc@char}}
    \nointerlineskip
    \box0
  }%
}
\makeatother

\usepackage{lmodern,babel,adjustbox,booktabs,multirow}

\crefformat{footnote}{#2\footnotemark[#1]#3}

\numberwithin{equation}{section}
\theoremstyle{plain}
\newtheorem{theorem}{Theorem}[section]

\newtheorem{prop}[theorem]{Proposition}
\newtheorem{conj}[theorem]{Conjecture}
\newtheorem{lem}[theorem]{Lemma}
\newtheorem{cor}[theorem]{Corollary}

\newtheorem*{question*}{Question}

\theoremstyle{plain}

\theoremstyle{definition}
\newtheorem{defn}[theorem]{Definition}

\newtheorem{example}[theorem]{Example}
\newtheorem*{remark}{Remark}

\newcommand{\R}{\mathbb{R}}
\newcommand{\C}{\mathbb{C}}

\newcommand{\Hyp}{\mathbb{H}}

\newcommand{\N}{\mathbb{N}}
\newcommand{\D}{\mathbb{D}}

\DeclareMathOperator{\EL}{EL}
\DeclareMathOperator{\EW}{EW}
\DeclareMathOperator{\ELL}{EL}
\DeclareMathOperator{\EWW}{EW}
\DeclareMathOperator{\CW}{CW}
\DeclareMathOperator{\QC}{QC}
\DeclareMathOperator{\Teich}{Teich}

\DeclareMathOperator{\PSL}{PSL}

\DeclareMathOperator{\Isom}{Isom}

\DeclareMathOperator{\Aut}{Aut}
\DeclareMathOperator{\Int}{Int}

\DeclareMathOperator{\diam}{diam}

\DeclareMathOperator{\area}{area}

\DeclareMathOperator{\core}{core}

\numberwithin{figure}{section}

\begin{document}
\title[Uniform bounded diameter conjecture]{Disk patterns, quasi-duality and the uniform bounded diameter conjecture}
\begin{author}[Y.~Luo]{Yusheng Luo}
\address{Department of Mathematics, Cornell University, 212 Garden Ave, Ithaca, NY 14853, USA}
\email{yusheng.s.luo@gmail.com}
\end{author}
\begin{author}[Y.~Zhang]{Yongquan Zhang}
\address{Institute for Mathematical Sciences, Stony Brook University, 100 Nicolls Rd, Stony Brook, NY 11794-3660, USA}
\email{yqzhangmath@gmail.com}
\end{author}
\thanks{The first-named author is partially supported by NSF Grant DMS-2349929}

\date{November 5, 2025}

\begin{abstract}
    We show that the diameter of the image of the skinning map on the deformation space of an acylindrical reflection group is bounded by a constant depending only on the topological complexity of the components of its boundary, answering a conjecture of Minsky in the reflection group setting. This result can be interpreted as a uniform rigidity theorem for disk patterns. Our method also establishes a connection between the diameter of the skinning image and certain discrete extremal width on the Coxeter graph of the reflection group.
\end{abstract}

\maketitle

\setcounter{tocdepth}{1}
\tableofcontents

\section{Introduction}
Let $\widetilde G \subseteq \PSL_2(\C)$ be a geometrically finite Kleinian group with connected limit set, and $M \coloneqq \widetilde G \backslash (\Hyp^3 \cup \Omega(\widetilde G))$ be the corresponding Kleinian 3-manifold, where $\Omega(\widetilde G)$ is the domain of discontinuity of $\widetilde G$.
The quasiconformal deformation space of $\widetilde G$ can be naturally identified with the Teichm\"uller space $\Teich(\partial M)$.
If $\widetilde G$ contains no accidental parabolics, the covering of $\Int(M)$ corresponding to $\partial M$ is a (potentially disconnected) quasifuchsian manifold whose conformal boundary is the union of $X \in \Teich(\partial M)$ and its \textit{skinning surface} $\sigma_M(X) \in \Teich(\overline{\partial M})$.
This defines a map between Teichm\"uller spaces
$$
\sigma_M: \Teich(\partial M) \longrightarrow \Teich(\overline{\partial M}),
$$
called the \textit{skinning map}.

Suppose that $M$ is \textit{acylindrical}, or equivalently, the limit set $\Lambda(\widetilde G)$ is homeomorphic to a \textit{round Schottky set}, i.e., the complement of infinitely many disjoint round open disks in $\widehat{\C}$ (see Figure \ref{fig:intro_exampleC} and \ref{fig:no_right_angled_2_connection}).
Thurston's Bounded Image Theorem (see \S\ref{subsec: acy} and \cite{Thu86}), which is a crucial step in Thurston's hyperbolization theorem, states that the image of $\sigma_M$ has compact closure.
Thus,
$$
\diam(\sigma_{M}( \Teich(\partial M))) < \infty
$$
in the Teichm\"uller metric. Here, the Teichm\"uller metric on $\Teich(\partial M)$ is defined as the supreme of the Teichm\"uller metrics on its components.

It has been suggested that an effective version of the Bounded Image Theorem may yield more explicit estimates on the hyperbolic structure, leading to an effective version of the hyperbolization theorem (e.g.~\cite{Ker05}). A quantitative bound on the diameter of the skinning image may be the first step towards this goal. It is conjectured by Minsky that
\begin{conj}\label{conj:main}
    Suppose that $M$ is acylindrical. Then there exists a constant $K$ depending only on the topological type of $\partial M$ so that 
    $$
    \diam(\sigma_{M}( \Teich(\partial M))) \le K
    $$
    in the Teichm\"uller metric on $\Teich(\partial M)$.
\end{conj}
There have been various recent results supporting this conjecture (see \cite{Ken10, KM14, BKM21}). This paper studies the case of reflection groups.

Let $G$ be a discrete group generated by reflections along finitely many circles in $\hat{\mathbb{C}}$ with connected limit set, and let $\QC(G)$ be the quasiconformal deformation space of $G$.
Let $\widetilde G \lhd G$ be the index 2 subgroup consisting of orientation preserving elements. Then $\widetilde G$ is a geometrically finite Kleinian group. Let $M\coloneqq \widetilde G \backslash (\Hyp^3 \cup \Omega(\widetilde G))$ be the corresponding Kleinian orbifold. Then
$$
\QC(G) \cong \Teich(\partial M/r) \cong \Teich^r(\partial M) \subseteq \Teich(\partial M),
$$
where $r: \partial M \longrightarrow \partial M$ is a orientation reversing involution, and $\Teich^r(\partial M)$ consists of conformal structures on $\partial M$ with an anti-conformal involution isotopic to $r$.
We remark that $\partial M/r$ is a finite union of hyperbolic polygons, potentially containing some ideal vertices.
The topological complexity $\mathscr{C}_{top}(G)$ of $G$ is defined as the maximal number of edges in a component of $\partial M/r$.
The skinning map $\sigma_{M}$ restricts to a map
$$
\sigma_{M}: \Teich(\partial M/r) \longrightarrow \Teich(\overline{\partial M/r}).
$$
Our main theorem confirms Conjecture \ref{conj:main} for reflection groups.
\begin{theorem}\label{thm:main}
    Let $G$ be an acylindrical reflection group with topological complexity $\mathscr{C}_{top}(G)$.
    Then there exists a constant $K$ depending only on $\mathscr{C}_{top}(G)$ so that
    $$
    \diam(\sigma_M(\Teich(\partial M/r))) \leq K
    $$
    in the Teichm\"uller metric.
\end{theorem}

We remark that the Bounded Image Theorem does not hold when $M$ is not acylindrical, as otherwise the skinning map may be used to produce a hyperbolic structure on the double $DM$ of $M$ (two copies of $M$ glued along the boundary); see \cite{McM90}. But $DM$ is toroidal, since essential cylinders in $M$ become tori in $DM$. In our setting, a more explicit and detailed discussion can be found in \cite[\S3]{LZ23}.

For a reflection group $G$, the skinning map takes a very concrete form in terms of the disk pattern associated to the generators of $G$ (see \S \ref{sec:dprg} and \S \ref{sec:cd}).
Our main theorem can be interpreted as uniform rigidity results for circle packings, or more generally disk patterns (see Theorem \ref{thm:uniformbounddiskpattern}).

The topological complexity in Theorem \ref{thm:main} is on par with the maximal absolute value of the orbifold Euler characteristic of each component of $\partial M$.
Indeed, each component $X$ of $\partial M$ satisfies $N/2-2\le |\chi(X)|\le N-2$ if $N$ is the number of sides of $X/r$.
Note that, in particular, the uniform bound $K$ does not depend on the number of components of $\partial M$. 
The Euler characteristic of the boundary of the smallest manifold cover of $M$ may have much larger absolute value, and generally depends on the cone angles of $\partial M$.
On the other hand, if all vertices of $\partial M/r$ are ideal vertices (i.e.\ in the case of \emph{kissing reflection groups}, cf. \cite{LLM22}), then $M$ itself is a manifold, and $\mathscr{C}_{top}(G)-2$ is the maximal absolute value of the Euler characteristic of each component of $\partial M$.

We remark that our main theorem complements the existing results in \cite{Ken10, KM14, BKM21} where some lower bounds on the injectivity radius or depth of collar about the convex core boundary are assumed. 
In our setting, 
\begin{itemize}
    \item the reflection group $G$ can contain parabolic elements, so $\widetilde G\backslash\Hyp^3$ can contain cusps, thus the injectivity radius can be $0$;
    \item degenerations in $\Teich(\partial M/r)$ always occur into the thin part of the Teichm\"uller space $\Teich(\partial M)$;
    \item for every $N\geq 3$, one can construct a sequence of acylindrical reflection groups with totally geodesic convex core boundary and topological complexity $N$, the depth of whose collar about the convex core boundary goes to $0$;
    \item the uniform bound does not depend on the number of components of $\partial M$, while in the results mentioned above it seems to.
\end{itemize}

We now discuss some of the ingredients in the proof of the main theorem, in particular the connections with disk patterns and extremal lengths.

\subsection{Disk pattern and reflection group}\label{sec:dprg}
Let $\mathcal{G}$ be a connected finite simple plane graph with vertex and edge sets $\mathcal{V}$ and $\mathcal{E}$. Let $\omega: \mathcal{E} \longrightarrow \{\frac{\pi}{n}\colon n \in \N_{\geq 2}\} \cup \{0\}$ be some weight function on its edge set.
We call such an edge-weighted graph $(\mathcal{G}, \omega)$ a \textit{Coxeter graph}.
A \textit{disk pattern} with combinatorics $(\mathcal{G}, \omega)$ is essentially a collection of disks $\mathcal{P} = \{D_v, v\in \mathcal{V}\}$ in the Riemann sphere $\hat{\mathbb{C}}$ whose intersection pattern and angles are described by the graph $\mathcal{G}$ and the weight function $\omega$ respectively (see Definitions \ref{defn:completion} and \ref{defn:dp} for the subtleties involving a parabolic face). 

The Koebe-Andreev-Thurston theorem (see Theorem \ref{thm:KAT}, c.f. \cite[Theorem~1.4]{RHD07}) gives a characterization on the realization problem for disk patterns, and its deformation space is identified with the product space of the Teichm\"uller space (see Theorem \ref{thm:moduli}, c.f. \cite[Theorem~1.3]{HL13} and \cite[Theorem~0.5]{HL17}).
Given a disk pattern $\mathcal{P}$ realizing the Coxeter graph $(\mathcal{G}, \omega)$, we consider the reflection group $G=G_\mathcal{P}$ generated by reflections along all disks in $\mathcal{P}$. This gives a correspondence between reflection groups and disk patterns.

In \S \ref{sec:cd}, we will give a characterization of acylindrical reflection group in terms of its Coxeter graph $(\mathcal{G}, \omega)$ (see Theorem~\ref{thm:acy}). In particular, we will show that if $G$ is acylindrical, then $\mathcal{G}$ is $3$-connected. Equivalently, this means that $\mathcal{G}$ is a polyhedral graph, i.e., it is the 1-skeleton of a convex polyhedron.
We have an identification
$$
\QC(G) \cong \Teich(\partial M/r) \cong \prod_{F} \Teich(\Pi_F),
$$
where $F$ is a hyperbolic face of $(\mathcal{G}, \omega)$ (see \S \ref{sec:cd}) and $\Pi_F$ is the corresponding hyperbolic polygon whose angles are determined by the weight function $\omega$.

Hence in our setting, the skinning map can be explicitly defined as follows.
Given a disk pattern $\mathcal{P}$ associated to the reflection group, the input of the skinning map is represented by a collection of polygons $\{\Pi_{F, \mathcal{P}}\}_F$.
For each hyperbolic face $F$, {if we configure the disk pattern under an appropriate M\"obius transformation so that the polygon $\Pi_{F,\mathcal{P}}$ is the unbounded component of $\hat{\mathbb{C}}-\bigcup_{v\in\mathcal{V}}\overline{D_v}$ (i.e.\ it contains $\infty$), then the corresponding component of the skinning image is represented by the bounded component of $\hat{\mathbb{C}}-\bigcup_{v\in\partial F}\overline{D_v}$, which is revealed by removing the disks for vertices in the complement of $F$. This bounded component is another polygon $\Pi_{F, \mathcal{P}}^{-}$ with the same number of sides.} See Figure~\ref{fig:intro_example} for an example illustrating the correspondence between reflection groups and circle patterns, as well as this explicit presentation of the skinning map.

\begin{figure}[htp]
    \centering
    \begin{subfigure}[t]{0.49\linewidth}
        \centering
        \includegraphics[width=\linewidth]{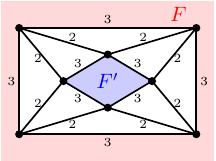}
        \caption{The Coxeter graph, where $n$ on an edge means weight $\pi/n$. Note that there are two quadrilateral hyperbolic faces.}
    \end{subfigure}\hspace{0.01\linewidth}
    \begin{subfigure}[t]{0.49\linewidth}
        \centering
        \includegraphics[width=\linewidth]{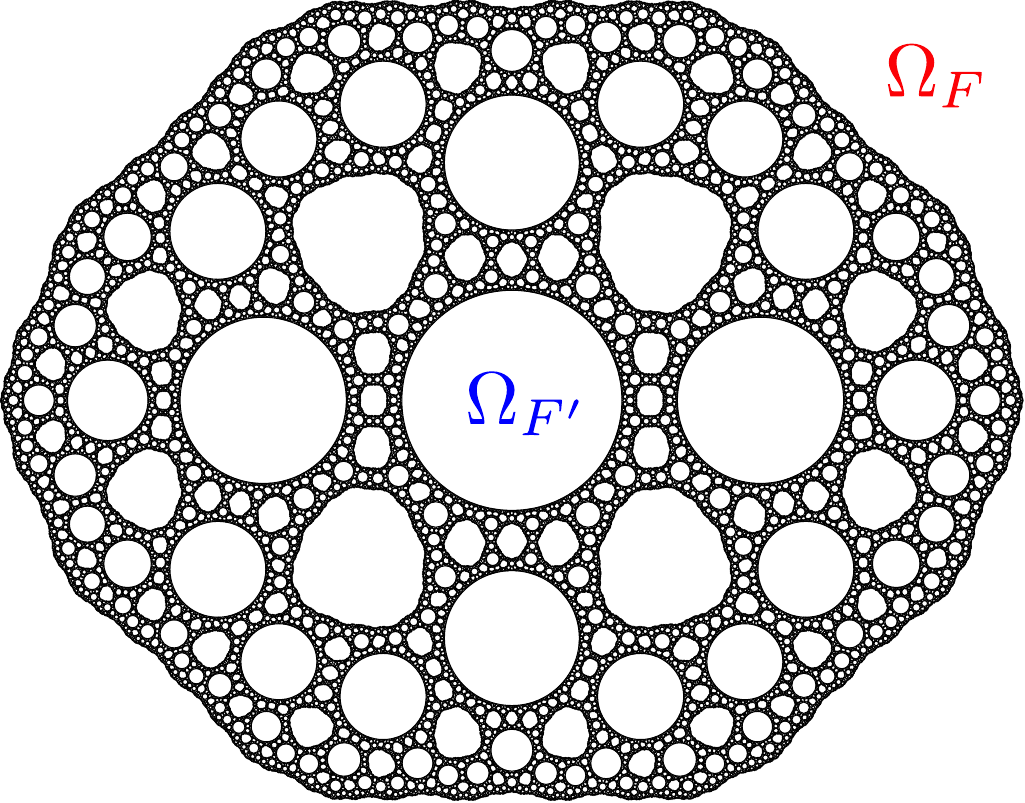}
        \caption{The limit set of a corresponding acylindrical reflection group. It is homeomorphic to a Sierpinski carpet.}
        \label{fig:intro_exampleC}
    \end{subfigure}
    \vspace{0.02\linewidth}
    
    \begin{subfigure}[t]{0.49\linewidth}
        \centering
        \includegraphics[width=\linewidth]{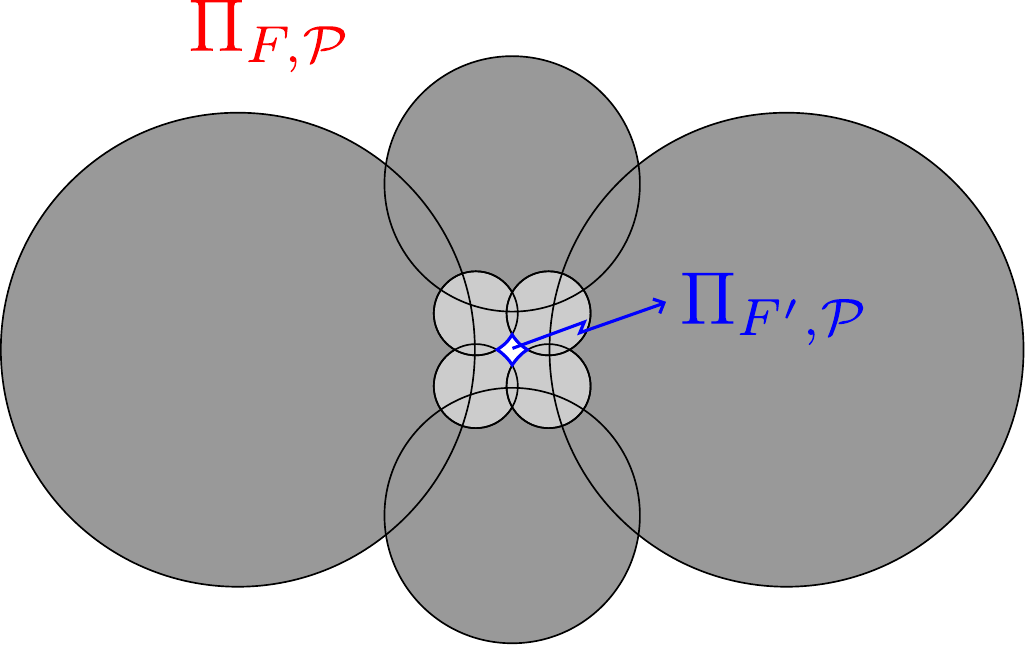}
        \caption{A disk pattern $\mathcal{P}$, normalized so that the interstice $\Pi_{F,\mathcal{P}}$ for the face $F$ contains $\infty$. The other interstice $\Pi_{F',\mathcal{P}}$ is also visible.}
    \end{subfigure}\hspace{0.01\linewidth}
    \begin{subfigure}[t]{0.49\linewidth}
        \centering
        \includegraphics[width=\linewidth]{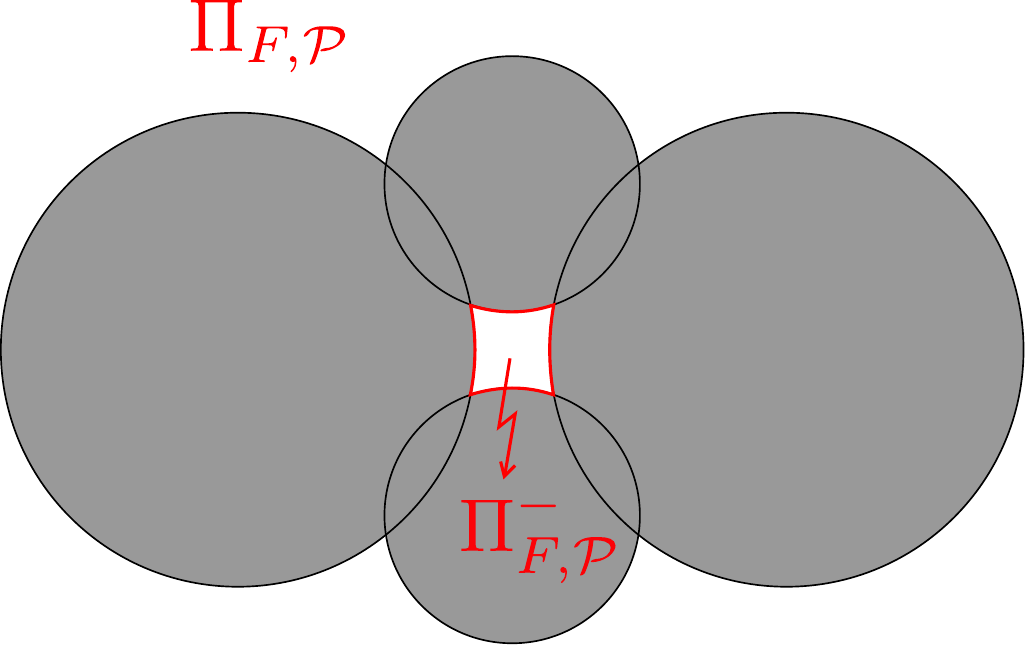}
        \caption{Removing the light gray disks reveals the component of the image of the skinning map for the face $F$, which is also a polygon $\Pi_{F,\mathcal{P}}^-$.}
    \end{subfigure}
    \caption{A reflection group, the corresponding disk pattern, and the skinning image.}
    \label{fig:intro_example}
\end{figure}

\iffalse
\begin{figure}[htp]
    \centering
    \begin{subfigure}{0.5\linewidth}
        \centering
        \includegraphics[width=0.8\linewidth]{intro_graph.pdf}
        \caption{The Coxeter graph, where $n$ on an edge means weight $\pi/n$. Note that all but one hyperbolic face is triangular.}
    \end{subfigure}
    
    \begin{subfigure}[t]{0.49\linewidth}
        \centering
        \includegraphics[width=\linewidth]{intro_patterns.pdf}
        \caption{The disk pattern. Removing the light grey disk reveals the image of the skinning map, which is the hyperbolic polygon marked in red.}
    \end{subfigure}\hspace{0.01\linewidth}
    \begin{subfigure}[t]{0.49\linewidth}
        \centering
        \includegraphics[width=\linewidth]{intro_circles.pdf}
        \caption{The limit set of the corresponding acylindrical reflection group. It is homeomorphic to a Sierpinski carpet.}
        \label{fig:intro_exampleC}
    \end{subfigure}
    \caption{An example of the correspondence between reflection groups and disk patterns.}
    \label{fig:intro_example}
\end{figure}
\fi

\subsection{Discrete extremal lengths / widths}
A key tool in this paper is extremal length, particularly in the context of disk patterns.
Let $(\mathcal{G}, \omega)$ be a Coxeter graph, and $F$ be a hyperbolic face of $(\mathcal{G}, \omega)$.
Let $a, b$ be a pair of non-adjacent vertices on $\partial F$.
Let $\Gamma_{a, b}$ be the family of paths $\gamma$ in $\mathcal{G}$ with $\Int\gamma \subseteq \mathcal{G} - \partial F$ connecting $a, b$.
Similarly, let $\Gamma_{a, b}^*$ be the family of paths $\gamma$ in $\mathcal{G}$ with $\Int\gamma \subseteq \mathcal{G} - \partial F$ and $\partial \gamma \subseteq \partial F$ separating $a, b$.
We consider the vertex extremal length, denoted by $\EL_\mathcal{G}(\Gamma_{a,b}, \partial F)$ and $\EL_\mathcal{G}(\Gamma_{a,b}^*, \partial F)$ for these families of paths in $\mathcal{G}$ (see \S \ref{sec:vel} for more details). We denote the vertex extremal width (or vertex modulus), i.e. the reciprocal of extremal length, by  $\EW_\mathcal{G}(\Gamma_{a,b}, \partial F)$ and $\EW_\mathcal{G}(\Gamma_{a,b}^*, \partial F)$

This discrete version of extremal length was first introduced by Cannon \cite{Can94} and in various other forms by Duffin \cite{Duf62} and Schramm \cite{Sch95}.
If $\mathcal{G}$ induces a triangulation of the complement of the face $F$ in $S^2$, then it follows from a classical result of Schramm (see \cite{Sch93}) that we have duality of extremal lengths / widths, i.e., 
$$
\EW_\mathcal{G}(\Gamma_{a,b}, \partial F) \EW_\mathcal{G}(\Gamma_{a,b}^*, \partial F) = 1.
$$
Duality fails in general (see Example \ref{ex:del}).
Instead, we prove a uniform quasi-duality in terms of the topological complexity (see Theorem \ref{thm:combunifbound})
\begin{align}
\label{eqn:qd}
\frac{1}{(4\mathscr{C}_{top}(G)+1)^2} \leq \EW_\mathcal{G}(\Gamma_{a,b}, \partial F) \EW_\mathcal{G}(\Gamma_{a,b}^*, \partial F) \leq 1.
\end{align}
We remark that the uniform lower bound in Equation \eqref{eqn:qd} is the key in our argument to obtain a uniform upper bound on the skinning diameter.
It is a discrete analogue of the \textit{reciprocal} condition in \cite{Raj17, NR22} (see \S \ref{subsec:qd} for more discussions).

\subsection*{Comparing with conformal extremal widths}
The vertex extremal length / width serves as a discrete analogue and gives good approximation of the classical conformal extremal length / width in the following sense.

Let $(S_F)_F \in \sigma_M(\Teich(\partial M/r))$. Let $\EW(\Gamma_{a,b,S_F})$ (or $\EW(\Gamma_{a,b,S_F}^*)$) be the conformal extremal width of families of paths $\Gamma_{a,b, S_F}$ connecting (or families of paths $\Gamma_{a,b, S_F}^*$ separating) the edges associated to $a, b$ of the hyperbolic polygon $S_F$ (see \S \ref{subsec:eldp} for more details). 
In Theorem \ref{thm:gvsc}, we prove that if $G$ is acylindrical, and if the extremal width $\EWW(\Gamma_{a,b,S_F})$ is bigger than some threshold depending only on $\mathscr{C}_{top}(G)$, then
\begin{align}
\label{eqn:dvsc}
    \EW_\mathcal{G}(\Gamma_{a,b}, \partial F) \lesssim \EWW(\Gamma_{a,b,S_F})
    \lesssim \frac{1}{\EW_\mathcal{G}(\Gamma_{a,b}^*, \partial F)}.
\end{align}

\subsection{Proof sketch for Theorem \ref{thm:main}}
Let $(S_F)_F, (S'_F)_F \in \sigma_M(\Teich(\partial M/r))$.
Let $a, b$ be a pair of non-adjacent vertices on $\partial F$.
By Equation \eqref{eqn:dvsc}, if the extremal widths for the polygons $(S_F)_F, (S'_F)_F$ between $a, b$ are big, then
$$
\max\left\{\frac{\EW(\Gamma_{a,b,S_F})}{\EW(\Gamma_{a,b,S_F'})}, \frac{\EW(\Gamma_{a,b,S_F'})}{\EW(\Gamma_{a,b,S_F})}\right\} \lesssim \frac{1}{\EW_\mathcal{G}(\Gamma_{a,b}, \partial F) \EW_\mathcal{G}(\Gamma_{a,b}^*, \partial F)}.
$$
Thus, by Equation \eqref{eqn:qd}, there exists some constant $K$ depending only on $\mathscr{C}_{top}(G)$ so that
\begin{align}
\label{eqn:rate}
\max\left\{\frac{\EW(\Gamma_{a,b,S_F})}{\EW(\Gamma_{a,b,S_F'})}, \frac{\EW(\Gamma_{a,b,S_F'})}{\EW(\Gamma_{a,b,S_F})}\right\} \leq K.
\end{align}
This essentially implies that the Teichm\"uller distance between the two hyperbolic polygons $(S_F)_F, (S'_F)_F \in \sigma_M(\Teich(\partial M/r))$ is bounded above by some constant depending only on the topological complexity $\mathscr{C}_{top}(G)$, giving the desired uniform upper bound on the diameter of the skinning map $\sigma_M(\Teich(\partial M/r))$ (see Theorem \ref{thm:uniformbounddiskpattern} and \S \ref{sec:uub} for more details).

We remark that as one varies reflection groups with the same topological complexity, the hyperbolic polygons in the skinning image can become degenerate. Our Theorem \ref{thm:main}, in particular Equation \eqref{eqn:rate}, states that different hyperbolic polygons in the skinning image must degenerate in the same way. {Indeed, the only way to go to infinity in the Teichm\"uller space of such a hyperbolic polygon is when the hyperbolic distance between some pair of non-adjacent edges goes to zero, by a generalization of \cite[Prop.~2.7]{LZ23} (which deals with the case when all angles are zero). Then the polygon becomes long and thin along the direction of the pair of edges, and the extremal width for paths connecting the two edges becomes large. Equation \eqref{eqn:rate} implies that for all polygons in the skinning image the corresponding extremal width is comparably large, so they are also long and thin in the same way.}

\subsection{Discussion on related works}

\subsubsection{Skinning maps}
Skinning maps play an important role in Thurston's hyperbolization theorem.
The unique fixed point of the skinning map, guaranteed by contraction, provides compatible hyperbolic structures to glue smaller pieces to obtain more complicated hyperbolic 3-manifolds (see \cite{Thu86}).

General properties of the skinning map remain mysterious. It is known that the skinning map of a compact acylindrical manifold is never constant \cite{DK09}, and in fact finite-to-one \cite{Dum15}. Recently, Gastor constructed an explicit family of skinning maps with critical points \cite{Gas16}. These examples come from deformation of a Kleinian group whose limit set is the Apollonian circle packing. However, it is not known whether skinning maps have critical points in general.

We refer to \cite{BBCM20} for a recent extension of Thurston's Bounded Image Theorem to pared 3-manifolds with incompressible boundary that are not necessarily acylindrical.

We also remark that as observed in \cite{BKM21}, the uniform upper bound of the derivative of the skinning map in \cite{McM90} depends only on the topological type of the surface $\partial M$ (c.f. \cite{BEK20} for a related uniform bound on the Thurston's pull back operator).

\subsubsection{Circle packings and discrete extremal lengths}
Circle packings (or, more generally, disk patterns), their deformation spaces, and rigidity problems have been studied in \cite{RS87, Sch91, He99, RHD07, HL13, HL17}.
More generally, these problems have been studied for circle packings on complex projective surfaces \cite{Thu22, KMT03, KMT06, Lam21, BW23}.
Circle packings and disk patterns have also long been employed to study geometric structures and their deformation spaces \cite{Bro85, Bro86, FS97}.
Recently, \cite{LZ23} has explored the connections between the skinning map, renormalization, and circle packings (see also \cite{LZ24}). Our main result suggests that it may be possible to extend the uniform contraction property of the renormalization operator, as discussed in \cite{LZ23}, to cases with non-fixed combinatorics.

Discrete and combinatorial extremal lengths have been employed to investigate various surface uniformization problems (see \cite{Sch93, Can94, CFP94, Sch95, BK02, BM13, Lee18, Thu19, NY20}). The deep connections between circle packings and these combinatorial extremal lengths are thoroughly discussed in \cite{Hai09}. A crucial aspect of these approaches is utilizing circle packings to effectively translate combinatorial data into analytical data (see also \cite{RS87, BS04, Wil01, IM23}). Our method also leverages this powerful principle.

\subsubsection{(Quasi-)duality}\label{subsec:qd}
Duality of extremal length / width in the conformal setting is already known by Ahlfors and Beurling (see e.g.\ \cite{AB50}). More generally, quasi-duality is known to hold for sufficiently regular metric spaces \cite{JL20, Loh21, LR21}, which can then be used to characterize quasiconformal maps between such spaces. See also \cite{Geh62, Zie67, Loh23} for higher-dimensional generalizations.

Quasi-duality has been applied to prove uniformization theorems of metric surfaces \cite{Raj17, RR19, RRR21, EBPC22, Iko22, MW24}. Notably, in \cite{Raj17}, it is shown that a metric space homeomorphic to $\mathbb{R}^2$ is quasiconformally homeomorphic to an open domain in $\C$ if and only if it is \emph{reciprocal}, which in particular requires uniform quasi-duality for all quadrilaterals in the space. This result is generalized in \cite{NR22}. See also a recent exposition by Rajala on quasi-duality and its applications \cite{Raj24}.

Duality for certain discrete analogues has also been established in the context of edge metrics on graphs and networks \cite{ACF+19}.

\subsection*{Acknowledgement}
The authors thank Y. Minksy for asking the question and suggesting the connection between uniform diameter bound and circle packings. We also thank the referee for many suggestions that improved the exposition of the paper.

\section{Disk patterns and reflection groups}\label{sec:cd}
In this section, we establish many connections between disk patterns and Kleinian reflection groups. Many results are generalized from circle packings and kissing reflection groups in \cite{LLM22}.

\subsection{Realizable Coxeter graphs}
In this subsection, we introduce a combinatorial object (Coxeter graphs) to encode both disk patterns with angles in $\{\pi/n:n\in\mathbb{N}_{\ge2}\}\cup\{0\}$ and reflection groups, and study their relationships.
In particular, we show that their deformation spaces are naturally identified (Theorem~\ref{thm:moduli}).
Our main ingredient is a version of Koebe-Andreev-Thurston's theorem on realizable disk patterns (Theorem~\ref{thm:KAT}).

Let $\mathcal{G}$ be a connected finite simple plane graph, and let $\mathcal{V}$ and $\mathcal{E}$ be the set of vertices and edges of $\mathcal{G}$ respectively. Let $\omega: \mathcal{E} \longrightarrow [0,\pi/2]$ be a weight function on the set of edges. We call such an edge-weighted graph a \emph{Coxeter graph}.

Given a Coxeter graph $(\mathcal{G},\omega)$, a face of $F$ of $(\mathcal{G},\omega)$ is said to be
\begin{itemize}
    \item \emph{elliptic} if it is triangular and the sum of weights $\omega(e_1)+\omega(e_2)+\omega(e_3)>\pi$ for the three edges $e_1,e_2,e_3$ bounding $F$;
    \item \emph{parabolic} if the sum of weights $\omega(e_1)+\cdots+\omega(e_n)=(n-2)\pi$ for the edges $e_1,\ldots,e_n$ bounding $F$;
    \item \emph{hyperbolic} otherwise.
\end{itemize}
It is easy to see that if $F$ is parabolic, then it is either a triangle or a quadrilateral.
It is also easy to see that if $F$ is hyperbolic with $n$ sides, then the weights on the edges add up to $<(n-2)\pi$.

For a more uniform presentation, we always assume that in the Coxeter graph $(\mathcal{G},\omega)$, there does not exist a pair of adjacent triangular parabolic faces sharing an edge with weight $0$.
If there does exist such a pair, it is easy to see that the four edges other than the one shared by the two triangles all have weight $\pi/2$. We can remove the common edge, and combine the two parabolic faces into a single quadrilateral parabolic face. This in fact does not affect the combinatorics of the disk patterns encoded by the graph, see Definitions~\ref{defn:completion} and \ref{defn:dp}, as well as the discussion between them.

\begin{defn}\label{defn:completion}
Given a Coxeter graph $(\mathcal{G},\omega)$, we define its \emph{completion} $(\overline{\mathcal{G}},\overline{\omega})$ as follows. First suppose $|\mathcal{V}|\ge5$. Let $\mathcal{F}_{pq}$ be the collection of all quadrilateral parabolic faces.
For any $F\in\mathcal{F}_{pq}$, let $v^F_1,v^F_2,v^F_3,v^F_4$ be the four vertices on $\partial F$ in a cyclic order. Let $e^F_{13}$ and $e^F_{24}$ be the two diagonals connecting $v^F_1,v^F_3$ and $v^F_2,v^F_4$ respectively. Then $\overline{\mathcal{G}}=(\mathcal{V},\overline{\mathcal{E}})$ with $$\overline{\mathcal{E}}=\mathcal{E}\cup\bigcup_{F\in\mathcal{F}_{pq}}\{e^F_{13},e^F_{24}\}.$$
The weight function satisfies $\overline{\omega}=\omega$ on $\mathcal{E}$, and $\overline{\omega}\equiv0$ on $\overline{\mathcal{E}}\backslash\mathcal{E}$.

If $\mathcal{G}$ is a quadrilateral with $\omega\equiv \pi/2$, then we only add the diagonals to one of the two faces of $\mathcal{G}$. In all other cases, $(\overline{\mathcal{G}},\overline{\omega})=(\mathcal{G},\omega)$.
\end{defn}

Note that if $\mathcal{F}_{pq}$ is nonempty, then the completion is no longer a plane graph. This definition is motivated by the observation that there are \emph{extraneous tangencies} for quadrilateral parabolic faces; see \cite[Ch.~13]{Thu22}. For easier references, we still refer to the quadrilateral $v_1v_2v_3v_4$ as a parabolic face of $\overline{\mathcal{G}}$. The four triangles $v_1v_2v_3$, $v_2v_3v_4$, $v_3v_4v_1$, $v_4v_1v_2$ formed by adding the diagonals are called \emph{extraneous parabolic faces} of $\overline{\mathcal{G}}$.

\begin{defn}\label{defn:dp}
    A \textit{disk pattern} with Coxeter graph $(\mathcal{G}, \omega)$ is a collection of closed round disks $\mathcal{P}\coloneqq\{D_v, v\in \mathcal{V}\}$ in the Riemann sphere $\hat{\mathbb{C}}$ so that
    \begin{itemize}
        \item $D_v \cap D_w = \emptyset$ if $v, w$ are not adjacent in the completion $(\overline{\mathcal{G}},\overline{\omega})$;
        \item $D_v$ intersects $D_w$ at an angle $\overline{\omega}(e)$ if $e$ is an edge connecting $v, w$ in $\overline{\mathcal{G}}$; in particular, this means that $D_v$ and $D_w$ are tangent to each other when $\overline{\omega}(e)=0$.
    \end{itemize}
If such a disk pattern exists, we say $\mathcal{P}$ \textit{realizes} $(\mathcal{G}, \omega)$, and $(\mathcal{G}, \omega)$ is \emph{realizable}.
We remark that by definition, each disk of $\mathcal{P}$ is marked by a vertex of $\mathcal{G}$.
We denote by $\Teich(\mathcal{G}, \omega)$ the space of disk patterns realizing $(\mathcal{G}, \omega)$ up to M\"obius transformations that preserves the markings.
\end{defn}
We remark that if $\omega\equiv 0$ on $\mathcal{E}$, then the disk pattern $\mathcal{P}$ realizing $(\mathcal{G}, \omega)$ is a circle packing.
We endow $\Teich(\mathcal{G},\omega)$ with the Hausdorff topology, i.e. $\mathcal{P}_i\to\mathcal{P}$ if (up to M\"obius transformations) the corresponding disks $D_{v,i}\to D_v$ on $\hat{\mathbb{C}}$.

The following statement is a slight generalization of the classical Koebe-Andreev-Thurston theorem (see e.g.\ Chapter 13 of \cite{Thu22}).
\begin{theorem}[Koebe-Andreev-Thurston]\label{thm:KAT}
    Assume $(\mathcal{G},\omega)$ has at least one hyperbolic face, or contains at least 6 vertices. Then $(\mathcal{G},\omega)$ is realizable if and only if the following two conditions are satisfied.
    \begin{enumerate}[label=\normalfont{(\Alph*)}]
        \item Given any 3-cycle of edges $e_1,e_2,e_3$ in the completion $(\overline{\mathcal{G}},\overline{\omega})$, if $\omega(e_1)+\omega(e_2)+\omega(e_3)\ge\pi$, then they bound a (elliptic or parabolic) face of $\mathcal{G}$, or an extraneous parabolic face of $\overline{\mathcal{G}}$.
        \item Given any 4-cycle of edges $e_1,e_2,e_3,e_4$, if  $\omega(e_1)+\omega(e_2)+\omega(e_3)+\omega(e_4)=2\pi$, then they bound a parabolic face of $\overline{\mathcal{G}}$ or two elliptic faces of $\mathcal{G}$.
    \end{enumerate}
    Moreover, if $(\mathcal{G},\omega)$ is realizable, then it has a unique realization up to M\"obius transformations if and only if all hyperbolic faces of $\mathcal{G}$ are triangular.
\end{theorem}
We remark that if $\omega\equiv 0$, then the two conditions are automatically satisfied, and the result above reduces to the classical one on circle packings.

We will briefly sketch a proof of realizability, and discuss more about (non)uniqueness in later sections. Many versions of this theorem found in current literature treat slightly different cases. We refer to \cite{RHD07} for some of the nuances.

\begin{proof}
Necessity of the two conditions follow from \cite[\S3]{RHD07}. We divide the proof of sufficiency into several steps, covering increasingly more cases.
\paragraph{\bfseries Step 1} (Triangular graph)
Suppose $(\mathcal{G},\omega)$ is a triangulation with $\ge6$ vertices, $\omega(e)>0$ for any edge $e$, and only contains elliptic faces.
Then realizability follows from \cite[Theorem~1.4]{RHD07} directly:
Conditions (1)-(2) there are satisfied by our extra assumptions ($\omega>0$ and only elliptic faces);
Conditions (3)-(4) are contrapositives of the two conditions in our version;
Checking Condition (5) is not necessary when we have at least 6 vertices, by \cite[Proposition~1.5]{RHD07}.

\paragraph{\bfseries Step 2} (Limiting argument)
We now allow triangular parabolic faces. For simplicity, suppose only one face is parabolic; the general case follows by induction. Suppose $e_1,e_2,e_3$ bounds a parabolic face $F$. Choose $\omega_{n,i}\to\omega_i=\omega(e_i)$ so that $\omega_{n,i}\in(0,\pi/2]$ and $\omega_{n,1}+\omega_{n,2}+\omega_{n,3}>\pi$. Consider a Coxeter graph with the same underlying graph, but with weight function $\omega_n$ satisfying $\omega_n(e_i)=\omega_{n,i}$, and having the same value as $\omega$ on all other edges. It is easy to see that for all $n$ large enough, Conditions (A) and (B) still hold. By Step 1, $(\mathcal{G},\omega_n)$ is realizable. Letting $n\to\infty$, we conclude that the original Coxeter graph is also realizable.

\paragraph{\bfseries Step 3} (Quadrilateral parabolic faces) We next allow quadrilateral parabolic faces. For simplicity, suppose there is only one quadrilateral parabolic face $F$, bounded by edges $e_1,e_2,e_3,e_4$. Let $e_F$ be either one of the diagonals of $F$. Consider a new Coxeter graph $(\mathcal{G}_n,\omega_n)$ by adding the new edge $e_F$ to $\mathcal{G}$, with weight $\pi/n$.
It is easy to see that both Conditions (A) and (B) are still satisfied, so $(\mathcal{G}_n,\omega_n)$ is realizable.
Letting $n\to\infty$, we conclude that $(\mathcal{G},\omega)$ is realizable.

\paragraph{\bfseries Step 4} (Triangulation of a general graph)
Suppose now $(\mathcal{G},\omega)$ has at least one hyperbolic face. For each hyperbolic face $F$, add a vertex $v_F$ in its interior and connect it to all vertices on $\partial F$.
We also assign weight $0$ to all the new edges.
Note that in this new graph, each hyperbolic face with $n$-sides is divided into $n$ hyperbolic triangles.
Now for each new hyperbolic triangle, add an additional vertex and connect it to the three vertices of the triangle.
We assign weight $\pi/2$ to these new edges.
It is not hard to see that the new Coxeter graph has at least $7$ vertices, and has no hyperbolic face. Conditions (A) and (B) remain true.
We can thus apply Step 2 and Step 3 to conclude that the new graph is realizable. Removing the additional disks, we conclude that the original graph is realizable.

Finally, we note that if $\overline{\omega}(e)=0$ for some edge $e$, then either $e$ is part of an extraneous parabolic face, or bounds a hyperbolic face.
In the latter case, the construction above will make $e$ part of an extraneous parabolic face in the completion of the new graph.
So we can also allow zero weights on the edges.

This completes the proof of the existence part of Theorem~\ref{thm:KAT}.
\end{proof}
We remark that the idea of subdividing non-triangular faces has already been applied in \cite{Thu22}.

We also need the following special case not covered by Theorem~\ref{thm:KAT} for later applications.
\begin{prop}[Koebe-Andreev-Thurston for triangular prism]\label{prop:prism}
    Suppose $(\mathcal{G},\omega)$ is the graph shown in Figure~\ref{fig:prism}. Suppose also that $(\mathcal{G},\omega)$ does not contain any hyperbolic face. Then $(\mathcal{G},\omega)$ is realizable if and only if the following conditions hold.
    \begin{enumerate}[label=\normalfont{(\Roman*)}]
        \item $\omega(v_1v_2)+\omega(v_2v_3)+\omega(v_3v_1)<\pi$;
        \item $\omega(av_i)+\omega(v_ib)+\omega(bv_j)+\omega(v_ja)+\omega(av_k)+\omega(v_kb)<3\pi$ and $\omega(av_i)+\omega(v_ib)+\omega(bv_j)+\omega(v_ja)+\omega(v_iv_k)+\omega(v_kv_j)<3\pi$
        for any $\{i,j,k\}=\{1,2,3\}$.
    \end{enumerate}
\end{prop}
\begin{figure}[htp]
    \centering
    \includegraphics[width=0.5\linewidth]{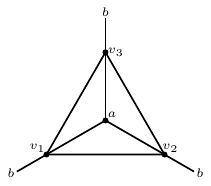}
    \caption{Coxeter graph for Proposition~\ref{prop:prism}. A vertex labelled $b$ is put at infinity.}
    \label{fig:prism}
\end{figure}
The proof is a combination of \cite[Theorem~1.4]{RHD07} (especially Condition (5) there) and a limiting argument as in Step 2 above to handle 0 weights.

\subsection{Reflection groups associated to a Coxeter graph}
For applications on reflection groups, from now on, we restrict the weight function $\omega: \mathcal{E} \longrightarrow \{\frac{\pi}{n}\colon n \in \N_{\geq 2}\} \cup \{0\}$, and assume that $(\mathcal{G},\omega)$ is realizable. 
We remark that without this restriction, many estimates in later sections still hold, but some constants will depend on $\omega$.

We also assume that $(\mathcal{G},\omega)$ has at least one hyperoblic face.
Let $\Aut^{\pm}(\hat{\mathbb{C}})\cong\Isom(\mathbb{H}^3)$ be the group of M\"obius and anti-M\"obius transformations.

Given $\mathcal{P}\in\Teich(\mathcal{G}, \omega)$, consider the group $G=G_{\mathcal{P}}$ generated by reflections $r_v$ in the boundary circles $C_v$ of $D_v$.
Note that for any $v,w\in\mathcal{V}$, $r_v\circ r_w$ is
\begin{itemize}
    \item an elliptic element of order $n$ if there is an edge $e$ connecting them with $\omega(e)=\pi/n$ for some integer $n\ge2$;
    \item a parabolic element if there is an edge $e$ connecting them in the completion $\overline{\mathcal{G}}$ with $\overline{\omega}(e)=0$;
    \item a hyperbolic element if no edge connects them in the completion.
\end{itemize}

We can construct a fundamental domain for the action of $G$ on $\mathbb{H}^3$ explicitly.
For each vertex $v$, let $P_v$ be the corresponding geodesic plane in $\mathbb{H}^3$ with $C_v$ as its boundary at infinity, oriented with normal vectors pointing towards $D_v$.
Let $\mathcal{H}_v$ be the half space bounded by $P_v$ so that the normals on $P_v$ points away from $\mathcal{H}_v$. Set
$$\mathcal{H}_{\mathcal{P}}:=\bigcap_{v\in\mathcal{V}}\mathcal{H}_v.$$
Note that the planes $P_v$ and $P_w$ intersect in a dihedral angle $\pi/n$ if $\omega(vw)=\pi/n$. Consider also the set
$$\Pi\coloneqq\hat{\mathbb{C}}-\bigcup_{v\in\mathcal{V}}\overline{D_v}.$$
Since $(\mathcal{G},\omega)$ contains at least one hyperbolic face, $\Pi$ is nonempty. In fact, each connected component of $\Pi$ is the interior of a polygon bounded by circular arcs, corresponding to a face $F$ of $\mathcal{G}$. We denote this connected component by $\Pi_F$, and calls it the \emph{interstice} of the pattern for the face $F$.

Note that the infinite ends of $\mathcal{H}_{P}$ extend to the sphere at infinity $\hat{\mathbb{C}}$ as interstices. In particular, we conclude that $\mathcal{H}_{P}$ is nonempty and in fact has nonempty interior. By Poincar\'e's polyhedron theorem, $G$ is a discrete subgroup of $\Aut^{\pm}(\hat{\mathbb{C}})$, and $\mathcal{H}_{\mathcal{P}}$ is a fundamental domain for the action of $G$ on $\mathbb{H}^3$. The group $G$ contains an index-2 subgroup of orientation-preserving elements, which we denote by $\widetilde G$. It follows from the construction of a fundamental domain above that $\widetilde G$ is \emph{geometrically finite}, i.e.\ the action of $\widetilde G$ on $\mathbb{H}^3$ has a finite-sided fundamental polyhedron.

Let $\Lambda$ and $\Omega$ be the limit set and domain of discontinuity of $\widetilde G$ respectively.
Let $\Pi_F$ be an interstice and suppose $G_F$ is the subgroup of $G$ generated by reflections in the circles corresponding to the vertices $v\in\partial F$. Define $\Omega_F=\bigcup_{g\in G_F}g\cdot\overline{\Pi_F}$.

We have the following simple lemma.
\begin{lem}
Suppose that $(\mathcal{G},\omega)$ with $\omega: \mathcal{E} \longrightarrow \{\frac{\pi}{n}\colon n \in \N_{\geq 2}\} \cup \{0\}$ is realizable and has at least one hyperbolic face. Fix $\mathcal{P}\in\Teich(\mathcal{G},\omega)$, and let $G=G_{\mathcal{P}}$ be the associated reflection group. Let $\Omega$ be the domain of discontinuity of $G$. Then
    \begin{enumerate}
        \item The group $G$ is nonelementary.
        \item There are bijective correspondences between the set of hyperbolic faces of $(\mathcal{G},\omega)$, the set of interstices of $\mathcal{P}$, and $G$-orbits of connected components of $\Omega$ given by $F\longleftrightarrow\Pi_F\longleftrightarrow G\cdot\Omega_F$.
    \end{enumerate}
\end{lem}
\begin{proof}
    For Part 1, it suffices to show that for any hyperbolic face $F$, the subgroup $G_F$ is nonelementary. Indeed, since $\mathcal{H}_{\mathcal{P}}$ gives a fundamental domain, $G_F\cdot\overline{\Pi_F}$ tiles $\Omega_F$. If $G_F$ is elementary, then $\hat{\mathbb{C}}-\Omega_F$ consists of at most $2$ points. On the other hand, if we set $\widetilde{G}_F$ to be the subgroup of orientation-preserving elements in $G_F$, then $\widetilde{G}_F\backslash\Omega_F$ is hyperbolic, so $\Omega_F$ admits a hyperbolic metric as well. This is a contradiction.

    For Part 2, let $\Pi_{F'}$ be a different interstice. Then $\Omega_{F'}$ is disjoint from $\Omega_{F}$. In particular, $\Omega_{F'}$ also admits a hyperbolic metric, hence $F'$ is hyperbolic.

    Noting that $\Omega_F$ is precisely the connected component of $\Omega$ containing $\Pi_F$, the bijective correspondences follow.
\end{proof}

\subsection{Moduli spaces and reflection groups}
We now elaborate more on the (non)uniqueness of disk patterns realizing $(\mathcal{G},\omega)$, in terms of the quasiconformal deformation space of the associated reflection groups.

A \emph{quasiconformal deformation} of $G$ is a discrete and faithful representation $\xi:G\longrightarrow\Aut^{\pm}(\hat{\mathbb{C}})$ that preserves parabolics, induced by a quasiconformal map $f:\hat{\mathbb{C}}\longrightarrow\hat{\mathbb{C}}$ (i.e.\ $\xi(g)=f\circ g\circ f^{-1}$ for any $g\in G$).

The \emph{quasiconformal deformation space} of $G$ is defined as
$$\QC(G):=\{\xi:G\longrightarrow\Aut^{\pm}(\hat{\mathbb{C}})\text{ is a quasiconformal deformation}\}/\sim$$
where $\xi\sim\xi'$ if they are conjugates of each other by a M\"obius transformation. Note that as a representation of a fixed group $G$, we view the image $\xi(G)$ as \emph{marked} by the group $G$.

We endow $\QC(G)$ with the algebraic topology, i.e. $\xi_i\to\xi$ if (up to M\"obius transformations) $\xi_i(g)\to\xi(g)$ for any $g\in G$. We have the following identification of deformation spaces.
\begin{prop}
    Fix $\mathcal{P}\in\Teich(\mathcal{G},\omega)$ with $\omega: \mathcal{E} \longrightarrow \{\frac{\pi}{n}\colon n \in \N_{\geq 2}\} \cup \{0\}$ and let $G$ be the corresponding reflection group. The association $\mathcal{P}'\in\Teich(\mathcal{G},\omega)\longmapsto G_{\mathcal{P}'}\in\QC(G)$ induces a homeomorphism $\Teich(\mathcal{G},\omega)\cong\QC(G)$.
\end{prop}
\begin{proof}
    We first show that the map is well-defined. For each hyperbolic face $F$ of $(\mathcal{G},\omega)$, there is a quasiconformal map $\Psi_F$ between interstices $\Pi_F$ of $\mathcal{P}$ and $\Pi'_F$ of $\mathcal{P}$ for the face $F$. Let $\mu_F$ be its Beltrami differential. Using the action of $G$, we then obtain an invariant Beltrami differential on $\Omega$. Since the limit set has zero area, this may be viewed as a Beltrami differential on $\hat{\mathbb{C}}$. The Measurable Riemann Mapping Theorem then provides a quasiconformal map that conjugates the actions.
    
    Clearly the map is injective.
    For surjectivity, let $\xi:G\longrightarrow\hat{\mathbb{C}}$ be a quasiconformal deformation induced by $f:\hat{\mathbb{C}}\longrightarrow\hat{\mathbb{C}}$.
    Then for any $v\in\mathcal{V}$, the reflection $r_v$ in $C_v$ is mapped to an element $\xi(r_v)$ whose fixed point set is the Jordan curve $f(C_v)$. Therefore $\xi(r_v)$ must be a reflection as well, and $f(C_v)$ is a circle.
    Moreover, as $\xi$ is faithful and type-preserving, $\xi(r_v)\circ\xi(r_w)$ has the same type and order as $r_v$ and $r_w$, and hence the angle between $f(C_v)$ and $f(C_w)$ remains the same as that between $C_v$ and $C_w$. Thus $f(\mathcal{P})$ is another disk pattern realizing $(\mathcal{G},\omega)$, and $\xi(G)$ is generated by reflections in the circles of $f(\mathcal{P})$.

    Finally, continuity can be easily checked from definition.
\end{proof}

Consider the Kleinian 3-orbifold $M\coloneqq \widetilde{G}\backslash(\mathbb{H}^3\cup\Omega)$. Its boundary $\partial M=\bigcup_FX_F$ has a connected component for each hyperbolic face $F$ of $\mathcal{G}$.
Conformally, $X_F\cong\widetilde{G}_F\backslash\Omega_F$.
In fact, $X_F$ can be constructed as the double of $\Pi_F$, with punctures or cone points of order $n$ if the corresponding edges have weight $0$ or $\pi/n$. The surface $X_F$ has an anti-conformal involution $r=r_F$ given by exchanging the two copies of $\Pi_F$.

Let $\Teich^{r}(X_F)$ be the quasiconformal deformation space of $X_F$ invariant under the mapping class given by $r$. In fact, we may view this as a deformation space $\Teich(\Pi_F)$ of the interstice $\Pi_F$. Using the quasiconformal deformation theory of Ahlfors, Bers, Maskit and others (see e.g.\ \cite{Sul81}), we can argue as \cite[\S2]{LZ23} and obtain the following identification.
\begin{theorem}\label{thm:moduli}
    For any realizable $(\mathcal{G},\omega)$ with $\omega: \mathcal{E} \longrightarrow \{\frac{\pi}{n}\colon n \in \N_{\geq 2}\} \cup \{0\}$,
    $$\Teich(\mathcal{G},\omega)\cong\QC(G)\cong\prod_{F\in\mathcal{F}_h}\Teich(\Pi_F),$$
    where $\mathcal{F}_h$ is the set of hyperbolic faces of $\mathcal{G}$.
\end{theorem}
    As mentioned in \cite[\S2]{LZ23}, the Fenchel-Nielson coordinates on $\Teich(X_F)$ give a diffeomorphism $\Teich(\Pi_F)\cong(\mathbb{R}^+)^{n-3}$, where $n$ is the number of sides of $F$. In particular, $\Teich(\mathcal{G},\omega)$ contains a unique point if and only if all hyperbolic faces are triangular.

\subsection{Limit sets of reflection groups}
A graph $\mathcal{G}$ is said to be \emph{$k$-connected} if it contains more than $k$ vertices and remains connected after removing any $k-1$ vertices together with edges incident to them.
An \emph{elliptic connection} of $(\mathcal{G},\omega)$ is an edge $e$ in $\mathcal{G}$ connecting two nonadjacent vertices on the boundary of a hyperbolic face with $\omega(e)>0$.

We now prove the following relation between connectedness of limit sets of $G_\mathcal{P}$ for any $\mathcal{P}\in\mathcal{G}$ and connectedness of $\mathcal{G}$. This generalizes Proposition~3.4 in \cite{LLM22}.
\begin{theorem}\label{thm:connected_limit_set}
    Let $(\mathcal{G},\omega)$ be a realizable connected simple plane graph with weight $\omega:\mathcal{E}\longrightarrow\{\frac\pi n:n\in\mathbb{N}_{\ge2}\}\cup\{0\}$. Let $\mathcal{P}\in\Teich(\mathcal{G},\omega)$ and set $G=G_\mathcal{P}$.
    Then the limit set of $G$ is connected if and only if $(\mathcal{G},\omega)$ is 2-connected, and contains no elliptic connections.
\end{theorem}
\begin{proof}
    We will use the following fact: the limit set $\Lambda$ of $G$ is connected if and only if every component of the domain of discontinuity $\Omega$ is simply connected.

    Suppose first that $\mathcal{G}$ is not 2-connected. Then there exists a vertex $v$ so that removing $v$ and all edges incident to it separates the graph. It is easy to see that $v$ lies on the boundary of a face $F$ that is not a Jordan domain. Note that the face $F$ has at least $4$ sides.

    If $F$ is parabolic, then $\mathcal{G}$ contains exactly $3$ vertices $u,v,w$ with two edges $e_1$ connecting $u,v$, and $e_2$ connecting $v,w$. Moreover, $\omega(e_1)=\omega(e_2)=\pi/2$. It is easy to see that $G$ is elementary, whose limit set consists of two points, which is not connected.

    If $F$ is hyperbolic, then it is easy to see that $(r_v\cdot \overline{\Pi_F})\cup\overline{\Pi_F}\subseteq\Omega_F$ disconnects the limit set of $G$; cf.\ \cite[\S3.1]{LLM22}.

    Suppose now that $(\mathcal{G},\omega)$ is 2-connected, but contains an elliptic connection. That is, there exists a hyperbolic face $F$, and an edge $e$ not on $\partial F$ connecting two vertices $v,w$ of $F$. Suppose $\omega(e)=\pi/n$. For $k=0,2,\ldots,2n-1$, set
    $$g_k=\begin{cases}(r_v\circ r_w)^{(k-1)/2}\circ r_v&k\text{ is odd,}\\(r_v\circ r_w)^{k/2}&k\text{ is even.}\end{cases}$$
    It is easy to see that $\bigcup_kg_k\cdot\overline{\Pi_F}\subseteq\Omega_F$ disconnects the limit set.

    Suppose now that $(\mathcal{G},\omega)$ is 2-connected and contains no elliptic connections. For any hyperbolic face $F$ of $\mathcal{G}$, it must bound a Jordan domain and any additional edges connecting vertices of $F$ have weight $0$. This means that the cycle of disks corresponding to vertices of $F$ separates $\hat{\mathbb{C}}$ into two parts: one of them is $\Pi_F$, and the other may contain additional tangency among the disks (but no overlaps). Arguing similarly as \cite[\S3.1]{LLM22}, the domain of discontinuity $\Omega_F$ containing $\Pi_F$ is simply connected.
\end{proof}
Note that when $\omega\equiv0$, the graph automatically contains no elliptic connection, and our result here reduces to \cite[Proposition~3.4]{LLM22}. See Figure~\ref{fig:elliptic_connection} for some examples illustrating the theorem.

\begin{figure}[htp]
    \centering
    \begin{subfigure}{0.45\linewidth}
        \centering
        \includegraphics[width=0.8\linewidth]{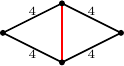}
        \caption{A 2-connected graph; the weights on all edges except the red one are $\pi/4$.}
    \end{subfigure}

    \begin{subfigure}[t]{0.46\linewidth}
    \centering
        \includegraphics[width=\linewidth]{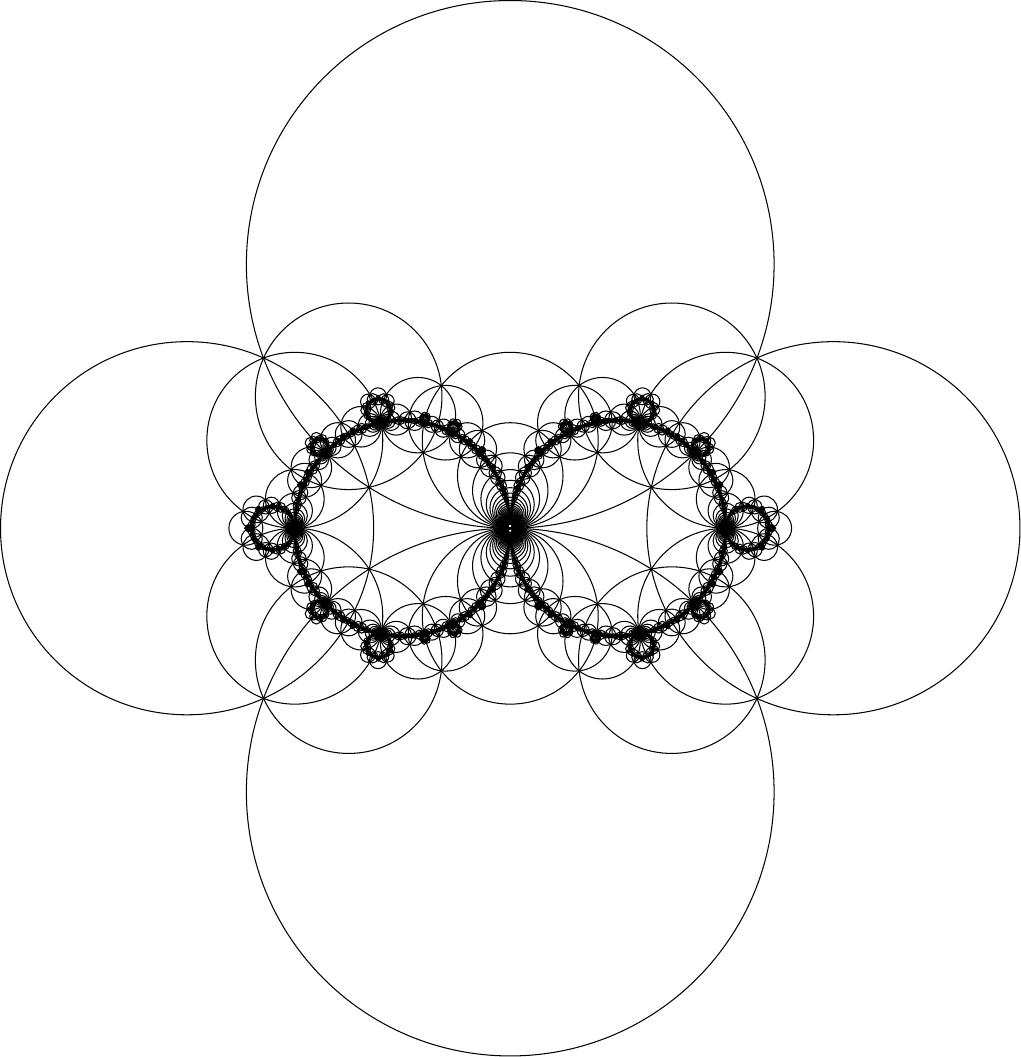}
        \caption{The weight on the red edge is $0$.}
    \end{subfigure}\hspace{0.01\linewidth}
    \begin{subfigure}[t]{0.52\linewidth}
    \centering
        \includegraphics[width=\linewidth]{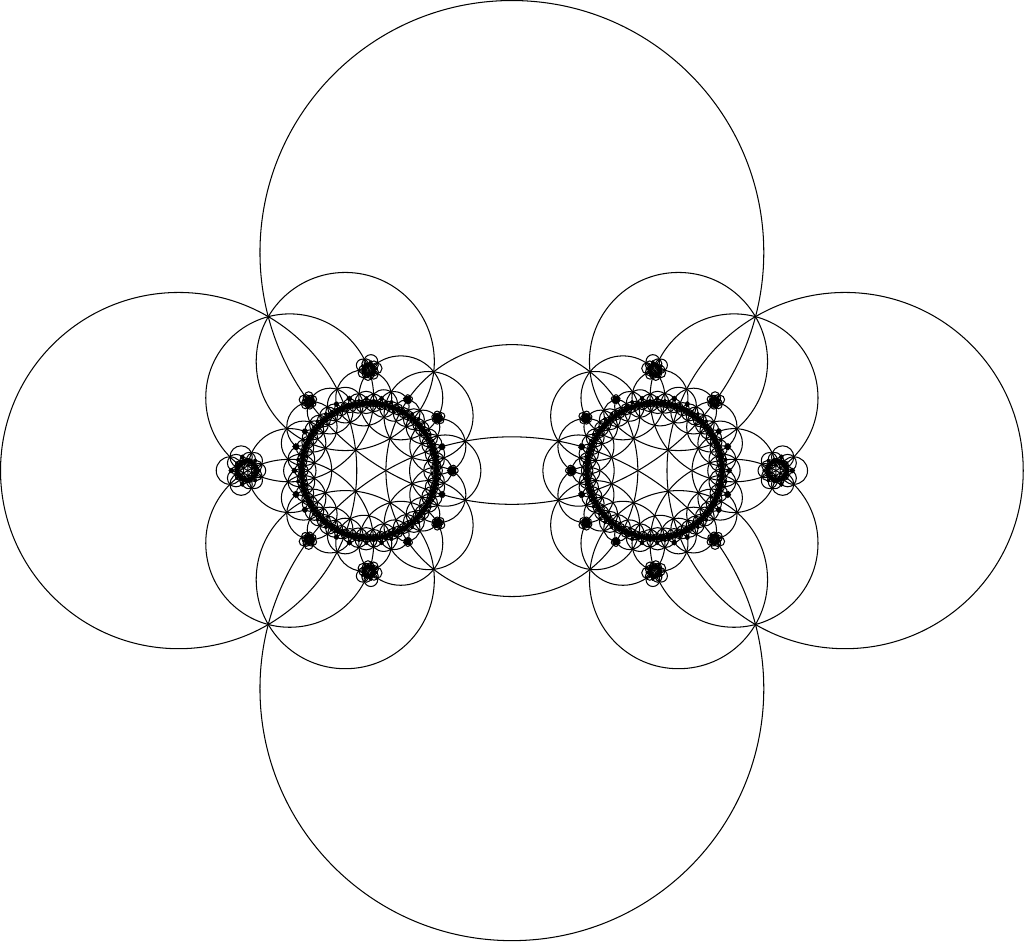}
        \caption{The weight on the red edge is $\pi/3$.}
    \end{subfigure}
    \caption{Two disk patterns with the same graph but different weights. The one on the right contains an elliptic connection.}
    \label{fig:elliptic_connection}
\end{figure}

\subsection{Acylindrical reflection groups}\label{subsec: acy}
We now consider the situation for acylindrical reflection groups. We briefly recall the topological condition of acylindricity. Let $(N,P)$ be a pared 3-manifold, where $N$ is a compact oriented 3-manifold with boundary, and $P\subseteq\partial N$ is a submanifold consits of incompressible tori and annuli. See \cite{Thu86} for a precise definition in arbitrary dimension. Set $\partial_0N=\partial N-P$. Then we say $(N,P)$ is \emph{acylindrical} if each component of $\partial_0N$ is incompressible, and every essential cylinder with both ends in $\partial_0N$ is boundary parallel.

For a geometrically finite hyperbolic 3-manifold $M$, let $\core_\epsilon(M)$ be the convex core of $M$ minus a small enough $\epsilon$-thin cuspidal neighborhoods for all cusps. Recall that the \emph{convex core} of $M$ is the smallest closed convex subset of $M$ containing all closed geodesics. Let $P\subseteq\partial\core_\epsilon(M)$ be the union of boundaries of all cuspidal neighborhoods, then $(\core_\epsilon(M),P)$ is a pared 3-manifold, and we say $M$ (and the corresponding Kleinian group $G$) is acylindrical if $(\core_\epsilon(M),P)$ is. More generally, a geometrically finite hyperbolic 3-orbifold is acylindrical if any of its finite manifold cover is.

It is well known that one can recognize acylindricity from the limit set when $M$ is geometrically finite of infinite volume: it is equivalent to the condition that every component of the domain of discontinuity of $G$ is a Jordan domain, and the closures of any two components share at most one point. See for example \cite[Proposition~8.4]{LZ23} and \cite[Lemma~11.2]{BO22}.

Given $(\mathcal{G},\omega)$, let $F$ be a hyperbolic face, and $v,w$ two nonadjacent vertices on $\partial F$.
If there exists a vertex $x\notin\partial F$ so that $\omega(xv)=\omega(xw)=\pi/2$, we call the path $vxw$ a \emph{right-angled 2-connection}.
We aim to prove the following statement relating connectedness of $(\mathcal{G},\omega)$ with acylindricity of the group $\widetilde G$. This generalizes Proposition~3.6 of \cite{LLM22}.
\begin{theorem}\label{thm:acy}
     Let $(\mathcal{G},\omega)$ be a realizable connected simple plane graph with at least 4 vertices and weight function $\omega:\mathcal{E}\longrightarrow\{\frac\pi n:n\in\mathbb{N}_{\ge2}\}\cup\{0\}$. Let $\mathcal{P}\in\Teich(\mathcal{G},\omega)$ and set $G=G_\mathcal{P}$, and $\widetilde G$ the index 2 subgroup of orientation-preserving elements.
     \begin{enumerate}
         \item Suppose $(\mathcal{G},\omega)$ is not a tetrahedron with only one hyperbolic face. Then the group $\widetilde G$ is acylindrical if and only if $(\mathcal{G},\omega)$ is 3-connected, and contains no right-angled 2-connections.
         \item Suppose $(\mathcal{G},\omega)$ is a tetrahedron with only one hyperbolic face. Set $e_1,e_2,e_3$ to be the three edges connecting the vertices of the hyperbolic face to the fourth vertex. Then the group $\widetilde G$ is acylindrical if and only if $\omega(e_1)+\omega(e_2)+\omega(e_3)<\pi$.
     \end{enumerate}
\end{theorem}

Note that if $\widetilde G$ is acylindrical, then its limit set is connected. It is also easy to see that if $(\mathcal{G},\omega)$ is 3-connected, then it contains no elliptic connection. So by Theorem~\ref{thm:connected_limit_set}, we may assume that $(\mathcal{G},\omega)$ is 2-connected and contains no elliptic connection.

Part of one direction follows from Lemma 3.7 and 3.8 in \cite{LLM22}.
\begin{lem}\label{lem:3conn}
    If $\widetilde G$ is acylindrical, then $(\mathcal{G},\omega)$ is 3-connected.
\end{lem}
\begin{proof}
    Suppose $\mathcal{G}$ is not 3-connected. Then Lemma 3.7 in \cite{LLM22} produces two vertices $v,w$ lying on the intersection of the boundaries of two faces $F_1, F_2$ of $\mathcal{G}$. Moreover, they are nonadjacent for at least one of the two faces, say $F_1$. In particular, $F_1$ must be a hyperbolic face.

    We can then argue similarly as Lemma 3.8 in \cite{LLM22} to finish the proof. The only modification needed is a third case: $r_v\circ r_w$ may be an elliptic element. But this means that the curve corresponding to $r_v\circ r_w$ in $X_{F_1}$ is homotopically trivial in (a manifold cover of) $M$, contradicting the fact that $X_{F_1}$ must be incompressible.
\end{proof}

For the other direction, we construct a new plane graph $(\widehat{\mathcal{G}},\widehat{\omega})$ as follows. For each hyperbolic face $F$ of $\mathcal{G}$, we add a new vertex $v_F$ and connect it to all vertices on $\partial F$. Moreover, we define the weight $\widehat{\omega}$ on these new edges to be $\pi/2$, and the same as $\omega$ otherwise.

Note that the new graph $(\widehat{\mathcal{G}},\widehat{\omega})$ has no hyperbolic face. Indeed, any face $F$ of $\widehat{\mathcal{G}}$ is either a face of $\mathcal{G}$ (in which case it is already elliptic or parabolic), or a triangle formed by $v_{F'}$ and two adjacent vertices $v,w$ of some face $F'$ of $\mathcal{G}$. Since $\widehat{\omega}(vv_{F'})+\widehat{\omega}(wv_{F'})+\widehat{\omega}(vw)\ge\pi$, the face $F$ is elliptic or parabolic.

Acylindricity can be characterized in terms of this new graph.
\begin{lem}
    The group $\widetilde G$ is acylindrical if and only if $(\widehat{\mathcal{G}},\widehat{\omega})$ is realizable.
\end{lem}
\begin{proof}
    Note that by a result of McMullen \cite{McM90}, $\widetilde G$ is acylindrical if and only if it has a quasiconformal deformation $\widetilde G_0$ whose domain of discontinuity consists of round disks. In fact, $\widetilde G_0$ arises as the unique fixed point of the skinning map. As the the skinning map maps the reflection locus $\QC(G)$ to itself, it is easy to see that the fixed point must lie on the reflection locus. That is, we may assume $\widetilde G_0$ is the index 2 subgroup of orientation-preserving elements in a reflection group $G_0$ corresponding to a disk pattern $\mathcal{P}_0\in\Teich(\mathcal{G},\omega)$.

    This means that for any hyperbolic face $F$ of $(\mathcal{G},\omega)$, we can add a circle perpendicular to all $D_v$ with $v\in\partial F$ -- this circle is the boundary of the corresponding component $\Omega_F$ of the domain of discontinuity. This is exactly equivalent to $(\widehat{\mathcal{G}},\widehat{\omega})$ being realizable, as desired.
\end{proof}

We are now ready to prove Theorem~\ref{thm:acy}.
\begin{proof}[Proof of Theorem~\ref{thm:acy}]
By Lemma~\ref{lem:3conn}, we may assume that $(\mathcal{G},\omega)$ is 3-connected.
If all faces of $(\mathcal{G},\omega)$ are elliptic or parabolic, then $\widetilde{G}$ is a lattice, and acylindrical by default. So we assume that $(\mathcal{G},\omega)$ contains at least one hyperbolic face.

If $\mathcal{G}$ contains exactly four vertices, then it must be a tetrahedron.  Suppose further that $(\mathcal{G},\omega)$ has only one hyperbolic face. Then $\widehat{\mathcal{G}}$ contains exactly 5 vertices. Since $\widetilde G$ is acylindrical if and only if $(\widehat{\mathcal{G}},\widehat{\omega})$ is realizable, we can apply apply Proposition~\ref{prop:prism} to conclude that the condition in Part (2) is necessary and sufficient.

For the remainder of the proof, we assume $(\mathcal{G},\omega)$ contains more than four vertices or has at least 2 hyperbolic faces. Then $(\widehat{\mathcal{G}},\widehat{\omega})$ has at least 6 vertices.

First suppose $\widetilde G$ is acylindrical. Then $(\widehat{\mathcal{G}},\widehat{\omega})$ is realizable. We claim that $(\mathcal{G},\omega)$ cannot contain any right-angled 2-connection. Suppose otherwise, and let $vxw$ be a right-angled 2-connection with $v,w\in\partial F$ for some hyperbolic face $F$. Then $v,x, w,v_F$ form a 4-cycle in $\widehat{\mathcal{G}}$ with weights on edges between them add up to $2\pi$. This is impossible by Theorem~\ref{thm:KAT}, as $v,w$ are assumed to be nonadjacent, and $v_F$ and $x$ are nonadjacent by construction.

Conversely, suppose the 3-connected graph $(\mathcal{G},\omega)$ contains no right-angled 2-connections.
We need to check that $(\widehat{\mathcal{G}},\widehat{\omega})$ satisfy all conditions of Theorem~\ref{thm:KAT}. First note that any new 3-cycles bound a new face, so the first condition is met. Furthermore, any new 4-cycle must be of one of the following two possibilities. 

The first possibility is $v_{F_1}vv_{F_2}w$ for two hyperbolic faces $F_1,F_2$ of $(\mathcal{G},\omega)$ and $v,w\in\partial F_1\cap\partial F_2$. As $(\mathcal{G},\omega)$ is 3-connected, $v,w$ must be adjacent in both $F_1$ and $F_2$. So the second condition is met in this case.

Another possibility is $v_Fvxw$ for some hyperbolic face $F$, $v,w\in\partial F$ and $x\notin\partial F$. Since no right-angled 2-connection exists, we conclude that $v,w$ are adjacent in $F$. So the second condition is also met in this case.

Hence $(\widehat{\mathcal{G}},\widehat{\omega})$ is realizable, and thus $\widetilde G$ is acylindrical, as desired.
\end{proof}

Note that when $\omega\equiv0$, the graph automatically contains no right-angled 2-connection, and our result here reduces to \cite[Proposition~3.6]{LLM22}. See Figure~\ref{fig:right_angled_2_connection} for some examples illustrating the theorem.

\begin{figure}[htp]
    \centering
    \begin{subfigure}{0.4\linewidth}
        \centering
        \includegraphics[width=0.75\linewidth]{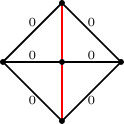}
        \caption{A 3-connected graph. The weights on all edges except the red ones are $0$.}
    \end{subfigure}
    
    \begin{subfigure}[t]{0.56\linewidth}
    \centering
        \includegraphics[width=\linewidth]{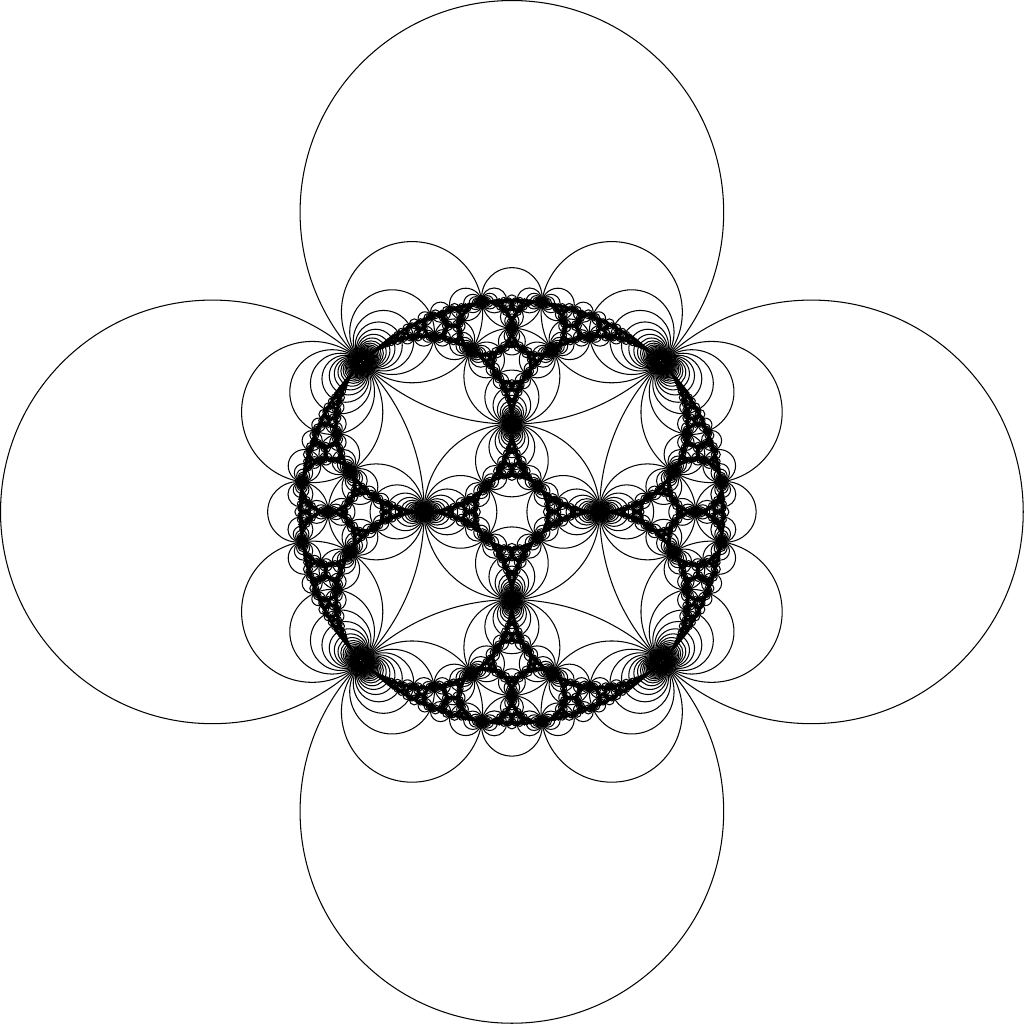}
        \caption{The weights on the red edges are $0$; this is acylindrical, and the limit set is homeomorphic to a circle packing.}
        \label{fig:no_right_angled_2_connection}
    \end{subfigure}
    \hspace{0.01\linewidth}
    \begin{subfigure}[t]{0.41\linewidth}
    \centering
        \includegraphics[width=\linewidth]{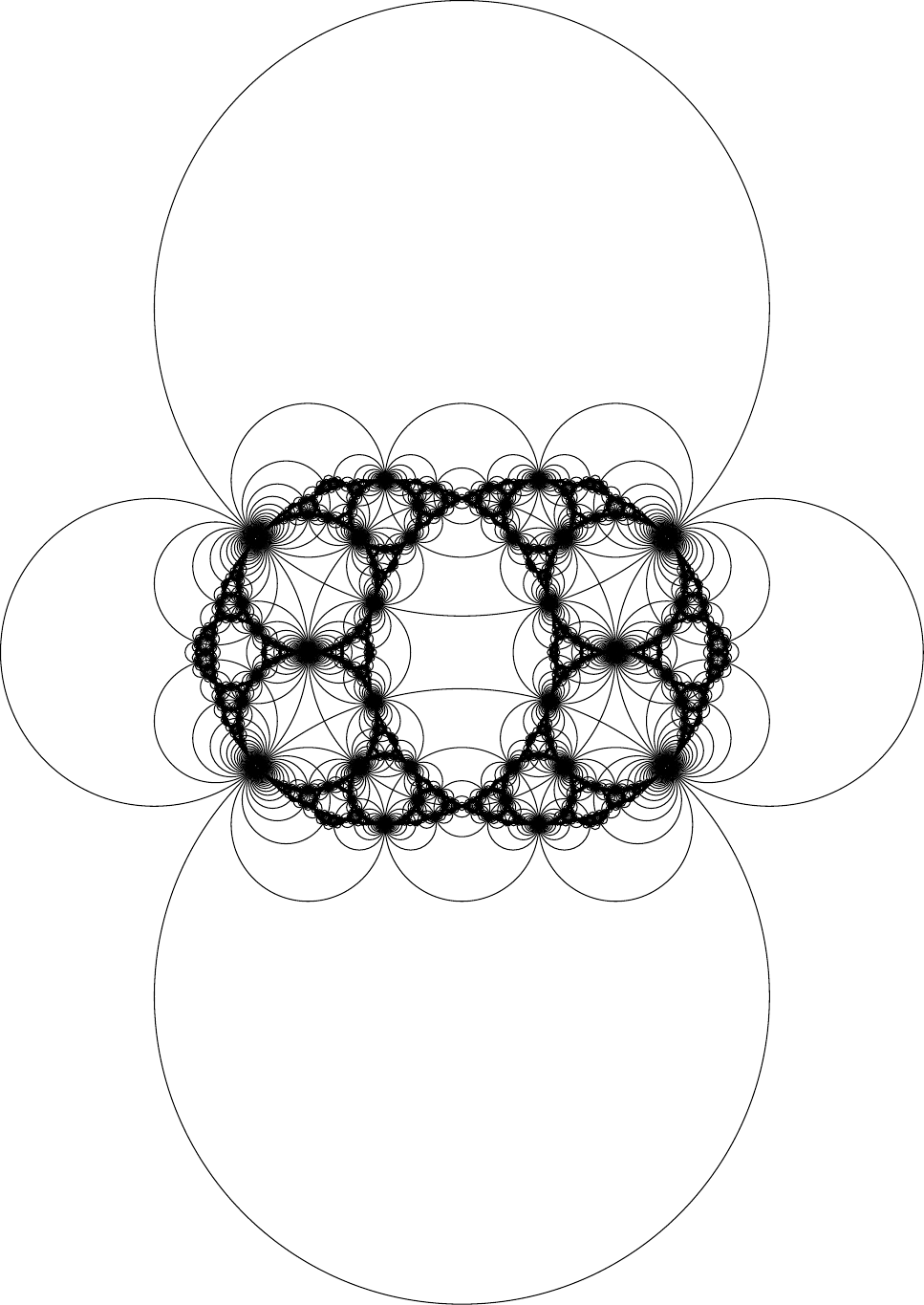}
        \caption{The weights on the red edges are $\pi/2$; this is not acylindrical.}
    \end{subfigure}
    \caption{Two disk patterns with the same graph but different weights. The one on the right contains a right-angled 2-connection.}
    \label{fig:right_angled_2_connection}
\end{figure}

\section{Uniform quasi-duality for discrete extremal width}\label{sec:vel}
In this section, we prove uniform quasi-duality for extremal widths of polygonal subdivision graphs. We start by introducing the notion of vertex extremal length and width for graphs and polygonal subdivision graphs in \S \ref{subsec:vel} and \S \ref{subsec:psg}. The main theorem of this section is Theorem \ref{thm:combunifbound}, whose proof occupies the remaining parts.
\subsection{Vertex extremal length / width}\label{subsec:vel}
The notion of discrete extremal length, in various different settings, was introduced and studied by Duffin, Cannon and Schramm \cite{Can94,Duf62,Sch95}.
It is an analogue of the extremal length for families of curves on Riemann surfaces, and has many applications in analysis and geometry.

Let $\mathcal{G}$ be a graph, and let $\mathcal{V}$ be the set of vertices.
A \textit{vertex metric} in the graph is a function $\mu: \mathcal{V} \longrightarrow[0, \infty)$.
The \textit{area} of the metric is defined by
$$
\area (\mu) \coloneqq \sum_{v \in \mathcal{V}} \mu(v)^2.
$$
Let $\gamma$ be a \textit{path} in $\mathcal{G}$, i.e., a sequence of vertices $v_0,..., v_{n+1}$ where $v_j$ and $v_{j+1}$ form an edge.
We define its \textit{length} with respect to the vertex metric $\mu$ by
\begin{align*}
    l_\mu(\gamma)\coloneqq \sum_{j=0}^{n+1} \mu(v_j).
\end{align*}
Let $\Gamma$ be a collection of paths in $\mathcal{G}$. A vertex metric $\mu$ is called \textit{$\Gamma$-admissible} if $l_\mu(\gamma) \geq 1$ for all $\gamma \in \Gamma$. The \textit{vertex modulus} or the \textit{vertex extremal width} of $\Gamma$ is defined by
$$
\EW (\Gamma) = \EW_\mathcal{G} (\Gamma) \coloneqq \inf \{ \area(\mu) \colon \mu \text{ is $\Gamma$-admissible}\}.
$$
A metric $\mu$ is called \textit{extremal} if it achieves the infimum in the definition.
The \textit{vertex extremal length} of $\Gamma$ is defined by
$$
\EL(\Gamma) = \EL_\mathcal{G} (\Gamma) \coloneqq \frac{1}{\EW (\Gamma)}.
$$
We will often drop the subscript $\mathcal{G}$ if the underlying graph is not ambiguous. 

More generally, let $\mathcal{W} \subseteq \mathcal{V}$.
We say a vertex metric $\mu$ is $\Gamma$-admissible relative to $\mathcal{W}$ if $\mu$ is $\Gamma$-admissible and $\mu(v) = 0$ for all $v \in \mathcal{W}$.
We define the \textit{relative vertex modulus} or the \textit{relative vertex extremal width} of $\Gamma$ with respect to $\mathcal{W}$ by
$$
\EW (\Gamma, \mathcal{W}) \coloneqq \inf \{ \area(\mu) \colon \mu \text{ is $\Gamma$-admissible relative to $\mathcal{W}$}\}.
$$
Similarly, the \textit{relative vertex extremal length} of $\Gamma$ is defined by
$$
\EL(\Gamma, \mathcal{W})\coloneqq \frac{1}{\EW (\Gamma, \mathcal{W})}.
$$
For the remainder of the paper, we mostly stick to extremal width for consistency, but all results can be stated in terms of extremal length as well.

\subsection{Polygonal subdivision graph}\label{subsec:psg}
Recall that a CW complex $Y$ is a \textit{subdivision} of a CW complex $X$ if the underlying topological spaces (which for simplicity we still denote by $X$ and $Y$) are the same, and every closed cell of $Y$ is contained in a closed cell of $X$.
We define a \textit{polygon} $P$ as a finite CW complex homeomorphic to a closed disk that contains one 2-cell, with at least three 0-cells.
$P$ is called $n$-gon if it has $n$ 0-cells.
We will also call 0-cells, 1-cells and 2-cells the {\em vertices, edges} and {\em faces} respectively.

\begin{defn}\label{defn:posg}
Let $P$ be a polygon. A \textit{polygonal subdivision} of $P$ is a subdivision $\mathcal{R}(P)$ that decomposes the polygon $P$ into $m \geq 2$ closed 2-cells
$$
P = \bigcup_{j=1}^{m} P_{j}, \;\; \text{ so that }
$$
\begin{itemize}
    \item each $P_j$ is a polygon with $n_j$ vertices; and 
    \item each edge of $\partial P$ contains no vertices of $\mathcal{R}(P)$ in its interior.
\end{itemize}
Let $\mathcal{G}$ be the 1-skeleton of $\mathcal{R}(P)$. 
We call the pair $(\mathcal{G}, \partial P)$ the \textit{polygonal subdivision graph}.
A path $\gamma \subseteq \mathcal{G}$ is called \textit{proper} (relative to $\partial P$) if $\partial \gamma \subseteq \partial P$ and $\Int(\gamma) \cap \partial P = \emptyset$.
The \textit{subdivision complexity} is defined by 
$$
\mathscr{C}(\mathcal{G}, \partial P) \coloneqq \max\{n_1,..., n_m\}.
$$
\end{defn}

Let $(\mathcal{G}, \partial P)$ be a simple polygonal subdivision graph.
Let $a, b \in \partial P$ be two non-adjacent vertices.
We denote by $\Gamma_{a,b}$ the family of proper paths in $\mathcal{G}$ that connect $a$ and $b$.
We denote by $\Gamma_{a,b}^*$ the family of proper paths in $\mathcal{G}$ that separate $a$ and $b$. Equivalently, $\Gamma_{a,b}^*$ consists of proper paths that connect the two components of $\partial P - \{a, b\}$.

\begin{theorem}\label{thm:combunifbound}
    Let $(\mathcal{G}, \partial P)$ be a simple polygonal subdivision graph with subdivision complexity $\mathscr{C}(\mathcal{G}, \partial P) = N$.
    Let $a, b$ be a pair of non-adjacent vertices in $\partial P$.  
    Suppose that both $\Gamma_{a,b},\Gamma_{a,b}^*$ are non-empty.
    \begin{enumerate}
        \item(Duality) If $N = 3$, i.e., $\mathcal{R}(P)$ is a triangulation, then 
        $$
        \EW(\Gamma_{a,b}, \partial P)\cdot \EW(\Gamma_{a,b}^*, \partial P) = 1.
        $$
        \item(Quasi-duality) More generally,
        $$
        \frac{1}{(4N+1)^2} \leq \EW(\Gamma_{a,b}, \partial P)\cdot \EW(\Gamma_{a,b}^*, \partial P)\leq 1.
        $$
        Equivalently, we have
        $$
        1 \leq \EL(\Gamma_{a,b}, \partial P)\cdot \EL(\Gamma_{a,b}^*, \partial P)\leq (4N+1)^2.
        $$
    \end{enumerate}
\end{theorem}
We remark that (1) in Theorem \ref{thm:combunifbound} essentially follows from a classical result of Schramm \cite{Sch93} and Cannon-Floyd-Parry \cite{CFP94}.
The upper bound for extremal width in (2) follows from (1). The novel part of the theorem is the lower bound for extremal width in (2).
We remark that if the largest valence of a vertex in $\mathcal{G}$, i.e. the degree of $\mathcal{G}$ is bounded by $K$, then by a theorem of Ha\"issinsky (see \cite[Proposition 2.3]{Hai09}), $\EW(\Gamma_{a,b}, \partial P)\cdot \EW(\Gamma_{a,b}^*, \partial P) \geq M(N, K)$ for some constant $M$ depending on $N$ and $K$.
Thus, the technical part of the theorem is to find a uniform lower bound that is independent of the degree.

\begin{figure}[htp]
        \centering
        \begin{subfigure}{0.4\linewidth}
            \includegraphics[width=\linewidth]{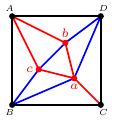}
            \caption{$\mathcal{R}(P)$ is a triangulation.}
            \label{subfig:duality}
        \end{subfigure}
        \begin{subfigure}{0.4\linewidth}
            \includegraphics[width=\linewidth]{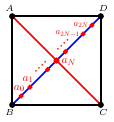}
            \caption{$\mathcal{R}(P)$ is not a triangulation.}
            \label{subfig:quasiduality}
        \end{subfigure}
        \caption{Some examples to illustrate Theorem~\ref{thm:combunifbound}.}
        \label{fig:duality_examples}
\end{figure}

\begin{example}\label{ex:del}
    We include here two examples to illustrate Theorem~\ref{thm:combunifbound}.

    \begin{enumerate}[label=(\Alph*)]
        \item Let $\mathcal{G}$ be the triangulation of a quadrilateral $P$ in Figure~\ref{subfig:duality}. We will calculate $\EW_\mathcal{G}(\Gamma_{A,C},\partial P)$ and $\EW_\mathcal{G}(\Gamma^*_{A,C},\partial P)=\EW_\mathcal{G}(\Gamma_{B,D},\partial P)$. Let $a,b,c$ be the weights assigned to the three interior vertices as labelled in the figure.
        \begin{itemize}
            \item Any admissible metric for $\Gamma_{A,C}$ satisfies $a+b\ge1$ and $a+c\ge 1$. To calculate extremal width we need to minimize $a^2+b^2+c^2$. It is easy to see that the minimum is achieved at $a=2/3$ and $b=c=1/3$. Hence $\EW_\mathcal{G}(\Gamma_{A,C},\partial P)=a^2+b^2+c^2=2/3$.
            \item Any admissible metric for $\Gamma_{B,D}$ satisfies $a\ge1$ and $b+c\ge 1$. The minimum of $a^2+b^2+c^2$ is achieved at $a=1$ and $b=c=1/2$. Hence $\EW_\mathcal{G}(\Gamma_{B,D},\partial P)=3/2$.
        \end{itemize}
        Clearly $\EW_\mathcal{G}(\Gamma_{A,C},\partial P)\cdot \EW_\mathcal{G}(\Gamma^*_{A,C},\partial P)=1$, as predicted in Part (1) of Theorem~\ref{thm:combunifbound}
        \item Let $\mathcal{H}$ be the graph in Figure~\ref{subfig:quasiduality}. Note that the subdivision complexity here is $N+3$. As above, we will calculate $\EW_\mathcal{H}(\Gamma_{A,C},\partial P)$ and $\EW_\mathcal{H}(\Gamma^*_{A,C},\partial P)=\EW_\mathcal{H}(\Gamma_{B,D},\partial P)$. Let $a_0,\ldots,a_{2N}$ be the weights assigned to the interior vertices.
        \begin{itemize}
            \item Any admissible metric for $\Gamma_{A,C}$ satisfies $a_N\ge1$. It is easy to see that the minimum of $\sum_{i=0}^{2N}a_{i}^2$ is achieved at $a_N=1$ and $a_i=0$ when $i\neq N$. Hence $\EW_\mathcal{H}(\Gamma_{A,C},\partial P)=1$.
            \item Any admissible metric for $\Gamma_{B,D}$ satisfies $\sum_{i=0}^{2N}a_i\ge1$. It is easy to see that the minimum of $\sum_{i=0}^{2N}a_{i}^2$ is achieved at $a_i=1/(2N+1)$. Hence $\EW_\mathcal{H}(\Gamma_{B,D},\partial P)=1/(2N+1)$.
        \end{itemize}
         Clearly $1\ge\EW_\mathcal{H}(\Gamma_{A,C},\partial P)\cdot \EW_\mathcal{H}(\Gamma^*_{A,C},\partial P)=1/(2N+1)\ge 1/(4(N+3)+1)^2$, as predicted in Part (2) of Theorem~\ref{thm:combunifbound}.
    \end{enumerate}
\end{example}

\subsection*{Generalizations}
For our application, we only need uniform quasi-duality for simple polygonal subdivision graphs. 
However, the following two simple reductions allow us to apply it to more general graphs.

First, we note that if $(\mathcal{G}, \partial P)$ is not simple, then by collapsing all regions bounded by multi-edges to a single edge and the regions bounded by a self-loop to a single point, we can construct a quotient simple polygonal subdivision graph $(\mathcal{H}, \partial P)$. 
Let $a, b$ be a pair of non-adjacent vertices in $\partial P$.
Then it is easy to see that $\EW_\mathcal{H}(\Gamma_{a,b}, \partial P) = \EW_\mathcal{G}(\Gamma_{a,b}, \partial P)$ and $\EW_\mathcal{H}(\Gamma_{a,b}^*, \partial P) = \EW_\mathcal{G}(\Gamma_{a,b}^*, \partial P)$.
Thus Theorem \ref{thm:combunifbound} applies to non-simple graphs.

Similarly, let $\mathcal{G}$ be a plane graph and $F$ be a Jordan face of $\mathcal{G}$.
Then there exists a maximal subgraph $\mathcal{H}$ of $\mathcal{G}$ containing $\partial F$ so that every face of $\mathcal{H}$ is a Jordan domain.
Then, $(\mathcal{H}, \partial F)$ is a polygonal subdivision graph.
Let $a, b$ be a pair of non-adjacent vertices in $\partial F$. Then it is easy to see that $\EW_\mathcal{H}(\Gamma_{a,b}, \partial F) = \EW_\mathcal{G}(\Gamma_{a,b}, \partial F)$ and $\EW_\mathcal{H}(\Gamma_{a,b}^*, \partial F) = \EW_\mathcal{G}(\Gamma_{a,b}^*, \partial F)$.

Thus, the following corollary follows immediately from Theorem \ref{thm:combunifbound}.
\begin{cor}\label{cor:combunifbound}
    Let $\mathcal{G}$ be a plane graph, $F$ be a Jordan face of $\mathcal{G}$ and $N \geq 3$.
    Suppose that each face of $\mathcal{G}$ other than $F$ has at most $N$ vertices in its ideal boundary.
    Let $a, b$ be a pair of non-adjacent vertices in $\partial F$.  
    Suppose that both $\Gamma_{a,b},\Gamma_{a,b}^*$ are non-empty.
    \begin{enumerate}
        \item(Duality) If $N = 3$, then 
        $$
        \EW(\Gamma_{a,b}, \partial P)\cdot \EW(\Gamma_{a,b}^*, \partial P) = 1.
        $$
        \item(Quasi-duality) More generally,
        $$
        \frac{1}{(4N+1)^2} \leq \EW(\Gamma_{a,b}, \partial P)\cdot \EW(\Gamma_{a,b}^*, \partial P)\leq 1.
        $$
        Equivalently, we have
        $$
        1 \leq \EL(\Gamma_{a,b}, \partial P)\cdot \EL(\Gamma_{a,b}^*, \partial P)\leq (4N+1)^2.
        $$
    \end{enumerate}
\end{cor}

\subsection{Triangulation \texorpdfstring{$\widetilde{\mathcal{G}}$}{G-tilde}}
Let $(\mathcal{G}, \partial P)$ be a simple polygonal subdivision graph associated with $\mathcal{R}(P)$.
We define a new graph $\widetilde{\mathcal{G}}$ from $\mathcal{G}$ as follows.
For each non-triangular face $F$ of $\mathcal{R}(P)$, we add a vertex $w_F$ and connect $w_F$ to every vertex on $\partial F$.
Note that $\widetilde{\mathcal{G}}$ gives a triangulation of $P$, and we have a natural embedding $\mathcal{G}\xhookrightarrow{} \widetilde{\mathcal{G}}$ (see Figure \ref{fig:setup}).
So $(\widetilde{\mathcal{G}}, \partial P)$ is a polygonal subdivision graph.
Let $a, b$ be two non-adjacent vertices in $\partial P$.
We denote by $\widetilde\Gamma_{a,b}$ the family of proper paths in $\widetilde{\mathcal{G}}$ that connect $a$ and $b$.
We denote by $\widetilde\Gamma_{a,b}^*$ the family of proper paths in $\widetilde{\mathcal{G}}$ that separate $a$ and $b$.
\begin{prop}\label{prop:ta}
    Let $(\mathcal{G}, \partial P)$ be a simple polygonal subdivision graph with subdivision complexity $\mathscr{C}(\mathcal{G}, \partial P) = N$.
    Let $(\widetilde{\mathcal{G}}, \partial P)$ be the corresponding triangulation of $(\mathcal{G}, \partial P)$.
    Let $a, b$ be two non-adjacent vertices in $\partial P$.
    Then 
    \begin{align}
    \label{eqn:con}\frac{1}{4N+1}\EW_{\widetilde{\mathcal{G}}}(\widetilde\Gamma_{a,b}, \partial P) &\leq \EW_{\mathcal{G}}(\Gamma_{a,b}, \partial P) \leq \EW_{\widetilde{\mathcal{G}}}(\widetilde\Gamma_{a,b}, \partial P);\\
    \label{eqn:sep}\frac{1}{4N+1}\EW_{\widetilde{\mathcal{G}}}(\widetilde\Gamma_{a,b}^*, \partial P) &\leq \EW_{\mathcal{G}}(\Gamma_{a,b}^*, \partial P) \leq \EW_{\widetilde{\mathcal{G}}}(\widetilde\Gamma_{a,b}^*, \partial P).
    \end{align}
\end{prop}

\begin{proof}[Proof of Theorem \ref{thm:combunifbound} assuming Proposition \ref{prop:ta}]
    (1) Suppose $N=3$. Let $\mathcal{H} \subseteq \mathcal{G}$ be the subgraph consisting of vertices in $\mathcal{G} - \partial P$. 
    Note that $\mathcal{H}$ is not empty, as we assume the polygon is decomposed into at least two 2-cells.
    Let $A , B$ be the subgraph of $\mathcal{H}$ consisting of vertices that are adjacent to $a$ and $b$ respectively.
    Let $\Gamma_{A,B}$ be the set of paths in $\mathcal{H}$ that connects $A$ and $B$.
    Similarly, let $\Gamma_{A,B}^*$ be the set of paths in $\mathcal{H}$ that separates $A$ and $B$.
    Then by definition, we have
    \begin{align*}
        \EW_{\mathcal{G}}(\Gamma_{a,b}, \partial P) &= \EW_{\mathcal{H}}(\Gamma_{A,B});\\
        \EW_{\mathcal{G}}(\Gamma_{a,b}^*, \partial P) &= \EW_{\mathcal{H}}(\Gamma_{A,B}^*).
    \end{align*}
    By \cite[\S 6]{Sch93}, we have that $\EW_{\mathcal{H}}(\Gamma_{A,B}) = \EW_{\mathcal{H}}(\Gamma_{A,B}^*)^{-1}$.
    Therefore, we have $\EW_\mathcal{G}(\Gamma_{a,b}, \partial P)\cdot \EW_\mathcal{G}(\Gamma_{a,b}^*, \partial P) = 1$.

    (2) Since $\widetilde{\mathcal{G}}$ gives a triangulation of $P$, by (1), we have that
    $$
    \EW_{\widetilde{\mathcal{G}}}(\widetilde\Gamma_{a,b}, \partial P) \cdot \EW_{\widetilde{\mathcal{G}}}(\widetilde\Gamma_{a,b}^*, \partial P) = 1.
    $$
    Therefore, by Proposition \ref{prop:ta}, we have that 
    \begin{align*}
        \frac{1}{(4N+1)^2} \leq \EW_\mathcal{G}(\Gamma_{a,b}, \partial P)\cdot \EW_\mathcal{G}(\Gamma_{a,b}^*, \partial P) \leq 1.
    \end{align*}
    The theorem follows.
\end{proof}

\subsection{Proof for the upper bound in Proposition \ref{prop:ta}}\label{subsec:lb}
We start with the easier direction of the Proposition \ref{prop:ta}, whose proof follows from the definition of extremal widths.
\begin{lem}\label{lem:lowerbound}
    Let $(\mathcal{G}, \partial P)$ be a simple polygonal subdivision graph with subdivision complexity $\mathscr{C}(\mathcal{G}, \partial P) = N$.
    Let $\widetilde{\mathcal{G}}$ be the triangulation of $\mathcal{G}$.
    Let $a, b$ be two non-adjacent vertices in $\partial P$.
    Then 
    \begin{align*}
        \EW_{\widetilde{\mathcal{G}}}(\widetilde\Gamma_{a,b}, \partial P) &\geq \EW_{\mathcal{G}}(\Gamma_{a,b}, \partial P),\\
        \EW_{\widetilde{\mathcal{G}}}(\widetilde\Gamma_{a,b}^*, \partial P) &\geq \EW_{\mathcal{G}}(\Gamma_{a,b}^*, \partial P).
    \end{align*}
\end{lem}
\begin{proof}
    Denote the vertex set of $\mathcal{G}$ and $\widetilde{\mathcal{G}}$ by $\mathcal{V}$ and $\widetilde{\mathcal{V}}$ respectively. 
    Let $\widetilde\mu\colon \widetilde{\mathcal{V}} \longrightarrow [0, \infty)$ be an extremal $\widetilde \Gamma_{a,b}$-admissible metric on $\widetilde{\mathcal{G}}$ relative to $\partial P$.
    We remark that the existence of the extremal metric follows from compactness of admissible metrics.
    We define $\mu$ to be the restriction of $\widetilde\mu$ on $\mathcal{V} \subseteq \widetilde{\mathcal{V}}$.
    Let $\gamma$ be a proper path in $\mathcal{G}$ that connects $a, b$.
    Then $\gamma$ is also a proper path in $\widetilde{\mathcal{G}}$ that connects $a, b$, so $l_{\widetilde{\mu}}(\gamma) \geq 1$.
    Thus, $l_{\mu}(\gamma) = l_{\widetilde{\mu}}(\gamma) \geq 1$. Therefore, $\mu$ is $\Gamma_{a,b}$-admissible metric on $\mathcal{G}$ relative to $\partial P$.
    Note that $\area(\mu) \leq \area(\widetilde\mu)$.
    Hence, 
    $$
    \EW_{\mathcal{G}}(\Gamma_{a,b}, \partial P) \leq \area(\mu) \leq \area(\widetilde\mu) = \EW_{\widetilde{\mathcal{G}}}(\widetilde\Gamma_{a,b}, \partial P).
    $$
    The proof for the other inequality is similar.
\end{proof}

\subsection{Proof for the lower bound in Proposition \ref{prop:ta}}
The remainder of this section is dedicated to the proof of the lower bound of Proposition \ref{prop:ta}.
\subsubsection{The setup}
For simplicity of our presentation, we will prove the upper bound for $\Gamma_{a,b}$. The proof of $\Gamma_{a,b}^*$ is similar.

To start our argument, let us label the vertices of $\mathcal{G}$ by 
$$
\mathcal{V} = \mathcal{V}_0 = \{v_1,..., v_r\}.
$$
Denote the space of all non-triangular faces of $\mathcal{R}(P)$ by $\mathcal{F}$.
We start with $\mathcal{G}_0 \coloneqq \mathcal{G}$.
The graph $\mathcal{G}_1$ is constructed from $\mathcal{G}_0$ by adding a vertex $w_{v_1, F}$ in each non-triangular face $F$ adjacent to $v_1$ and connecting $w_{v_1, F}$ to $v_1$ and the two adjacent vertices of $v_1$ on $\partial F$.

Note that $(\mathcal{G}_1, \partial P)$ gives a subdivision $\mathcal{R}_1(P)$ of the polygon $P$. By construction, each non-triangular face $F$ adjacent to $v_1$ is subdivided into the union of two triangles and a polygon which has the same number of sides as $F$.
Thus, there is a natural correspondence of non-triangular faces of the graph $\mathcal{R}_1(P)$ and $\mathcal{R}(P)$.
Moreover, for any vertex $v \neq v_1$ in $\mathcal{V}$, a non-triangular face $F^1$ of $\mathcal{R}_1(P)$ is adjacent to $v$ if and only if the corresponding non-triangular face $F$ of $\mathcal{R}(P)$ is adjacent to $v$ (see Figure~\ref{fig:setup}).

\begin{figure}[htp]
    \centering
    \begin{subfigure}{0.4\linewidth}
        \centering
        \includegraphics[width=\linewidth]{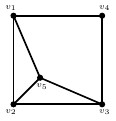}
        \caption{The graph $\mathcal{G}=\mathcal{G}_0$}
        \label{subfig:g_0}
    \end{subfigure}
     \begin{subfigure}{0.4\linewidth}
        \centering
        \includegraphics[width=\linewidth]{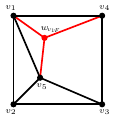}
        \caption{First subdivision $\mathcal{G}_1$}
        \label{subfig:g_1}
    \end{subfigure}
     \begin{subfigure}{0.4\linewidth}
        \centering
        \includegraphics[width=\linewidth]{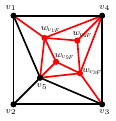}
        \caption{Last subdivision $\mathcal{G}_5$}
        \label{subfig:g_5}
    \end{subfigure}
    \begin{subfigure}{0.4\linewidth}
        \centering
        \includegraphics[width=\linewidth]{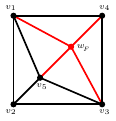}
        \caption{The triangulation $\widetilde{\mathcal{G}}$}
        \label{subfig:g_tilde}
    \end{subfigure}
    \caption{Illustration of the setup.}
    \label{fig:setup}
\end{figure}

Inductively, suppose that $\mathcal{G}_{k-1}$ is constructed. 
The graph $\mathcal{G}_{k}$ is constructed from $\mathcal{G}_{k-1}$ by adding a vertex $w_{v_{k}, F}$ in each non-triangular face $F$ adjacent to $v_k$ and connecting $w_{v_{k}, F}$ to $v_{k}$ and the two adjacent vertices of $v_k$ on $\partial F$.
Note that $(\mathcal{G}_k, \partial P)$ gives a subdivision $\mathcal{R}_k(P)$ of the polygon $P$. By construction, there is a natural correspondence of non-triangular faces of the graph $\mathcal{R}_k(P)$ and $\mathcal{R}(P)$.
Similarly, for $v \neq v_1,..., v_k$ in $\mathcal{V}$, a non-triangular face $F'$ of $\mathcal{R}_k(P)$ is adjacent to $v$ if and only if the corresponding non-triangular face $F$ of $\mathcal{R}(P)$ is adjacent to $v$.

Note that $\mathcal{G}$ embeds as an induced subgraph in $\mathcal{G}_k$.
We define $\Gamma_{a,b,k}$ as the family of proper paths in $\mathcal{G}_k$ relative to $\partial P$ that connect $a, b$.
Denote the vertex set of $\mathcal{G}_k$ by $\mathcal{V}_k$.
Note that
$$
\mathcal{V} = \mathcal{V}_0 \subseteq \mathcal{V}_1 \subseteq... \subseteq \mathcal{V}_r.
$$

Abusing the notations, we use $\mathcal{F}$ to denote the space of non-triangular faces of $\mathcal{R}_k(P)$ for any $k=0,..., r$.
We denote the additional vertices in $\mathcal{G}_k$ by $w_{v, F}$, where $F\in \mathcal{F}$ is adjacent to $v$ and $v \in\{v_1,..., v_r\}$.
Note that there is a natural quotient map $q:\mathcal{G}_r\longrightarrow\widetilde{\mathcal{G}}$, which collapses the vertices $w_{v,F}$ to $w_F$, for $F \in \mathcal{F}$ (see Figure \ref{fig:setup}).
It is easy to see that the projection satisfies the following property.
\begin{lem}\label{lem:connectedpreimage}
    Let $X$ be a connected subgraph of $\widetilde{\mathcal{G}}$. Then $q^{-1}(X)$ is a connected subgraph of $\mathcal{G}_r$.
\end{lem}

\subsubsection{A sequence of admissible metrics \texorpdfstring{$\mu_k$}{mu-k}}
Let $\mu_0 \coloneqq \mu$ be an extremal $\Gamma_{a,b}$-admissible metric on $\mathcal{G}_0 = \mathcal{G}$ relative to $\partial P$.
We will use induction to construct a sequence of $\Gamma_{a,b,k}$-admissible metrics $\mu_k$.
\begin{prop}\label{prop:existencemetric}
    Let $\mu_0 \coloneqq \mu$ be an extremal $\Gamma_{a,b}$-admissible metric on $\mathcal{G}$ relative to $\partial P$.
    There is a sequence of metrics $\mu_k: \mathcal{V}_k \longrightarrow [0, \infty)$ with the following properties.
    \begin{enumerate}[label=(P.\arabic*)]
        \item\label{pro:1} $\mu_k = \mu_{k-1}$ on $\mathcal{V}_{k-1}$;
        \item\label{pro:2} $\mu_k$ is $\Gamma_{a,b,k}$-admissible;
        \item\label{pro:3} $\sum_{v \in \mathcal{V}_k - \mathcal{V}_{k-1}}\mu_k(v) = \sum_{F \in \mathcal{F} \text{ adjacent to }v_k}\mu_k(w_{v_k, F}) \leq 2\mu(v_k)$.
    \end{enumerate}
\end{prop}
To simplify the notations, we will denote $\delta_{v_k,F}\coloneqq \mu_k(w_{v_k, F})$.
\begin{remark}
    We remark that by compactness of admissible metrics, it is easy to see that an extremal $\Gamma_{a,b}$-metric for $\mu_0$ on $\mathcal{G}_0$ exists.
\end{remark}

\subsubsection{Proof of Proposition \ref{prop:ta} assuming Proposition \ref{prop:existencemetric}}
Let $\mu_r$ be the metric on $\mathcal{G}_r$ in Proposition \ref{prop:existencemetric}.
Recall that we have a natural quotient map $q: \mathcal{G}_r\longrightarrow \widetilde{\mathcal{G}}$.
We define the \textit{projection} $\widetilde\mu \coloneqq q_*(\mu_r)$ on $\widetilde{\mathcal{G}}$ by the formula
$$
\widetilde\mu (v) = \sum_{w\in q^{-1}(v)} \mu_r(w).
$$
\begin{lem}\label{lem:adm}
    The metric $\widetilde{\mu}$ is $\widetilde\Gamma_{a,b}$-admissible, and $\widetilde\mu = 0$ on $\partial P$.
\end{lem}
\begin{proof}
    From the construction, $\mu_r = 0$ on $\partial P$. Since $q^{-1}(\partial P) = \partial P$ and $q$ is injective on $\partial P$, $\widetilde\mu = 0$ on $\partial P$.
    
    Let $\gamma \in \widetilde\Gamma_{a,b}$. Then $\gamma$ is a proper path in $\widetilde{\mathcal{G}}$.
    By Lemma \ref{lem:connectedpreimage}, $q^{-1}(\gamma)$ is connected. 
    Thus $q^{-1}(\gamma)$ contains a proper path $\gamma'$ connecting $a, b$.
    By Property \ref{pro:2}, $\mu_r$ is $\Gamma_{a,b,r}$-admissible. Thus $l_{\mu_r}(\gamma') \geq 1$.
    Since $\widetilde\mu = q_*(\mu_r)$, we have
    $$
    l_{\widetilde\mu}(\gamma) = l_{\mu_r}(q^{-1}(\gamma)) \geq l_{\mu_r}(\gamma') \geq 1.
    $$
    Therefore, $\widetilde\mu$ is $\widetilde\Gamma_{a,b}$-admissible.
\end{proof}
\begin{proof}[Proof of Proposition \ref{prop:ta} assuming Proposition \ref{prop:existencemetric}]
    The upper bound follows from Lemma \ref{lem:lowerbound}.

    To prove the lower bound, by Proposition \ref{prop:existencemetric}, let $\widetilde\mu \coloneqq q_*(\mu_r)$ and $\delta_{v_k,F}\coloneqq \mu_k(w_{v_k, F})$.
    Then
    \begin{align}
    \notag\area(\widetilde\mu) &= \sum_{k=1}^r \widetilde\mu^2(v_k) + \sum_{F \in \mathcal{F}}\widetilde\mu^2(w_F)\\
    \label{eqn:1}&= \sum_{k=1}^r \mu^2(v_k) + \sum_{F \in \mathcal{F}}(\sum_{v\in \partial F}\delta_{v,F})^2\\
    \label{eqn:2}&\leq \area(\mu) + \sum_{F \in \mathcal{F}}N\sum_{v\in \partial F} \delta_{v,F}^2\\
    \notag&=\area(\mu) + N\sum_{v \in \mathcal{V}} \sum_{F \in \mathcal{F} \text{ adjacent to }v} \delta_{v,F}^2\\
    \label{eqn:3}&\leq \area(\mu) + N\sum_{v \in \mathcal{V}} \left(\sum_{F \in \mathcal{F} \text{ adjacent to }v} \delta_{v,F}\right)^2\\
    \label{eqn:4}&\leq \area(\mu) + 4N\sum_{v \in \mathcal{V}} \mu(v)^2\\
    \notag&= (4N+1)\area(\mu),
    \end{align}
    where the Equality \eqref{eqn:1} follows from the Property \ref{pro:1} and the definition of $\widetilde\mu$; the Inequality \eqref{eqn:2} follows from the Cauchy-Schwarz inequality and the fact that each face of $\mathcal{R}(P)$ has at most $N$ vertices on its boundary; the Inequality \eqref{eqn:3} follows from the fact that $\delta_{v, F} \geq 0$; and the Inequality \eqref{eqn:4} follows from the Property \ref{pro:3}.

    Since $\mu$ is an extremal $\Gamma_{a,b}$-admissible metric on $\mathcal{G}$ relative to $\partial P$, we have $\EW_\mathcal{G}(\Gamma_{a,b}, \partial P) = \area(\mu)$.
    By Lemma \ref{lem:adm}, $\widetilde\mu$ is a $\widetilde\Gamma_{a,b}$-admissible metric on $\widetilde{\mathcal{G}}$ relative to $\partial P$.
    Therefore, we have
    $$
    \EW_{\widetilde{\mathcal{G}}}(\widetilde\Gamma_{a,b},\partial P) \leq \area(\widetilde{\mu}) \leq (4N+1)\area(\mu) = (4N+1)\cdot \EW_\mathcal{G}(\Gamma_{a,b}, \partial P).
    $$
    This proves the lower bound for $\Gamma_{a,b}$. The proof for $\Gamma_{a,b}^*$ is similar.
\end{proof}

\subsubsection{The construction of \texorpdfstring{$\mu_k$}{mu-k}.}
Suppose that $\mu_{k-1}$ is constructed on $\mathcal{G}_{k-1}$.
We first set up some notations for our construction.
Let $x\coloneqq v_k$.
Let $x_1,..., x_s$ be the list of vertices in $\mathcal{G}_{k-1}$ that are adjacent to $x$. We label them so that they are in counterclockwise orientation around $x$.

Let $F_i$ be the face of $\mathcal{R}_{k-1}(P)$ whose boundary contains $x_i,x,x_{i+1}$, where the subscripts are considered $\mod n$.
If $F_i$ is non-triangular, then $F_i$ contains a vertex, denoted by $y_i$, in $\mathcal{V}_{k} - \mathcal{V}_{k-1}$.
Note that in this case, $\mathcal{G}_k$ contains edges $y_ix$, $y_ix_i$ and $y_ix_{i+1}$.
If $F_i$ is a triangle, then there is an edge in $\mathcal{G}_{k-1}$ that connects $x_i$ and $x_{i+1}$.

We denote by $\mathcal{I}$ the index set for $y_i$, i.e., we have
$$
\mathcal{V}_{k} - \mathcal{V}_{k-1} = \{y_i \colon i \in \mathcal{I}\}.
$$

\begin{figure}[htp]
    \centering
    \includegraphics[width=0.5\linewidth]{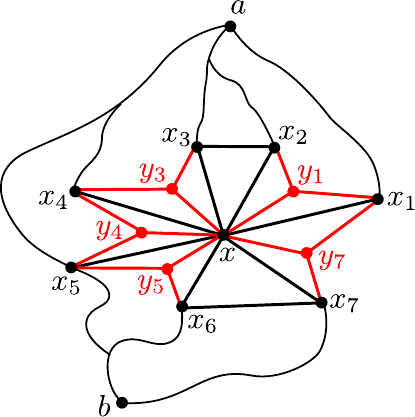}
    \caption{An illustration of the setup of the construction of the metric $\mu_k$.}
    \label{fig:construction_metric}
\end{figure}

If $v, w$ are two vertices in $\mathcal{G}_{k-1}$, we use $[v,w] = [v,w]_{k-1}$ to denote a geodesic (with respect to the metric $\mu_{k-1}$) that connects $v, w$.
Note that the geodesic $[v,w]$ may not be unique.
We will often drop the subscript $k-1$ if the underlying graph $\mathcal{G}_{k-1}$ is not ambiguous.
We denote
\begin{align*}
    m_i&\coloneqq \mu_{k-1}(x_i);\\
    d_i&\coloneqq l_{\mu_{k-1}}([x_i,a]);\\
    \mathring d_i&\coloneqq d_i-m_i;\\
    e_i&\coloneqq l_{\mu_{k-1}}([x_i,b]);\\
    \mathring e_i&\coloneqq e_i-m_i.
\end{align*}
Similarly, we define
\begin{align*}
    m&\coloneqq \mu_{k-1}(x);\\
    d&\coloneqq l_{\mu_{k-1}}([x,a]);\\
    \mathring d&\coloneqq d-m;\\
    e&\coloneqq l_{\mu_{k-1}}([x,b]);\\
    \mathring e&\coloneqq e-m.
\end{align*}
Since $\mu_{k-1}$ is $\Gamma_{a,b,k-1}$-admissible, we have
\begin{align}
    \label{e:1}d_i + \mathring e_i = \mathring d_i + e_i = \mathring d_i + \mathring e_i + m_i &\geq 1, \text{ and}\\
    \label{e:2}d + \mathring e = \mathring d + e = \mathring d + \mathring e + m &\geq 1.
\end{align}
Since $x_i$ is adjacent to $x$, we have
\begin{align}
    \label{e:3}\mathring d_i &\leq \mathring d + m = d,\\
    \label{e:4}\mathring e_i &\leq \mathring e + m = e,\\
    \label{e:5}\mathring d &\leq \mathring d_i + m_i = d_i,\\
    \label{e:6}\mathring e &\leq \mathring e_i + m_i = e_i.
\end{align}
\begin{lem}\label{lem:geonx}
    If $\mathring d_i < d$, then any geodesic connecting $x_i$ to $a$ in $\mathcal{G}_{k-1}$ with respect to $\mu_{k-1}$ metric does not pass through $x$.
\end{lem}
\begin{proof}
    If a geodesic connecting $x_i$ to $a$ passes through $x$, then $\mathring d_i = d$.
\end{proof}

Consider the set of all vertices $x_i, i \in S$ with $\mathring d_i < d$. Then by Lemma \ref{lem:geonx}, any geodesic connecting $x_i$ to $a$ does not pass through $x$.
Consider the set of all geodesics connecting $a$ to $x_i, i \in S$.
Since the graph $\mathcal{G}_{k-1}$ is embedded in the plane, by relabeling the indices if necessary, we denote the right most geodesic not passing through $x$ by $[x_1, a]$ and the left most one by $[x_l, a]$.
Then $x_2,..., x_{l-1}$ are all contained in the region bounded by $[a,x_1]\cup [x_1, x] \cup [x, x_l] \cup [x_l, a]$.
Note that by construction, we have
\begin{align}
    \label{eqn:<d}\mathring d_1, \mathring d_l < d.
\end{align}

Let
$$
j_0\coloneqq \min\{1\leq i \leq l \colon \mathring d_i = \min\{\mathring d_j: j = 1,..., l\}\}.
$$
Then $\mathring d_{j_0}\leq \mathring d_{j}$ for all $j = 1,..., l$.
Inductively, for $n \geq 1$, we define
$$
j_n\coloneqq \min\{j_{n-1}+1\leq i \leq l \colon \mathring d_i = \min\{\mathring d_j: j = j_{n-1}+1,..., l\}\}.
$$
Similarly, we also define for $n \leq -1$,
$$
j_{n}\coloneqq \max\{1\leq i \leq j_{(n+1)}-1 \colon \mathring d_i = \min\{\mathring d_j: j = 1,..., j_{(n+1)}-1\}\}.
$$
Thus, we obtain a finite sequence $j_n, n=p,...,0,..., q$ with $j_p = 1$ and $j_q = l$, so that
\begin{align}
    \label{eqn:ineqd}\mathring d_{j_0} \leq \mathring d_{j_1} \leq ... \leq \mathring d_{j_q}, \text{ and }
    \mathring d_{j_0} \leq \mathring d_{j_{-1}} \leq ... \leq \mathring d_{j_p}.
\end{align}
To make our notations uniform, we also define $j_{q+1} = j_{p-1} = l+1$ if $l < s$.
Note that 
\begin{align}
\mathring d_{j_{q+1}} = \mathring d_{j_{p-1}} = d.
\end{align}

With the notations as above, we are ready to define the metric $\mu_k$ on $\mathcal{G}_k$.
\begin{defn}\label{defn:mu}
    We define the metric $\mu_k$ on $\mathcal{G}_k$ by 
$$
\mu_k(v) = \begin{cases}
    \mu_{k-1}(v), & \text{ if $v \in \mathcal{V}_{k-1}$}\\
    \max\{0, \mathring d_{j_{n+1}} - d_{j_n}\}, & \text{ if $v = y_{j_n}, n = 0,..., q$ and $j_n \in \mathcal{I}$}\\
    \max\{0, \mathring d_{j_{n}} - d_{j_{n+1}}\}, & \text{ if $v = y_{j_{n+1}-1}, n = p-1,..., -1$ and $j_n \in \mathcal{I}$}\\
    0, & \text{ otherwise.}
\end{cases}
$$
\end{defn}
\begin{lem}\label{lem:13}
    The metric $\mu_k$ satisfies Property \ref{pro:1} and \ref{pro:3}.
\end{lem}
\begin{proof}
    By construction, $\mu_k = \mu_{k-1}$ on $\mathcal{V}_{k-1}$, so it satisfies Property \ref{pro:1}.

    To show that it satisfies Property \ref{pro:3}, we first note that for $n\geq 0$, we have that $\mathring d_{j_{n+1}} \geq \mathring d_{j_n}$ by Equation \eqref{eqn:ineqd}. Therefore, we have
    $$
    \max\{0, \mathring d_{j_{n+1}} - d_{j_n}\} \leq \mathring d_{j_{n+1}} - \mathring d_{j_n}.
    $$
    Let $t \geq 0$ be the smallest $n\geq 0$ so that $\mathring d_{j_{n+1}} - d_{j_n} > 0$.
    Then
    \begin{align*}
        \sum_{n\geq 0} \max\{0, \mathring d_{j_{n+1}} - d_{j_n}\} &= \sum_{n \geq t}\max\{0, \mathring d_{j_{n+1}} - d_{j_n}\}\\
        & = \mathring d_{j_{t+1}} - d_{j_t} + \sum_{n > t}\max\{0, \mathring d_{j_{n+1}} - d_{j_n}\}\\
        &\leq \mathring d_{j_{t+1}} - d_{j_t} + \sum_{n > t} \mathring d_{j_{n+1}} - \mathring d_{j_n}\\
        &= \mathring d_{j_{q+1}} - d_{j_t} \leq m,
    \end{align*}
    where the last inequality follows from Equation \eqref{e:3} and \eqref{e:5}.
    Similarly, we have 
    $$
    \sum_{n < 0} \max\{0, \mathring d_{j_{n+1}} - d_{j_n}\} \leq m.
    $$
    Thus, we have
    $$
    \sum_{v \in \mathcal{V}_k - \mathcal{V}_{k-1}} \mu_k(v) \leq \sum_n \max\{0, \mathring d_{j_{n+1}} - d_{j_n}\} \leq 2m.
    $$
    Therefore, $\mu_k$ satisfies Property \ref{pro:3}.
\end{proof}

\subsubsection{Admissibility of \texorpdfstring{$\mu_k$}{mu-k}.}
We use notations with a prime superscript to denote associated quantities with respect to $\mu_k$. 
Note that $m'\coloneqq\mu_k(x) = \mu_{k-1}(x) = m$ and $m_i'\coloneqq\mu_k(x_i) = \mu_{k-1}(x_i) = m_i$.
We also denote
\begin{align*}
    d_i' &\coloneqq l_{\mu_k}([x_i, a]_k);\\
    \mathring d_i' &\coloneqq d_i' - m_i';\\
    e_i' &\coloneqq l_{\mu_k}([x_i, b]_k);\\
    \mathring e_i' &\coloneqq e_i' - m_i'.
\end{align*}
It is easy to see that 
\begin{align*}
    d' &\coloneqq l_{\mu_k}([x, a]_k) = d;\\
    e' &\coloneqq l_{\mu_k}([x, b]_k) = e.
\end{align*}
We define $\mathring d'$ and $\mathring e'$ similarly.

\begin{lem}\label{lem:sufconadm}
    Suppose that $\mathring d_i' + e_i' = d'_i + \mathring e'_i \geq 1$ for all $i=1,..., s$. Then $\mu_k$ is $\Gamma_{a,b,k}$-admissible.
\end{lem}
\begin{proof}
    Suppose not. Then there exists a proper path $\gamma$ connecting $a, b$ in $\mathcal{G}_k$ whose length is strictly less than $1$.
    We may assume that $\gamma$ passes through each vertex at most once.

    If $\gamma$ does not pass through any vertex in $\mathcal{V}_k - \mathcal{V}_{k-1}$, then $\gamma$ is a path in $\mathcal{G}_{k-1}$. Therefore, $l_{\mu_k}(\gamma) = l_{\mu_{k-1}}(\gamma) \geq 1$, which is a contradiction.

    Therefore, $\gamma$ passes through some $y \in \mathcal{V}_k - \mathcal{V}_{k-1}$.
    Since $y$ is connected to $x, x_j, x_{j+1}$ for some $j$, $\gamma$ must pass through either $x_j$ or $x_{j+1}$.
    Therefore, we have 
    $\mathring d_j' + e_j' < 1$ or $\mathring d_{j+1}' + e_{j+1}' < 1$, which is a contradiction.
\end{proof}

We define the following function $f: \{1,..., s\} \longrightarrow \R_{\geq 0}$.
$$
f(i)\coloneqq \begin{cases}
    \mathring d_{j_n}, & \text{if $j_{n} \leq i < j_{n+1}$ and $n = p,..., -1$}\\
    \mathring d_{j_0}, & \text{if $i = j_0$}\\
    \mathring d_{j_{n+1}}, & \text{if $j_n < i \leq j_{n+1}$ and $n = 0,..., q-1$}\\
    d, & \text{if $i > j_q = l$.}
\end{cases}
$$
Note that by Equation \eqref{e:3}, if $l \neq s$, then $\max\{f(i): i =1,..., s\} = d$.
By definition of the $j_n$, we have
\begin{align}
    \label{eqn:digeqfi}\mathring d_i \geq f(i).
\end{align}
See Figure~\ref{fig:graph_d_f} for an example.

\begin{figure}[htp]
    \centering
    \includegraphics[width=0.6\linewidth]{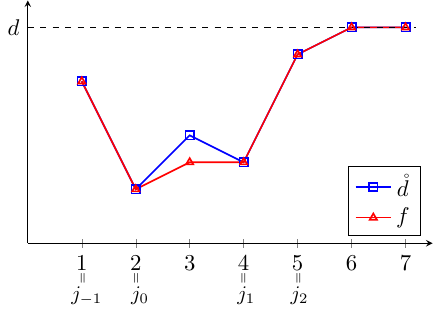}
    \caption{The graphs of $\mathring{d}$ and $f$ for the example in Figure~\ref{fig:construction_metric}; note that $\mathring{d}_i\ge f(i)$.}
    \label{fig:graph_d_f}
\end{figure}

\begin{lem}\label{lem:lengthestimate}
    Let $\gamma$ be a path in $\mathcal{G}_k$ connecting $x_i \neq x_j$ with $f(i) \geq f(j)$. Suppose that $\Int(\gamma)$ contains no vertex in $\{x, x_1,..., x_s\}$.
    Then 
    $$
    l_{\mu_k}(\gamma) \geq \mu_k(x_i) + f(i)-f(j).
    $$
\end{lem}
\begin{proof}
    If $f(i) = f(j)$, then the inequality holds trivially.
    Thus, we assume $f(i) > f(j)$.
    In particular, $f(j) \neq d$, so $j \leq l = j_q$.
    For simplicity of the presentation, we assume that $j_n<j \leq j_{n+1}$ for some $n = 0, ..., q-1$. The other case where $j\leq j_0$ can be proved similarly.

    Suppose that $\gamma$ is contained in $\mathcal{G}_{k-1}$.
    Let $[x_{j_n},a]_{k-1}$ and $[x_{j_{n+1}},a]_{k-1}$ be geodesics in $\mathcal{G}_{k-1}$ connecting $x_{j_n}$ and $x_{j_{n+1}}$ to $a$ respectively.
    Since $f(i) > f(j) = \mathring d_{j_{n+1}} \geq \mathring d_{j_n}$, $[a, x_{j_{n}}]_{k-1} \cup [x_{j_{n}}, x]_{k-1} \cup [x, x_{j_{n+1}}]_{k-1} \cup [x_{j_{n+1}},a]_{k-1}$ separates $x_j$ from $x_i$.
    Since $\Int(\gamma)$ contains no vertex in $\{x, x_1,..., x_s\}$, $\gamma$ cuts either $[x_{j_n},a]_{k-1}$ or $[x_{j_{n+1}},a]_{k-1}$.
    Let $\gamma'$ be the subpath of $\gamma$ that connects $x_i$ to a vertex $v \in [x_{j_n},a]_{k-1} \cup [x_{j_{n+1}},a]_{k-1}$.
    Let $[a,v)_{k-1}$ be the subpath of either $[x_{j_n},a]_{k-1}$ or $[x_{j_{n+1}},a]_{k-1}$ connecting $a$ and $v$ but excluding the vertex $v$.
    Note that $[a,v)_{k-1} \cup \gamma'$ is a proper path in $\mathcal{G}_{k-1}$ connecting $a, x_i$. Thus by Equation \eqref{eqn:digeqfi},
    $$
    l_{\mu_{k-1}}([a,v)_{k-1} \cup \gamma') - \mu_{k-1}(x_i) \geq \mathring d_i \geq f(i).
    $$
    Since
    $l_{\mu_{k-1}}([a,v)_{k-1}) \leq \max\{\mathring d_{j_n}, \mathring d_{j_{n+1}}\} = \mathring d_{j_{n+1}} = f(j)$, we have that
    $$
    l_{\mu_{k-1}}(\gamma) \geq l_{\mu_{k-1}}(\gamma') \geq \mu_{k-1}(x_i) + f(i) - f(j).
    $$

    Suppose that $\gamma$ is not contained in $\mathcal{G}_{k-1}$. Since $f(i) > f(j)$ and $\Int(\gamma)$ contains no vertex in $\{x, x_1,..., x_s\}$, we must have $j = j_{n+1}$ and $i = j+1$ and $\gamma = x_j y_j x_{j+1} = x_j y_j x_{i}$. By construction, we have $\mu_k(y_j) = \max\{0, f(i) - f(j) - \mu_k(x_j)\}$.
    Therefore,
    \begin{align*}
        l_{\mu_k}(\gamma) &= \mu_k(x_j) + \mu_k(y_j) + \mu_k(x_{i}) \\
        &\geq \mu_k(x_j) + f(i) - f(j) - \mu_k(x_j) + \mu_k(x_i) \\
        &= \mu_k(x_i) + f(i)-f(j).
    \end{align*}
    Therefore, the lemma follows.
\end{proof}

\begin{lem}\label{lem:d'f}
    For any $i = 1,..., s$, we have
    $$
    \mathring d_i' \geq f(i).
    $$
\end{lem}
\begin{proof}
    Let $[x_i,a]_k$ be a geodesic in $\mathcal{G}_k$ connecting $x_i$ and $a$.
    If $x\in [x_i,a]_k$, then $\mathring d_i' \geq d \geq f(i)$.
    Therefore, we may assume that $x\notin [x_i,a]_k$.
    Then we can break $[x_i,a]_k = \gamma_1\cup \gamma_2 \cup ... \cup \gamma_t$ into finitely many pieces so that $\Int(\gamma_j)$ contains no vertices in $\{x, x_1,..., x_s\}$ and $\partial \gamma_j \in \{a, x_1,..., x_s\}$.
    
    The proof is by induction on $t$.
    Suppose that $t = 1$. Then $\gamma$ is a path in $\mathcal{G}_{k-1}$. Therefore, $\mathring d'_i = \mathring d_i \geq f(i)$ by Equation \eqref{eqn:digeqfi}.

    Suppose that $\gamma_1$ connects $x_i$ and $x_j$. Then $\gamma_2 \cup ... \cup \gamma_t$ must be a geodesic connecting $x_j$ to $a$.
    By induction hypothesis, we have that $\mathring d_j' \geq f(j)$.
    If $f(i) \leq f(j)$, then 
    $$
    \mathring d_i' \geq \mathring d_j' \geq f(j) \geq f(i).
    $$
    Otherwise, by Lemma \ref{lem:lengthestimate}, we have $l_{\mu_k}(\gamma_1) \geq \mu_k(x_i) + f(i) - f(j)$.
    Therefore, we have
    $$
    \mathring d_i' = l_{\mu_k}(\gamma_1) - \mu_k(x_i) + \mathring d_j' \geq l_{\mu_k}(\gamma_1) - \mu_k(x_i) + f(j) \geq f(i).
    $$
    The lemma now follows.
\end{proof}

\begin{lem}\label{lem:de1}
    For $i =1,..., s$, we have
    $$
    \mathring d'_i + e'_i = d'_i + \mathring e'_i \geq 1.
    $$
\end{lem}
\begin{proof}
    Suppose $i = l+1,..., s$. By Lemma \ref{lem:d'f}, we have that $\mathring d'_i \geq d' = d$.
    Since $e'_i \geq \mathring e' = \mathring e$, and $d + \mathring e \geq 1$, we conclude that $\mathring d'_i + e'_i \geq 1$.

    Otherwise, let $[x_i, b]_k$ be the geodesic in $\mathcal{G}_k$ connecting $x_i$ and $b$.
    If $x\in [x_i,b]_k$, then $\mathring e_i' \geq e$. Since $d_i' \geq \mathring d$, we conclude that $\mathring d'_i + e'_i \geq 1$.
    Therefore, we may assume that $x\notin[x_i,b]_k$.
    Then we can break $[x_i,b]_k = \gamma_1\cup \gamma_2 \cup ... \cup \gamma_t$ into finitely many pieces so that $\Int(\gamma_j)$ contains no vertices in $\{x, x_1,..., x_s\}$ and $\partial \gamma_j \in \{b, x_1,..., x_s\}$. We assume that $b \in \partial \gamma_t$.

    The proof is by induction on $t$.
    Suppose that $t = 1$. Then $\gamma$ is a path in $\mathcal{G}_{k-1}$, and $e_i' = l_{\mu_k}(\gamma) = l_{\mu_{k-1}}(\gamma) = e_i$.
    
    If $i = j_0$, then by Lemma \ref{lem:d'f}, $\mathring d_i' \geq f(i) = \mathring d_i$. In fact, it is easy to see that we must have equality here, but we do not need that. Therefore, $\mathring d_i' + e_i' \geq  \mathring d_i + e_i \geq 1$.
    
    Otherwise, either $j_n < i \leq j_{n+1}$ for some $n = 0,..., q-1$, or $j_n \leq i < j_{n+1}$ for some $n=p,..., -1$.
    Without loss of generality, we assume that we are in the first case.
    Let $[x_{j_n}, a]_{k-1}$ and $[x_{j_{n+1}}, a]_{k-1}$ be the geodesics in $\mathcal{G}_{k-1}$ connecting $x_{j_n}$ and $x_{j_{n+1}}$ to $a$ respectively.
    Note that $\mathring d_{j_n}, \mathring d_{j_{n+1}} \leq \mathring d_l < d$ by Equation \eqref{eqn:<d}. Thus by Lemma \ref{lem:geonx}, $x \notin [x_{j_{n}}, a]_{k-1} \cup [x_{j_{n+1}}, a]_{k-1}$.
    Then $\gamma$ must intersect either $[x_{j_n}, a]_{k-1}$ or $[x_{j_{n+1}}, a]_{k-1}$.
    Since $\mu_{k-1}$ is $\Gamma_{a,b,k-1}$-admissible, either $\mathring d_{j_n} + e_i' \geq 1$ or $\mathring d_{j_{n+1}}+ e_i' \geq 1$.
    Thus, 
    $$
    \max\{\mathring d_{j_n}, \mathring d_{j_{n+1}}\} + e_i' \geq 1.
    $$
    By Lemma \ref{lem:d'f} and the definition of $f(i)$, 
    $$
    \mathring d_i' \geq f(i) = \mathring d_{j_{n+1}} = \max\{\mathring d_{j_n}, \mathring d_{j_{n+1}}\}.
    $$
    Therefore, we have $\mathring d_i' + e_i' \geq 1$.

    For $t > 1$, suppose that $\gamma_1$ connects $x_i$ and $x_j$. Then $\gamma_2 \cup ... \cup \gamma_t$ must be a geodesic connecting $x_j$ to $b$.
    By induction hypothesis, $\mathring d_j' + e_j' \geq 1$.
    Thus,
    \begin{align*}
        \mathring d_i' + e_i' &= \mathring d_i' + l_{\mu_k}(\gamma_1) + \mathring e_j' \\
        &\geq \mathring d_j' + e_j' \geq 1.
    \end{align*}
    Therefore, the lemma follows.
\end{proof}

\begin{proof}[Proof of Proposition \ref{prop:existencemetric}]
    We construct $\mu_k$ as in Definition \ref{defn:mu}.
    By Lemma \ref{lem:13}, the metric $\mu_k$ satisfies the Property \ref{pro:1} and \ref{pro:3}.
    By Lemma \ref{lem:de1} and Lemma \ref{lem:sufconadm}, we conclude that it satisfies the Property $\ref{pro:2}$.
\end{proof}

\section{Discrete vs conformal extremal widths}
In this section, we relate extremal widths for disk patterns with vertex extremal widths for polygonal subdivision graphs.
The main theorem of this section is Theorem \ref{thm:gvsc}.

\subsection{Acylindrical edge-weighted polygonal subdivision graph}\label{subsec:awg}
Let $(\mathcal{G}, \partial P)$ be a simple polygonal subdivision graph associated to $\mathcal{R}(P)$.
    Let $\omega: \mathcal{E} \longrightarrow \{\frac{\pi}{n}\colon n \in \N_{\geq 2}\} \cup \{0\}$.
Motivated by Theorem \ref{thm:acy}, we define the following notion of acylindricity for polygonal subsdivision graphs.
\begin{defn}\label{defn:acy}
    Let $(\mathcal{G}, \partial P)$ be a simple polygonal subdivision graph associated to $\mathcal{R}(P)$.
    Let $\omega: \mathcal{E} \longrightarrow \{\frac{\pi}{n}\colon n \in \N_{\geq 2}\} \cup \{0\}$.
    We say $(\mathcal{G}, \partial P, \omega)$ is \textit{acylindrical} if for any pair of non-adjacent vertices $v, w \in \partial P$,
    \begin{itemize}
        \item $\Gamma_{v,w}^* \neq \emptyset$; and
        \item if there exists a vertex $x \in \mathcal{G} - \partial P$ such that $xv$ and $xw$ are edges of $\mathcal{G}$, then $\omega(xv) +  \omega(xw) < \pi$.
    \end{itemize}
\end{defn}
We remark that it follows from \cite[Proposition 3.7]{LZ23} that the first condition is equivalent to $\Gamma_{v,w}$ being nonempty.
\begin{prop}
    If $(\mathcal{G}, \partial P, \omega)$ is acylindrical, then for any pair of non-adjacent vertices $v, w \in \partial P$, $\Gamma_{v,w} \neq \emptyset$.
\end{prop}

\subsubsection{From reflection groups to subdivision graphs}
Let $G$ be an acylindrical reflection group associated with $(\mathcal{G}, \omega)$.
Let $F$ be a hyperbolic face of $(\mathcal{G}, \omega)$.
Let $P_F\coloneqq S^2 - \Int(F)$ be the complement of $F$. Then $\mathcal{G}$ induces a polygonal subdivision $\mathcal{R}_F$ of $P_F$ as each face of $\mathcal{G}$ is a polygon by Theorem \ref{thm:acy}. Therefore, $(\mathcal{G}, \partial P_F, \omega)$ is a polygonal subdivision graph for $\mathcal{R}_F(P_F)$.
\begin{prop}\label{prop:equivd}
    Let $G$ be an acylindrical reflection group associated with $(\mathcal{G}, \omega)$. Let $F$ be a hyperbolic face of $(\mathcal{G}, \omega)$. Then $(\mathcal{G}, \partial P_F, \omega)$ is acylindrical.
\end{prop}
\begin{proof}
    By Theorem \ref{thm:acy}, $\mathcal{G}$ is $3$-connected. Thus, for any pairs of non-adjacent vertices $v, w$ on $\partial P_F = \partial F$, $\mathcal{G} - \{v,w\}$ is connected. This implies that $\Gamma_{v,w}^* \neq \emptyset$. By Theorem \ref{thm:acy}, there is no right-angled 2-connection. Thus, if $x$ is a vertex in $\mathcal{G} - \partial P_F$ so that $xv$ and $xw$ are edges of $\mathcal{G}$, then $\omega(xv) +  \omega(xw) < \pi$.
    The proposition follows.
\end{proof}

\subsection{Extremal length / width for disk patterns}\label{subsec:eldp}
Let $(\mathcal{G}, \partial P)$ be a simple polygonal subdivision graph associated to $\mathcal{R}(P)$, and let  $\omega: \mathcal{E} \longrightarrow \{\frac{\pi}{n}\colon n \in \N_{\geq 2}\} \cup \{0\}$ be some weight function on the edge set.
\begin{defn}\label{defn:si}
    Let $\mathcal{P}\coloneqq\{D_v, v\in \mathcal{V}\}$ be a disk pattern realizing $(\mathcal{G}, \omega)$. The union $\bigcup_{v \in \partial P} D_v \subseteq \widehat\C$ has two connected complementary components, and exactly one has non-trivial intersection with the disk pattern. 
    
    We denote this complementary component by $\Pi_\mathcal{P}$ (see Figure~\ref{fig:path_family}), and call it the \textit{skinning interstice} of $\mathcal{P}$.

 Let $e \subseteq \partial P$ be an edge connecting $v, w$, and let $x_e = \partial \Pi_{\mathcal{P}} \cap \partial D_v \cap \partial D_w$. Then $(\Pi_{\mathcal{P}}, \{x_e\colon e\subseteq \partial P\})$ is conformally equivalent to a polygon.
We define the curve family
$$
\Gamma_{a,b, \mathcal{P}} \coloneqq\{\alpha\colon \alpha \text{ is a proper path in $\Pi_\mathcal{P}$ and connects $\partial D_a$ to $\partial D_b$} \}.
$$
Similarly, we define 
$$
\Gamma_{a,b, \mathcal{P}}^* \coloneqq\{\alpha\colon \alpha \text{ is proper path in $\Pi_\mathcal{P}$ and separates $\partial D_a$ from $\partial D_b$} \}.
$$
The \textit{(conformal) extremal length} for $\mathcal{P}$ between $a, b$ (and separating $a, b$) are defined by
$$
\ELL(\Gamma_{a,b,\mathcal{P}})\coloneqq \sup_{\rho}\frac{\inf_{\alpha \in \Gamma_{a,b,\mathcal{P}}} l^2_\rho(\alpha)}{\area_\rho(\Pi_{\mathcal{P}})}, \text{ and } 
$$
$$
\ELL(\Gamma_{a,b,\mathcal{P}}^*)\coloneqq \sup_{\rho}\frac{\inf_{\alpha \in \Gamma_{a,b,\mathcal{P}}^*} l^2_\rho(\alpha)}{\area_\rho(\Pi_{\mathcal{P}})},
$$
where the sup is over all conformal metrics on $\Pi_{\mathcal{P}}$.
Similarly, the \textit{(conformal) extremal width} for $\mathcal{P}$ between and separating $a, b$ are defined by
$$
\EWW(\Gamma_{a,b,\mathcal{P}}) = \frac{1}{\ELL(\Gamma_{a,b,\mathcal{P}})}\text{ and } \EWW(\Gamma_{a,b,\mathcal{P}}^*) = \frac{1}{\ELL(\Gamma_{a,b,\mathcal{P}}^*)}.
$$
\end{defn}

\begin{figure}[htp]
    \centering
    \begin{subfigure}[t]{0.49\linewidth}
        \centering
        \includegraphics[width=\linewidth]{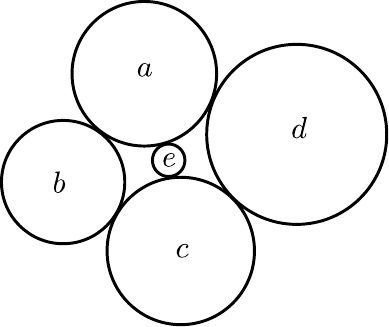}
        \caption{The disk pattern $\mathcal{P}$ realizing $(\mathcal{G}, \partial P)$ where $\partial P$ consists of vertices $a,b,c,d$.}
        \label{subfig:circle_packing}
    \end{subfigure}
    \begin{subfigure}[t]{0.49\linewidth}
        \centering
        \includegraphics[width=\linewidth]{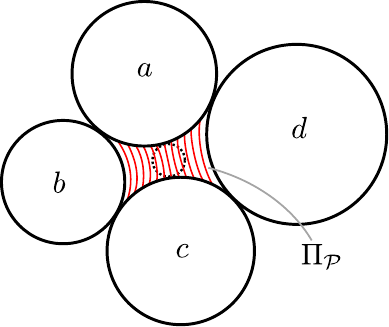}
        \caption{The red curves are part of $\Gamma_{a,c,\mathcal{P}}$.}
        \label{subfig:path_family}
    \end{subfigure}
    \caption{Illustration of the curve family $\Gamma_{a,c,\mathcal{P}}$ in $\Pi_{\mathcal{P}}$}
    \label{fig:path_family}
\end{figure}

The main theorem of the section is the following estimate relating the conformal and discrete extremal widths.
\begin{theorem}\label{thm:gvsc}
   Let $(\mathcal{G}, \partial P)$ be a simple polygonal subdivision graph associated to $\mathcal{R}(P)$, and $\omega: \mathcal{E} \longrightarrow \{\frac{\pi}{n}\colon n \in \N_{\geq 2}\} \cup \{0\}$. Let $N\coloneqq|\partial P|$ where $|\partial P|$ is the number of vertices in $\partial P$. Let $a,b$ be a pair of non-adjacent vertices on $\partial P$.

    Suppose that $(\mathcal{G}, \partial P, \omega)$ is acylindrical. Then there exist universal constants $C$ and $R_0$ so that if $\mathcal{P}\in \Teich(\mathcal{G}, \omega)$ with $\EWW(\Gamma_{a,b,\mathcal{P}})\geq 25\max\{N, R_0\}$, then
    \begin{align*}
        \frac{2}{C}\cdot \EW_\mathcal{G}(\Gamma_{a,b}, \partial P) - 25N \leq \EWW(\Gamma_{a,b,\mathcal{P}}) &\leq \frac{C}{2\cdot\EW_\mathcal{G}(\Gamma_{a,b}^*, \partial P)} + 25N. 
    \end{align*}
    In particular, there exists some constant $L$ so that if $\EW_\mathcal{G}(\Gamma_{a,b}, \partial P)\geq L\cdot N$ and $\frac{1}{\EW_\mathcal{G}(\Gamma_{a,b}^*, \partial P)}\geq L \cdot N$, then 
    \begin{align*}
        \frac{\EW_\mathcal{G}(\Gamma_{a,b}, \partial P)}{C} \leq \EWW(\Gamma_{a,b,\mathcal{P}}) \leq \frac{C}{\EW_\mathcal{G}(\Gamma_{a,b}^*, \partial P)}. 
    \end{align*}
\end{theorem}

\subsection{Circular rectangles}
Let $B(0, R)$ and $B(0, R+1)$ be open disks centered at $0$ with radius $R$ and $R+1$ respectively.
We denote the annulus 
$$
\mathscr{A}_R\coloneqq B(0, R+1) - \overline{B(0, R)}.
$$
By a \textit{circular rectangle} in $\mathscr{A}_R$, we mean the region
$$
\mathscr{R}\coloneqq\{re^{i\theta}: r\in (R, R+1), \theta \in (\theta_1, \theta_2)\}.
$$
We define its \textit{circular width} $\CW(\mathscr{R})$ to be the length of $\partial \mathscr{R} \cap \partial B(0, R)$.
Its \textit{horizontal boundary} is defined by 
$$
\partial_h\mathscr{R} \coloneqq\partial \mathscr{R} \cap (\partial B(0, R)\cup \partial B(0, R+1)),
$$
and \textit{vertical boundary} is defined by
$$
\partial_v\mathscr{R} \coloneqq \overline{\partial \mathscr{R} - \partial_h\mathscr{R}}.
$$
Its extremal width $\EWW(\mathscr{R})$ is the extremal width of families of paths connecting the two horizontal boundary components.

We call the family of radial arcs connecting $\partial_h\mathscr{R}$ the \textit{vertical foliation} of $\mathscr{R}$, and denote it by $\mathcal{F}_{ver, \mathscr{R}}$.
Similarly, we call the family of circular arcs of $\partial B(0, r)$ with $r \in (R, R+1)$ connecting $\partial_v\mathscr{R}$ the \textit{horizontal foliation} of $\mathscr{R}$, and denote it by $\mathcal{F}_{hor, \mathscr{R}}$.

We define the \textit{vertical foliation} of the annulus $\mathscr{A}_R$ as the family of radial arcs connecting the two boundary components of $\mathscr{A}_R$, and denote it by $\mathcal{F}_{v, \mathscr{A}_R}$. The extremal width $\EWW(\mathscr{A}_R)$ of the annulus $\mathscr{A}_R$ is defined as the extremal width of the vertical foliation $\mathcal{F}_{v, \mathscr{A}_R}$.
We define the family of circles $\partial B(0, r)$ with $r \in (R, R+1)$ the \textit{horizontal foliation} of $\mathscr{A}_R$, and denote it by $\mathcal{F}_{h, \mathscr{A}_R}$.
We also define its \textit{circular width} $\CW(\mathscr{A}_R)$ by $2\pi R$.

Since a circular rectangle is the image of a Euclidean rectangle by the exponential map, an easy computation using the logarithm map shows the following.
\begin{lem}\label{lem:cww}
    Let $\mathscr{R}$ be a circular rectangle in $\mathscr{A}_R$. Then
    \begin{align*}
        \EWW(\mathscr{R}) = \frac{1}{R\log(1+\frac{1}{R})}\CW(\mathscr{R}), \text{ and }
        \EWW(\mathscr{A}_R) = \frac{2\pi}{\log(1+\frac{1}{R})}.
    \end{align*}
\end{lem}

Therefore, the circular width is a good approximation of the extremal width of a circular rectangle in the following sense.
\begin{lem}\label{lem:cr}
    Let $\mathscr{R}$ be a circular rectangle in $\mathscr{A}_R$. Suppose that $R \geq 1$. Then
    $$
    \left|\EWW(\mathscr{R}) - \CW(\mathscr{R})\right| \leq \frac{1}{R}\CW(\mathscr{R}) \leq 2\pi.
    $$
    Similarly, we have
    $$
    \left|\EWW(\mathscr{A}_R)-2\pi R\right| \leq 2\pi. 
    $$
\end{lem}
\begin{proof}
    By Lemma \ref{lem:cww}, $\left|\EWW(\mathscr{R}) - \CW(\mathscr{R})\right| \leq \left|\frac{1}{R\log(1+\frac{1}{R})}-1\right| \CW(\mathscr{R})$. Note that if $R \geq 1$, then 
    $$
    1 \geq R\log(1+\frac{1}{R}) \geq R(\frac{1}{R} - \frac{1}{2R^2}) = 1-\frac{1}{2R} \geq \frac{1}{2}.
    $$
    Thus, we have
    \begin{align*}
        \left|\frac{1}{R\log(1+\frac{1}{R})}-1\right| &= \frac{1-R\log(1+\frac{1}{R})}{R\log(1+\frac{1}{R})} \\
        &\leq \frac{1/(2R)}{1/2} = \frac{1}{R}.
    \end{align*}
    Since $\CW(\mathscr{R}) \leq 2\pi R$, the lemma follows.
    The second statement follows from a similar computation.
\end{proof}

\subsection*{Geometric estimates for admissible disks in \texorpdfstring{$\mathscr{A}_R$}{A-R}}
Let $D_i, i=1,2$ denote the two disk components of $\widehat\C - \mathscr{A}_R$.
Let $D$ be a disk in $\widehat \C$. We say $D$ is \textit{admissible} in $\mathscr{A}_R$ if 
\begin{itemize}
    \item $D\cap \mathscr{A}_R \neq \emptyset$; 
    \item $D$ is either disjoint from $D_i$, or it intersects $D_i$ at an angle $\omega_i \in \{\frac{\pi}{n}\colon n \in \N_{\geq 2}\} \cup \{0\}$; and
    \item if $D$ intersects both $D_1, D_2$, then $\omega_1+\omega_2 < \pi$.
\end{itemize}
\begin{lem}\label{lem:ge1}
    There exists a threshold $R_0$ so that for all $R \geq R_0$, an admissible disk $D$ in $\mathscr{A}_R$ has diameter bounded by $5$.
\end{lem}
\begin{proof}
    Consider first the region $\mathscr{S}$ bounded by the two horizontal lines $\Im(z) = 0$ and $\Im(z) = 1$. Let $U_1, U_2\subseteq \C$ be the region defined by $\Im(z) < 0$ and $\Im(z) > 1$.
    Let $D$ be a disk so that $D \cap \mathscr{S} \neq \emptyset$. Suppose that $D$ is either disjoint from $U_i$, or it intersects $U_i$ at an angle $\omega_i \in \{\frac{\pi}{n}\colon n \in \N_{\geq 2}\} \cup \{0\}$.
    We set $\omega_i = 0$ if $D$ is disjoint from $U_i$ and suppose that $\omega_1+\omega_2 < \pi$.
    Then it is easy to see that the radius
    $$
    r(D) \leq \min\{1/\cos(\omega_i)\colon i =1, 2\}.
    $$
    Since $\omega_1+\omega_2 < \pi$, we have $\min\{1/\cos(\omega_i)\colon i =1, 2\} \leq 1/\cos(\pi/3) = 2$ (see Figure~\ref{fig:disk_estimates}).
    Thus, $\diam(D) \leq 4$.

    Note that as $R \to \infty$, the annulus $\mathscr{A}_R$ converges to the strip $\mathscr{S}$ under appropriate normalization Euclidean isometry, the lemma follows.
\end{proof}
\begin{figure}[htp]
    \centering
    \includegraphics[width=0.6\linewidth]{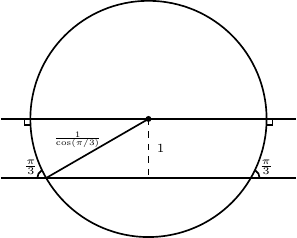}
    \caption{The restriction on the angles gives an upper bound for the radii of admissible disks.}
    \label{fig:disk_estimates}
\end{figure}

\begin{lem}\label{lem:ge2}
    There exists a threshold $R_0 \gg 1$ so that the following holds.
    Let $\mathscr{R}$ be a circular rectangle and $D$ be an admissible disk in $\mathscr{A}_R$ with $R \geq R_0$. Let $l$ be the length of the orthogonal projection of $D\cap \mathscr{R}$ onto $\partial B(0, R)$. Then there exists some universal constant $A$ so that
    $$
    l^2 \leq A\cdot \area(D\cap \mathscr{R}).
    $$
\end{lem}
\begin{proof}
    We adapt the same proof strategy and the notations as in the proof of Lemma \ref{lem:ge1}. Let $\mathscr{R}$ be the rectangle bounded by $\Im(z) = 0$, $\Im(z) = 1$, $\Re(z) = 0$ and $\Re(z) = x$ and $\mathscr{S}$ be the strip bounded by $\Im(z) = 0$ and $\Im(z) = 1$. Let $U_1, U_2\subseteq \C$ be the region defined by $\Im(z) < 0$ and $\Im(z) > 1$.
    
    Let $D$ be a disk so that $D \cap \mathscr{S} \neq \emptyset$. Suppose that $D$ is either disjoint from $U_i$, or it intersects $U_i$ at an angle $\omega_i \in \{\frac{\pi}{n}\colon n \in \N_{\geq 2}\} \cup \{0\}$, with $\omega_1+\omega_2 < \pi$, where we set $\omega_i = 0$ if $D$ is disjoint from $U_i$. Let $l$ be the length of the length of the orthogonal projection of $D \cap \mathscr{R}$ onto the horizontal line $\Im(z) = 0$.

    Suppose $D$ is contained in the strip bounded by $\Re(z) = 0$ and $\Re(z) = x$. Since $D$ intersects $U_i$ at an angle $\leq \pi/2$, the center of $D$ is contained in the strip $\overline{\mathscr{S}}$. Thus, $l = \diam(D)$. The region $D - \mathscr{R} = D - \mathscr{S}$ is a union of two circular segments (potentially empty) of angle $2\omega_1$ and $2\omega_2$. Thus, $D\cap \mathscr{R}$ contains two sectors whose angles add up to $2(\pi - \omega_1 - \omega_2) \geq \frac{\pi}{3}$. Thus,
    $$
    \area(D\cap \mathscr{R}) \geq \frac{1}{2}\cdot 2(\pi - \omega_1 - \omega_2)(\frac{l}{2})^2 \geq \frac{\pi}{24} l^2.
    $$

    Suppose $D$ intersects $\Re(z) = 0$ but not $\Re(z) = x$.
    Let $\mathscr{H}$ be the right half plane $\Re(z) > 0$. Then $D\cap \mathscr{H}$ is circular segment. Let $\theta$ be the angle of the circular segment $D\cap \mathscr{H}$. Then 
    $$
    \area(D\cap \mathscr{H}) = \frac{1}{2}(\theta-\sin\theta)r(D)^2.
    $$
    Since the center of $D$ is contained in the strip $\mathscr{S}$, we conclude that the length of the length of the orthogonal projection of $D \cap \mathscr{H}$ onto the horizontal line $\Im(z) = 0$ equals $l$.
    Thus, $l = r(D) (1-\cos(\theta/2))$. By Taylor expansion at $\theta = 0$, we conclude that there exists some constant $A_1>0$ so that for all $\theta \in [0, 2\pi]$, we have $$\frac{1}{2}(\theta-\sin\theta) \geq A_1(1-\cos(\theta/2))^2.$$
    Thus, $\area(D\cap \mathscr{H}) \geq A_1l^2$.    
    Since $\omega_i \in \{\frac{\pi}{n}\colon n \in \N_{\geq 2}\} \cup \{0\}$ and $\omega_1 + \omega_2 \leq \pi/2+\pi/3$, there exists a constant $A_2>0$ so that
    $$
    \area(D\cap \mathscr{R}) \geq A_2 \area(D\cap \mathscr{H}).
    $$
    Therefore, $\area(D\cap \mathscr{R}) \geq A_1 A_2 l^2$.

    The case $D$ intersecting both $\Re(z) = 0$ and $\Re(z) = x$ can be proved similarly.

    Since the annulus $\mathscr{A}_R$ converges to the strip $\mathscr{S}$ under appropriate normalization Euclidean isometry as $R \to\infty$, the lemma follows.
    \end{proof}

We remark that the last condition in the definition of admissible disks is crucial here as Lemma \ref{lem:ge1} and Lemma \ref{lem:ge2} are both false for a disk $D$ perpendicular to both $D_1$ and $D_2$.

\subsection{Circular rectangle approximating \texorpdfstring{$\Pi_{\mathcal{P}}$}{Pi-P}}
To set up the proof of Theorem \ref{thm:gvsc}, we normalize by some M\"obius map so that the circles $\partial D_a$ and $\partial D_b$ are $\partial B(0, R)$ and $\partial B(0, R+1)$ respectively.
\begin{lem}\label{lem:admD}
    Let $v$ be a vertex in $\mathcal{G} -\{a,b\}$. Then the corresponding disk $D_v$ is admissible in $\mathscr{A}_R$.
\end{lem}
\begin{proof}
    Since $\mathcal{P}$ realises $(\mathcal{G}, \omega)$, it is easy to see that $D_v\cap \mathscr{A}_R \neq \emptyset$ and $D_v$ is either disjoint from $D_a$ (or $D_b$), or it intersects $D_a$ (or $D_b$) at an angle $\omega(av)$ (or $\omega(bv)$) in $\{\frac{\pi}{n}\colon n \in \N_{\geq 2}\} \cup \{0\}$.
    Since $(\mathcal{G}, \omega)$ is acylindrical, if $D_v$ intersects both $D_a$ and $D_b$, then $\omega(av) + \omega(bv) < \pi$.
    Thus, $D_v$ is admissible.
\end{proof}

\begin{lem}\label{lem:crapp}
    Let $R_0$ be the threshold in Lemma \ref{lem:ge1}.
    Suppose that 
    $$
    \EWW(\Gamma_{a,b,\mathcal{P}})\geq 25\max\{N, R_0\}.
    $$
    Then the skinning interstice $\Pi_{\mathcal{P}}$ contains a circular rectangle $\mathscr{R}_-$ and is contained in $\mathscr{R}_+$ which is either a circular rectangle or $\mathscr{A}_R$ with $$\left|\EWW(\Gamma_{a,b,\mathcal{P}}) - \CW(\mathscr{R}_{\pm})\right| \leq 25N.$$
\end{lem}
\begin{proof}
    Denote the circular arc $\partial \Pi_{\mathcal{P}} \cap \partial D_a$ by $\arc{[x_1,x_2]}$. Let $\alpha_i \subseteq \partial \Pi_{\mathcal{P}}$ so that $\Int(\alpha_i)$ is the component of $\partial \Pi_{\mathcal{P}} \cap \mathscr{A}_R$ connecting $x_i$ to $\partial D_b$. Let $p(t)$ be the orthogonal projection of $t\in \alpha_i$ onto $\partial D_a$. 
    
    We claim that the circular arc $\arc{[x_i,p(t)]}$ has length $\leq 5N$ for all $t \in \alpha_i$.
    Indeed, let $v\in \partial P$ be a vertex other than $a, b$. Then the corresponding disk $D_v$ is admissible by Lemma \ref{lem:admD}. By Lemma \ref{lem:cr}, 
    $$
    R \geq \EWW(\mathscr{A}_R)/2\pi -1 \geq \EWW(\Gamma_{a,b,\mathcal{P}})/2\pi-1 \geq R_0.
    $$
    Thus, by Lemma \ref{lem:ge1}, the diameter of $D_v$ is bounded by $5$. Since $\partial P$ has $N$ number of vertices, $x_i, t$ are connected by a chain of at most $N$ disks with diameter $\leq 5$. Therefore, there is a path in $\mathscr{A}_R$ of length $\leq 5N$ connecting $x_i$ to $t$. Since the orthogonal projection is distance non-increasing, we conclude that the circular arc $\arc{[x_i,p(t)]}$ has length $\leq 5N$.

    Let $W$ be the circular length of $\arc{[x_1,x_2]}$. If $W+10N \leq 2\pi R$, then by the previous claim, $\Pi_{\mathcal{P}}$ is contained in a circular rectangle $\mathscr{R}_+$ of circular width $\leq W+10N$. So by Lemma \ref{lem:cr},
    $$
    \EWW(\Gamma_{a,b,\mathcal{P}}) \leq \EWW(\mathscr{R}_+) \leq W+10N+2\pi.
    $$ 
    If $W+10N \geq 2\pi R$, then we define $\mathscr{R}_+\coloneqq \mathscr{A}_R$ and by Lemma \ref{lem:cr},
    $$
    \EWW(\Gamma_{a,b,\mathcal{P}}) \leq \EWW(\mathscr{R}_+) \leq W+10N+2\pi.
    $$ 
    Thus $W + 10N + 2\pi \geq 25N$. Since $N \geq 3$, so $W - 10N - 2\pi > 0$ in either case.
    Thus, $\Pi_{\mathcal{P}}$ contains a circular rectangle $\mathscr{R}_-$ of circular width $\geq W-10N$.
    Therefore, by Lemma \ref{lem:cr}, 
    $$
    W-10N-2\pi \leq \EWW(\mathscr{R}_-) \leq \EWW(\Gamma_{a,b,\mathcal{P}}) \leq \EWW(\mathscr{R}_+) \leq W+10N+2\pi.
    $$
    Moreover, we have
    $$
    W-10N \leq \CW(\mathscr{R}_\pm) \leq W+10N.
    $$
    Therefore,
    $\left|\EWW(\Gamma_{a,b,\mathcal{P}}) - \CW(\mathscr{R}_{\pm})\right| \leq 20N + 2\pi \leq 25 N$.
\end{proof}

\subsection{Overflow of vertical and horizontal foliations}
We say a plane graph $\mathcal{H}$ is a {\em graph extension} of $\mathcal{G}$ if
\begin{itemize}
    \item $\mathcal{V}(\mathcal{H}) = \mathcal{V}(\mathcal{G})$;
    \item $\mathcal{G}$ is a subgraph of $\mathcal{H}$.
\end{itemize}
\begin{defn}
Let $\mathscr{R}_-$ be a circular rectangle contained in the skinning interstice $\Pi_\mathcal{P}$.
We say the vertical foliation $\mathcal{F}_{ver, \mathscr{R}_-}$ of $\mathscr{R}_-$ \textit{combinatorially overflows} $\mathcal{H}$ if for every path $\alpha \in \mathcal{F}_{ver, \mathscr{R}_-}$, there exists a proper path $\gamma \subseteq \mathcal{H}$ connecting $a,b$ so that for any $v \in \Int(\gamma)$, $\alpha \cap D_v \neq \emptyset$.

Similarly, let $\mathscr{R}_+$ be a circular rectangle contains the skinning interstice $\Pi_\mathcal{P}$ or let $\mathscr{R}_+\coloneqq \mathscr{A}_R$. We say the horizontal foliation $\mathcal{F}_{hor, \mathscr{R}_+}$ of $\mathscr{R}_+$ \textit{combinatorially overflows} $\mathcal{H}$ if for every path $\alpha \in \mathcal{F}_{hor, \mathscr{R}_+}$, there exists a proper path $\gamma \subseteq \mathcal{H}$ separating $a,b$ so that for any $v \in \Int(\gamma)$, $\alpha \cap D_v \neq \emptyset$.
\end{defn}

\begin{lem}\label{lem:superg}
    Let $\mathscr{R}_-$ be a circular rectangle contained in $\Pi_\mathcal{P}$. There exists a simple plane graph extension $\mathcal{H}_{ver}$ of $\mathcal{G}$ so that $\mathcal{F}_{ver, \mathscr{R}_-}$ combinatorially overflows $\mathcal{H}_{ver}$.

    Similarly, let $\mathscr{R}_+$ be a circular rectangle contains $\Pi_\mathcal{P}$ or let $\mathscr{R}_+\coloneqq \mathscr{A}_R$. There exists a simple plane graph extension $\mathcal{H}_{hor}$ of $\mathcal{G}$ so that $\mathcal{F}_{hor, \mathscr{R}_+}$ combinatorially overflows $\mathcal{H}_{hor}$.

    By adding more edges if necessary, we may assume that $\mathcal{H}_{ver}$ and $\mathcal{H}_{hor}$ are triangulations of $P$.
\end{lem}
\begin{proof}
    Let $\alpha$ be an arc in the vertical foliation $\mathcal{F}_{ver, \mathscr{R}_-}$. Let $F$ be a hyperbolic face of $\mathcal{R}(P)$. Let $\Pi_{F, \mathcal{P}}$ be the complementary region of $\widehat\C - \bigcup_{v\in \partial F} D_v$ that has empty intersection with the disk pattern $\mathcal{P}$. If $\alpha \cap \Pi_{F, \mathcal{P}}$ connects $\partial D_v$ and $\partial D_w$, then we add the edge $vw$ to the graph if no such edge exists. Since the arcs in $\mathcal{F}_{ver, \mathscr{R}_-}$ do not cross, it is easy to see that the additional edges we add do not cross. Therefore, there are only finitely many edges we can add, and we obtain a graph extension $\mathcal{H}_{ver}$ satisfying the requirement of the lemma.
    By adding more edges if necessary, we may assume that $\mathcal{H}_{ver}$ is a triangulation of $P$.    
    The construction of $\mathcal{H}_{hor}$ is similar.
\end{proof}
\begin{figure}[htp]
    \centering
    \begin{subfigure}{0.35\linewidth}
        \centering
        \includegraphics[width=\linewidth]{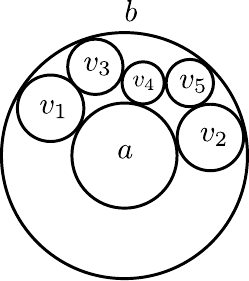}
        \caption{A disk pattern $\mathcal{P}$}
        \label{subfig:disk_pattern_overflow}
    \end{subfigure}
    
    \begin{subfigure}{0.49\linewidth}
        \centering
        \includegraphics[width=\linewidth]{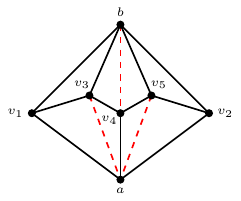}
        \caption{The graph $\mathcal{H}_{ver}$}
        \label{subfig:H_ver}
    \end{subfigure}
    \begin{subfigure}{0.49\linewidth}
        \centering
        \includegraphics[width=\linewidth]{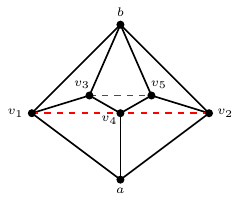}
        \caption{The graph $\mathcal{H}_{hor}$}
        \label{subfig:H_hor}
    \end{subfigure}
    \caption{An example illustrating Lemma~\ref{lem:superg}.}
    \label{fig:disk_pattern_overflow}
\end{figure}

We remark that the graph extensions $\mathcal{H}_{ver}$ and $\mathcal{H}_{hor}$ in Lemma \ref{lem:superg} depend on the disk pattern $\mathcal{P}$.
See Figure~\ref{fig:disk_pattern_overflow} for an example illustrating the Lemma.

\begin{lem}\label{lem:appb}
    Let $R_0$ be the threshold in Lemma \ref{lem:ge2}, and suppose that $R \geq R_0$.
    There exists some universal constant $B$ with the following property.
    Let $\mathscr{R}_-$ be a circular rectangle contained in $\Pi_\mathcal{P}$. Suppose that $\mathcal{F}_{ver, \mathscr{R}_-}$ combinatorially overflow $\mathcal{H}_{ver}$. Then
    $$
    \CW(\mathscr{R}_-) \leq B\cdot\EW_{\mathcal{H}_{ver}}(\Gamma_{a,b,\mathcal{H}_{ver}}, \partial P).
    $$

    Similarly, let $\mathscr{R}_+$ be a circular rectangle contains $\Pi_\mathcal{P}$ or let $\mathscr{R}_+\coloneqq \mathscr{A}_R$. Suppose that $\mathcal{F}_{hor, \mathscr{R}_+}$ combinatorially overflow $\mathcal{H}_{hor}$. Then
    $$
    \frac{1}{\CW(\mathscr{R}_+)} \leq B\cdot\EW_{\mathcal{H}_{hor}}(\Gamma_{a,b,\mathcal{H}_{hor}}^*, \partial P).
    $$
\end{lem}
\begin{proof}
    Let $\mu$ be a $\Gamma_{a,b,\mathcal{H}_{ver}}$-admissible extremal metric on $\mathcal{H}_{ver}$ relative to $\partial P$. 
    Let $\arc{[x_1, x_2]}$ be the circular arc $\partial \mathscr{R}_- \cap \partial D_a$, and let 
    $$
    u:[0, \CW(\mathscr{R}_-)] \longrightarrow \arc{[x_1, x_2]} 
    $$
    be the parameterization by arc-length. Let $t \in [0, \CW(\mathscr{R}_-)]$, and $\alpha_t \in \mathcal{F}_{ver, \mathscr{R}_-}$ be the radial arc connecting $u(t)$ to $\partial D_b$.
    We define the function 
    $$
    L(t)\coloneqq \sum_{v\colon \alpha_t \cap D_v \neq \emptyset} \mu(v).
    $$
    Since $\mathcal{F}_{ver, \mathscr{R}}$ combinatorially overflow $\mathcal{H}_{ver}$, and $\mu$ is $\Gamma_{a,b,\mathcal{H}_{ver}}$-admissible, we have $L(t) \geq 1$ for all $t$.
    Therefore we have
    \begin{align*}
        \CW(\mathscr{R}_-) \leq \int_{0}^{\CW(\mathscr{R}_-)} L(t) \, dt.
    \end{align*}
    Let $v$ be a vertex in $\mathcal{H}_{ver} - \partial P$. We define $l(v)$ as the Lebesgue measure of the interval 
    $$
    \{t\in [0, \CW(\mathscr{R}_-)] \colon \alpha_t \cap D_v \neq \emptyset\}.
    $$
    Then 
    \begin{align} 
    \notag\int_{0}^{\CW(\mathscr{R}_-)} L(t) \, dt &= \sum_{v\in \mathcal{H}_{ver} - \partial P}\mu(v)l(v)\\
    \label{eqn:i1}&\leq \left(\sum_{v\in \mathcal{H}_{ver} - \partial P}\mu(v)^2\right)^{\frac{1}{2}} \left(\sum_{v\in \mathcal{H}_{ver} - \partial P}l(v)^2\right)^{\frac{1}{2}}
    \end{align}
    By Lemma \ref{lem:admD}, $D_v$ is admissible in $\mathscr{A}_R$. Therefore, by Lemma \ref{lem:ge2}, there exists some universal constant $A$ so that
    $$
    l(v)^2 \leq A\cdot\area(D_v \cap \mathscr{R}_-).
    $$
    Since each point in $\mathscr{R}_-$ is covered by at most $3$ different disks, we have that
    \begin{align}
    \label{eqn:i2}\sum_{v\in \mathcal{H}_{ver} - \partial P}l(v)^2 \leq 3A\cdot\area (\mathscr{R}_-) \leq 4A\cdot\CW(\mathscr{R}_-).
    \end{align}
    Note that the last inequality follows from the equality 
    $$
    \area (\mathscr{R}_-) = \frac{R+1/2}{R} \CW(\mathscr{R}_-)
    $$ 
    and the fact that $R \geq R_0 \gg 1$.
    Since $\mu$ is extremal, we have
    \begin{align}
    \label{eqn:i3}\sum_{v\in \mathcal{H}_{ver} - \partial P}\mu(v)^2 \leq \area(\mu) = \EW_{\mathcal{H}_{ver}}(\Gamma_{a,b,\mathcal{H}_{ver}}, \partial P).
    \end{align}
    Combining Equations \eqref{eqn:i1}, \eqref{eqn:i2} and \eqref{eqn:i3}, we have
    $$
    \CW(\mathscr{R}_-)\leq \left(\EW_{\mathcal{H}_{ver}}(\Gamma_{a,b,\mathcal{H}_{ver}}, \partial P)\right)^{\frac{1}{2}}\left(4A\cdot\CW(\mathscr{R}_-)\right)^{\frac{1}{2}}.
    $$
    The first part follows.

    For the second statement, let $\mu^*$ be a $\Gamma_{a,b,\mathcal{H}_{hor}}^*$-admissible extremal metric on $\mathcal{H}_{hor}$ relative to $\partial P$. Let $\beta_t \in \mathcal{F}_{hor}$ with $\beta_t \subseteq \partial B(0, t)$. For $t \in (R, R+1)$, we define
    $$
    L^*(t)\coloneqq \sum_{v\colon \beta_t \cap D_v \neq \emptyset} \mu^*(v).
    $$
    Since $\mathcal{F}_{hor, \mathscr{R}_+}$ combinatorially overflow $\mathcal{H}_{hor}$, $L^*(t) \geq 1$. Therefore, 
    $$
    1 \leq \int_{R}^{R+1} L^*(t)\, dt.
    $$
    Similarly, let $v$ be a vertex in $\mathcal{H}_{hor} - \partial P$. We define $l^*(v)$ as the Lebesgue measure of the interval 
    $$
    \{t\in [R, R+1] \colon \beta_t \cap D_v \neq \emptyset\}.
    $$
    Then by a similar argument, we have
    \begin{align*} 
    1&\leq\int_{R}^{R+1} L^*(t) \, dt\\
    &= \sum_{v\in \mathcal{H}_{hor} - \partial P}\mu^*(v)l^*(v)\\
    &\leq \left(\sum_{v\in \mathcal{H}_{hor} - \partial P}\mu^*(v)^2\right)^{\frac{1}{2}} \left(\sum_{v\in \mathcal{H}_{hor} - \partial P}l^*(v)^2\right)^{\frac{1}{2}}\\
    &\leq \left(\EW_{\mathcal{H}_{hor}}(\Gamma_{a,b,\mathcal{H}_{hor}}^*, \partial P)\right)^{\frac{1}{2}}\left(4A\cdot\CW(\mathscr{R}_+)\right)^{\frac{1}{2}},
    \end{align*}
    where the last inequality follows from a similar bound of $l^*(v)^2$ in terms of $\area(D_v \cap \mathscr{R}_+)$ as in Lemma \ref{lem:ge2}.
    The lemma now follows.
\end{proof}

\subsection{Comparison of extremal widths}
\begin{lem}\label{lem:cew}
    Let $(\mathcal{G}, \partial P)$ be a polygonal subdivision graph for $\mathcal{R}(P)$ and let $\mathcal{H}$ be a graph extension of $\mathcal{G}$. Suppose that $\mathcal{H}$ is a triangulation of $P$. Then
    \begin{align*}
    \EW_{\mathcal{H}}(\Gamma_{a,b,\mathcal{H}}, \partial P) &\leq \frac{1}{\EW_{\mathcal{G}}(\Gamma_{a,b,\mathcal{G}}^*, \partial P)}, \text{ and }\\
    \EW_{\mathcal{H}}(\Gamma_{a,b,\mathcal{H}}^*, \partial P)&\leq \frac{1}{\EW_{\mathcal{G}}(\Gamma_{a,b,\mathcal{G}}, \partial P)}.
    \end{align*}
\end{lem}
\begin{proof}
    Since $\mathcal{H}$ is a triangulation of $P$, by Theorem \ref{thm:combunifbound}, we have
    $$
     \EW_{\mathcal{H}}(\Gamma_{a,b,\mathcal{H}}, \partial P) = \frac{1}{\EW_{\mathcal{H}}(\Gamma_{a,b,\mathcal{H}}^*, \partial P)}.
    $$
    Let 
    $$
    \mu\colon \mathcal{V}(\mathcal{H}) = \mathcal{V}(\mathcal{G}) \longrightarrow [0, \infty)
    $$ 
    be a $\Gamma_{a,b,\mathcal{H}}^*$-admissible extremal metric on $\mathcal{H}$ relative to $\partial P$.
    Since $\mathcal{H}$ is a graph extension of $\mathcal{G}$, $\mu$ is a $\Gamma_{a,b,\mathcal{G}}^*$-admissible metric on $\mathcal{G}$ relative to $\partial P$.
    Therefore,
    $$
    \EW_{\mathcal{G}}(\Gamma_{a,b,\mathcal{G}}^*, \partial P) \leq \EW_{\mathcal{H}}(\Gamma_{a,b,\mathcal{H}}^*, \partial P).
    $$
    The proof for the second inequality is similar.
\end{proof}

\begin{proof}[Proof of Theorem \ref{thm:gvsc}]
     By Lemma \ref{lem:crapp}, $\Pi_{\mathcal{P}}$ contains a circular rectangle $\mathscr{R}_-$ and is contained in $\mathscr{R}_+$ which is either a circular rectangle or $\mathscr{A}_R$ with 
     \begin{align}
     \label{eqn:4n4}\left|\EWW(\Gamma_{a,b,\mathcal{P}}) - \CW(\mathscr{R}_{\pm})\right| \leq 25N.
     \end{align}
     By Lemma \ref{lem:superg}, there exists simple plane graph extensions $\mathcal{H}_{ver}$ and $\mathcal{H}_{hor}$ of $\mathcal{G}$ so that $\mathcal{F}_{ver, \mathscr{R}_-}$ and $\mathcal{F}_{hor, \mathscr{R}_+}$ combinatorially overflows $\mathcal{H}_{ver}$ and $\mathcal{H}_{hor}$ respectively.
     We can assume that $\mathcal{H}_{ver}$ and $\mathcal{H}_{hor}$ are triangulations of $P$. By Lemma \ref{lem:appb}, Lemma \ref{lem:cew} and Equation \eqref{eqn:4n4}, there exists a universal constant $C$ so that
     \begin{align*}
     \EWW(\Gamma_{a,b,\mathcal{P}})&\leq \CW(\mathscr{R}_{-}) + 25N \\
     &\leq \frac{C}{2}\cdot\EW_{\mathcal{H}_{ver}}(\Gamma_{a,b,\mathcal{H}_{ver}}, \partial P) + 25N,\\
     &\leq  \frac{C}{2\cdot\EW_{\mathcal{G}}(\Gamma_{a,b,\mathcal{G}}^*, \partial P)} + 25N, \text{ and }\\
     \EWW(\Gamma_{a,b,\mathcal{P}})&\geq \CW(\mathscr{R}_{+}) - 25N \\
     &\geq \frac{2}{C\cdot\EW_{\mathcal{H}_{hor}}(\Gamma_{a,b,\mathcal{H}_{hor}}^*, \partial P)} - 25N\\
     &\geq \frac{2}{C}\cdot\EW_{\mathcal{G}}(\Gamma_{a,b,\mathcal{G}}, \partial P)-25N.
     \end{align*}
     The theorem follows.
\end{proof}

\section{The uniform diameter bound for skinning maps}\label{sec:uub}
In this section, we will first prove the following uniform upper bound for disk patterns, which implies our main theorem. Recall that the subdivision complexity $\mathscr{C}(\mathcal{G}, \partial P)$ and the skinning interstice $\Pi_\mathcal{P}$ of $\mathcal{P}$ are defined in Definition \ref{defn:posg} and Definition \ref{defn:si}.
\begin{theorem}\label{thm:uniformbounddiskpattern}
    Let $(\mathcal{G}, \partial P)$ be a simple polygonal subdivision graph associated to $\mathcal{R}(P)$ and $\omega: \mathcal{E} \longrightarrow \{\frac{\pi}{n}\colon n \in \N_{\geq 2}\} \cup \{0\}$.
    Let
    $N\coloneqq \max \{\mathscr{C}(\mathcal{G}, \partial P), |\partial P|\}$,
    where $|\partial P|$ is the number of vertices on $\partial P$.

    Suppose that $(\mathcal{G}, \partial P, \omega)$ is acylindrical.
    Then there exists some constant $K = K(N)$ so that for any two disk patterns $\mathcal{P}, \mathcal{P}' \in \Teich(\mathcal{G}, \omega)$, the Teichm\"uller distance $d(\Pi_\mathcal{P}, \Pi_{\mathcal{P}'}) \leq K$.  
    \end{theorem}
\begin{proof}[Proof of Theorem \ref{thm:main} assuming Theorem \ref{thm:uniformbounddiskpattern}]
        Let $F$ be a hyperbolic face of $\mathcal{G}$. Let $P_F\coloneqq S^2 - \Int(F)$ be the complement of $F$. Then $(\mathcal{G}, \partial P_F)$ is a polygonal subdivision graph for $P_F$. Note that by definition,
        $$
        \mathscr{C}_{top}(G) = \max\{\mathscr{C}(\mathcal{G}, \partial P_F), |\partial P_F| = |\partial F|\}.
        $$
        By Proposition \ref{prop:equivd}, $(\mathcal{G}, \partial P_F, \omega)$ is acylindrical as a polygonal subdivision graph for $\mathcal{R}_F(P_F)$. 
        Let $\Pi_{F, \mathcal{P}}^-$ (or $\Pi_{F, \mathcal{P}'}^-$) be the skinning interstices of $\mathcal{P}$ (or $\mathcal{P}'$ respectively).
        Thus by Theorem \ref{thm:uniformbounddiskpattern}, there exists a constant $K = K(\mathscr{C}_{top}(G))$ so that for any $\mathcal{P}, \mathcal{P}' \in \Teich(\mathcal{G}, \omega)$, we have
        $$
        d(\Pi_{F, \mathcal{P}}^-, \Pi_{F, \mathcal{P}'}^-) \leq K
        $$
        Since this is true for all hyperbolic faces of $\mathcal{G}$, the theorem follows.
    \end{proof}

\subsection{Lamination on \texorpdfstring{$P$}{P}}
We now set up for the proof of Theorem \ref{thm:uniformbounddiskpattern}.
Let $(\mathcal{G}, \partial P)$ be a simple polygonal subdivision graph associated to $\mathcal{R}(P)$ with a weight function $\omega$ so that $(\mathcal{G}, \partial P, \omega)$ is acylindrical.
Recall that $|\partial P|$ is the number of vertices on $\partial P$, i.e., $P$ is an $|\partial P|$-gon.
\begin{defn}
    Let $a,b$ and $a',b'$ be two pairs of non-adjacent vertices on $\partial P$. We say that they are \textit{unlinked} if there exist $\gamma$ and 
    $\gamma'$ in the polygon $P$ connecting $a,b$ and $a',b'$ respectively so that $\Int(\gamma) \cap \Int(\gamma') = \emptyset$.
    They are called \textit{linked} otherwise.

    A collection of pairwise unlinked pairs of non-adjacent vertices $\mathcal{L}\coloneqq\{\{a_i, b_i\}, i = 1,..., k\}$ is called a \textit{lamination} on $P$.
    We say a pair of non-adjacent vertices $\{a, b\}$ is unlinked with $\mathcal{L}$ if it is unlinked with any pairs $\{a', b'\} \in \mathcal{L}$.
\end{defn}

Let $\mathcal{P} \in \Teich(\mathcal{G}, \omega)$.
Recall that $\Gamma_{a,b,\mathcal{P}}$ is the family of paths in the skinning interstice $\Pi_{\mathcal{P}}$ connecting $\partial D_a$ and $\partial D_b$ (see \S \ref{subsec:eldp}).
Since wide family of paths do not cross each other, we have the following.
\begin{lem}\label{lem:unlink}
    Let $a,b$ and $a',b'$ be two pairs of non-adjacent vertices on $\partial P$. Suppose that both 
    $$
    \EWW(\Gamma_{a,b,\mathcal{P}})> 2  \text{ and }\EWW(\Gamma_{a',b',\mathcal{P}}) > 2.
    $$
    Then $a,b$ and $a',b'$ are unlinked.
\end{lem}

\subsection*{Thick-thin decomposition}
As mentioned in the introduction, some of the previous works assume the conformal boundary of the manifold lie in the thick part of the Teichm\"uller space. Indeed, the arguments rely on certain uniform hyperbolicity that holds in the thick part, but fails in the thin part.

In our setting, we need to handle degeneration into the thin part of $\Teich(\mathcal{G}, \omega)$. For this, we have the following lemma, which is
a variation of a result of Minsky (see \cite[Theorem 6.1]{Min96}) in our setting.
\begin{lem}\label{lem:min}
    Let $\mathcal{P}, \mathcal{P}' \in \Teich(\mathcal{G}, \omega)$. Let $K_1 > K_2 > 2$ be some constants that are larger than some universal threshold and $M>0$. 
    Suppose that there exists a lamination $\mathcal{L}$ such that for any pair of non-adjacent vertices $a, b$, we have that 
    \begin{enumerate}
    \item if $\{a, b\} \notin \mathcal{L}$, then 
    $\EWW(\Gamma_{a,b,\mathcal{P}}), \EWW(\Gamma_{a,b,\mathcal{P}}) \leq K_1$.
    \item if $\{a, b\}\in \mathcal{L}$, then 
    \begin{enumerate}
        \item $\EWW(\Gamma_{a,b,\mathcal{P}}), \EWW(\Gamma_{a,b,\mathcal{P}'}) \geq K_2$; and
        \item $\frac{1}{M} \leq \EWW(\Gamma_{a,b,\mathcal{P}})/\EWW(\Gamma_{a,b,\mathcal{P}'})\leq M$.
    \end{enumerate}
    \end{enumerate}
    Then there exists a constant $H$ depending on $|\partial P|, K_1, K_2$ and $M$ so that
    $$
    d(\Pi_{\mathcal{P}}, \Pi_{\mathcal{P}'}) \leq H.
    $$
\end{lem}
\begin{proof}
    Note that marked polygons $\Pi_{\mathcal{P}}, \Pi_{\mathcal{P}'}$ are conformally equivalent to 
    \begin{align*}
        (\D, \,\,\{t_1, t_2, ..., t_{|\partial P|}\})\text{ and }(\D, \,\,\{s_1, s_2, ..., s_{|\partial P|}\}),
    \end{align*}
    where $t_i, s_i \in S^1 = \partial \D$.
    By doubling the surface, we obtain two marked punctured spheres $X\coloneqq(\widehat\C,\,\,\{t_1, t_2, ..., t_{|\partial P|}\})$ and $X'\coloneqq(\widehat\C,\,\,\{s_1, s_2, ..., s_{|\partial P|}\})$.
    Denote by $S$ the topological punctured sphere that gives the marking.

    Let $r$ be the reflection along $S^1$. Then non-trivial isotopy classes of simple closed curves on $X$ invariant under $r$ are in one-to-one correspondence with pairs of non-adjacent vertices of $\partial P$.
    Let $\gamma \subseteq S$ be the multi-curve associated to the lamination $\mathcal{L}$.
    For sufficiently large $K_2$, the lamination $\mathcal{L}$ gives a Thick-Thin decomposition of the surfaces $X$ and $X'$ along $\gamma$ as in \cite[\S 2.4]{Min96}. 
    There is a natural homeomorphism induced by the Fenchel-Nielsen coordinates $\Pi: \Teich(S) \longrightarrow \Teich(S-\gamma) \times \Hyp_1 \times ... \times \Hyp_k$, where $k$ is the number of components of $\gamma$ (see \cite[\S 6]{Min96}). Since two surfaces $X, X'$ are symmetric with respect the unit circle $S^1$, the twist parameters can be chosen to be zero.
    By Condition (2), $X, X'$ lies in the thin part of $(S, \gamma)$. By Condition (1), the projections of $\pi_1\circ \Pi(P), \pi_1\circ \Pi(P')$ lie in the compact set of $\Teich(S-\gamma)$, where $\pi_1$ is the projection map onto the first coordinate.
    This follows from \cite[Lemma 4.2 and Lemma 4.3]{Min96} (see also \cite[Lemma 8.4]{Min92}).
    Thus, there exists a constant $L$ depending on $|\partial P|, K_1$ so that $d(\pi_1\circ \Pi(P), \pi_1\circ \Pi(P')) < L$.
    Here $d$ is the Teichm\"uller metric on $\Teich(S-\gamma)$.
    Therefore, the lemma follows from \cite[Theorem 6.1]{Min96}.
\end{proof}

\subsection{The uniform bound}
Let $L$ and $R_0$ be the constant in Theorem \ref{thm:gvsc}. 
We assume $R_0$ is sufficiently large so that Lemma \ref{lem:min} applies if $K_2 \geq R_0$.
\begin{lem}\label{lem:conC}
    There exists a universal constant $\lambda$ so that the following holds.
    Let $a, b \in \partial P$ be a pair of non-adjacent vertices.  
    If there exists $\mathcal{P} \in \Teich(\mathcal{G}, \omega)$ so that
    $$
    \EWW(\Gamma_{a,b,\mathcal{P}}) \geq \lambda N^3,
    $$
    then $\EW_\mathcal{G}(\Gamma_{a,b}, \partial P)\geq L\cdot N$ and $\frac{1}{\EW_\mathcal{G}(\Gamma_{a,b}^*, \partial P)}\geq L\cdot N$.
    
    Moreover, for any $\mathcal{P}' \in \Teich(\mathcal{G}, \omega)$, we have 
    $$
    \EWW(\Gamma_{a,b,\mathcal{P}'}) \geq 25\max\{N, R_0\}.
    $$
    \end{lem}
    \begin{proof}
        Suppose that $\EWW(\Gamma_{a,b,\mathcal{P}}) \geq \lambda N^3$ for some $\lambda$ to be determined. By Theorem \ref{thm:gvsc}, $\frac{1}{\EW_\mathcal{G}(\Gamma_{a,b}^*, \partial P)} \geq \lambda_1N^3$ for some constant $\lambda_1$ depending only on $\lambda$ and $\lambda_1 \to \infty$ as $\lambda \to \infty$. By Theorem \ref{thm:combunifbound}, we have
        $$
        \EW_\mathcal{G}(\Gamma_{a,b}, \partial P)\geq \frac{1}{\EW_\mathcal{G}(\Gamma_{a,b}^*, \partial P)\cdot(4N+1)^2} \geq \lambda_2 N
        $$
        for some constant $\lambda_2$ depending only on $\lambda_1$ and $\lambda_2 \to \infty$ as $\lambda_1 \to \infty$. 
        We choose $\lambda \gg 1$ large enough so that $\lambda_1, \lambda_2 \geq L$.

        For the moreover part, suppose not. Then by continuity of extremal widths, we can find $\mathcal{P}' \in \Teich(\mathcal{G}, \omega)$ with $\EWW(\Gamma_{a,b,\mathcal{P}'}) = 25\max\{N, R_0\}$. By Theorem \ref{thm:gvsc}, we have that 
        $$
        25\max\{N, R_0\} = \EWW(\Gamma_{a,b,\mathcal{P}'}) \geq \frac{1}{C}\cdot \EW_\mathcal{G}(\Gamma_{a,b}, \partial P) \geq \lambda_2 N/C.
        $$
        By increase $\lambda$ if necessary, we may assume that $\lambda_2 N/C > 25\max\{N, R_0\}$, which gives a contradiction. The lemma follows.
    \end{proof}

    \begin{lem}\label{lem:ub}
    There exist a constant $M$ and a lamination $\mathcal{L}\coloneqq\{\{a_i, b_i\}, i = 1,..., k\}$ on $P$ so that for any pair $\mathcal{P}, \mathcal{P}' \in \Teich(\mathcal{G}, \omega)$,
    \begin{enumerate}
    \item if $\{a, b\}\notin\mathcal{L}$, then $\EWW(\Gamma_{a,b,\mathcal{P}}), \EWW(\Gamma_{a,b,\mathcal{P}'}) \leq MN^3$.
    \item if $\{a, b\}\in \mathcal{L}$, then 
    \begin{enumerate}
        \item $\EWW(\Gamma_{a,b,\mathcal{P}}), \EWW(\Gamma_{a,b,\mathcal{P}'}) \geq 25\max\{N, R_0\}$; and 
        \item $\frac{1}{M^2N^8} \leq \EWW(\Gamma_{a,b,\mathcal{P}})/\EWW(\Gamma_{a,b,\mathcal{P}'})\leq M^2N^8$.
    \end{enumerate} 
    \end{enumerate}

    \end{lem}
    \begin{proof}
        Let $M \geq \max\{\lambda, C\}$ where $\lambda$ is the constant in Lemma \ref{lem:conC} and $C$ is the constant in Theorem \ref{thm:gvsc}.
        We define a lamination $\mathcal{L}$ as the collection of pairs of non-adjacent vertices $a_i, b_i$ of $\partial P$ with
        $$
        \EWW(\Gamma_{a_i,b_i,\mathcal{P}_i}) \geq M N^3 \text{ for some $\mathcal{P}_i \in \Teich(\mathcal{G}, \omega)$.}
        $$
        By Lemma \ref{lem:conC}, we have that $\EW_\mathcal{G}(\Gamma_{a_i,b_i}, \partial P)\geq LN$ and $\frac{1}{\EW_\mathcal{G}(\Gamma_{a_i,b_i}^*, \partial P)}\geq LN$, and $\EWW(\Gamma_{a_i,b_i,\mathcal{P}}) \geq 25\max\{N, R_0\}$ for all $\mathcal{P} \in \Teich(\mathcal{G}, \omega)$. Therefore, by Theorem \ref{thm:gvsc}, we have that for all $\mathcal{P} \in \Teich(\mathcal{G}, \omega)$,
        $$
        \frac{\EW_\mathcal{G}(\Gamma_{a_i,b_i}, \partial P)}{C} \leq \EWW(\Gamma_{a_i,b_i,\mathcal{P}}) \leq \frac{C}{\EW_\mathcal{G}(\Gamma_{a_i,b_i}^*, \partial P)}.
        $$
        By Theorem \ref{thm:combunifbound}, $\frac{1}{C}\cdot \EW_\mathcal{G}(\Gamma_{a_i,b_i}, \partial P) \geq \frac{1}{C\cdot (4N+1)^2\cdot\EW_\mathcal{G}(\Gamma_{a_i,b_i}^*, \partial P)}$. Thus for all $\mathcal{P} \in \Teich(\mathcal{G}, \omega)$,
        $$
         \EWW(\Gamma_{a_i,b_i,\mathcal{P}}) \in [\frac{1}{C(4N+1)^2}, C]\cdot\frac{1}{\EW_\mathcal{G}(\Gamma_{a_i,b_i}^*, \partial P)}.
        $$
        Thus, we have
        $$
        \frac{1}{M^2N^8} \leq \frac{1}{C^2(4N+1)^2} \leq \frac{\EWW(\Gamma_{a_i,b_i,\mathcal{P}})}{\EWW(\Gamma_{a_i,b_i,\mathcal{P}'})}\leq C^2(4N+1)^2 \leq M^2N^8.
        $$

        Since $\EWW(\Gamma_{a_i,b_i,\mathcal{P}}) \geq 25\max\{N, R_0\} > 2$ for all $\mathcal{P} \in \Teich(\mathcal{G}, \omega)$, by Lemma \ref{lem:unlink}, $a_i, b_i$ and $a_j, b_j$ are unlinked if $i \neq j$.
        Thus, $\mathcal{L}$ is indeed a lamination.
        
        Let $a, b$ be a pair of non-adjacent vertices of $\partial P$. Suppose that $\{a,b\} \notin \mathcal{L}$.
                Then $\EWW(\Gamma_{a,b,\mathcal{P}}) \leq MN^3$ for all $\mathcal{P} \in \Teich(\mathcal{G}, \omega)$.
    \end{proof}

    \begin{proof}[Proof of Theorem \ref{thm:uniformbounddiskpattern}]
        This theorem follows from Lemma \ref{lem:ub} and Lemma \ref{lem:min}.
    \end{proof}

\end{document}